\documentclass{article}

\usepackage{microtype}
\usepackage{graphicx}
\usepackage{subfigure}
\usepackage{booktabs} % for professional tables

\usepackage{hyperref}

\usepackage{tikz}
\usepackage{pgfplots}
\usepackage[algo2e,linesnumbered,vlined,ruled]{algorithm2e}
\usepackage{graphicx}
\usepackage{multirow}
\usepackage{multicol}
\usepackage{graphicx}
\usepackage{color}

\usepackage{amsmath,amsxtra,amsfonts,amscd,amssymb,bm,amsthm}

\newtheorem{theorem}{Theorem}

\newtheorem{assumption}[theorem]{Assumption}
\newtheorem{lemma}[theorem]{Lemma}

\newcommand{\be}{\begin{equation}}
\newcommand{\ee}{\end{equation}}
\newcommand{\bee}{\begin{equation*}}
\newcommand{\eee}{\end{equation*}}

\newcommand{\beaa}{\begin{eqnarray*}}
\newcommand{\eeaa}{\end{eqnarray*}}

\newcommand{\Prob}{\mathbb{P}}
\newcommand{\Rn}{\mathbb{R}}

\newcommand{\ovec}{\operatorname{vec}}

\newcommand{\fracp}[2]{\frac{\partial{#1}}{\partial{#2}}}

\newcommand{\Expe}{\mathbb{E}}

\usepackage[accepted]{icml2021}

\begin{document}

\twocolumn[
\icmltitle{Sketchy Empirical Natural Gradient Methods for Deep Learning}



\icmlsetsymbol{equal}{*}

\begin{icmlauthorlist}
\icmlauthor{Minghan Yang}{af1,af2}
\icmlauthor{Dong Xu}{af1,af2}
\icmlauthor{Zaiwen Wen}{af2,af3,af4}
\icmlauthor{Mengyun Chen}{af5}
\icmlauthor{Pengxiang Xu}{af6}
%\icmlauthor{Tateu H.~Yasehe}{ed,to,goo}
%\icmlauthor{Aaoeu Iasoh}{goo}
%\icmlauthor{Buiui Eueu}{ed}
%\icmlauthor{Aeuia Zzzz}{ed}
%\icmlauthor{Bieea C.~Yyyy}{to,goo}
%\icmlauthor{Teoau Xxxx}{ed}
%\icmlauthor{Eee Pppp}{ed}
\end{icmlauthorlist}

\icmlaffiliation{af1}{School of Mathematical Sciences, Peking University, China}
\icmlaffiliation{af2}{Beijing International Center for Mathematical Research, Peking University, China}
\icmlaffiliation{af3}{Center for Data Science, Peking University, China}
\icmlaffiliation{af4}{National Engineering Laboratory for Big Data Analysis and Applications, Peking University, China}
\icmlaffiliation{af5}{Huawei Technologies Co. Ltd, China}
\icmlaffiliation{af6}{Peng Cheng Laboratory, Shenzhen, China}

\icmlcorrespondingauthor{Zaiwen Wen}{wenzw@pku.edu.cn}
%\icmlcorrespondingauthor{Minghan Yang}{yangminghan@pku.edu.cn}

\icmlkeywords{Machine Learning, ICML}

\vskip 0.3in
]



\printAffiliationsAndNotice{}  % leave blank if no need to mention equal contribution

\begin{abstract}
In this paper, we develop an efficient sketchy empirical natural gradient method (SENG) for large-scale deep learning problems. The empirical Fisher information matrix is usually low-rank since the sampling is only practical on a small amount of data at each iteration. Although the corresponding natural gradient direction lies in a small subspace, both the computational cost and memory requirement are still not tractable due to the high dimensionality. We design randomized techniques for different neural network structures to resolve these challenges. For layers with a reasonable dimension, sketching can be performed on a regularized least squares subproblem. Otherwise, since the gradient is a vectorization of the product between two matrices, we apply sketching on the low-rank approximations of these matrices to compute the most expensive parts. A distributed version of SENG is also developed for extremely large-scale applications. Global convergence to stationary points is established under some mild assumptions and a fast linear convergence is analyzed under the neural tangent kernel (NTK) case. Extensive experiments on convolutional neural networks show the competitiveness of SENG compared with the state-of-the-art methods. On the task ResNet50 with ImageNet-1k, SENG achieves 75.9\% Top-1 testing accuracy within 41 epochs. Experiments on the distributed large-batch training show that the scaling efficiency is quite reasonable. \end{abstract}


%\begin{figure}
%\centering
%\subfigure[Total Time]{
%    \begin{tikzpicture}[scale=0.2]
%    \pgfplotstableread[row sep=\\]{
%thread	DCRBR  idea\\
%4	1  1\\
%8	1.8311  2 \\
%16	3.4467   4\\
%32	6.9559  8\\
%}\mydata
%
%\begin{axis}[
%    grid,
%    width=\textwidth,
%    height=0.66\textwidth,
%    legend pos=north west,
%    xtick=data,
%    ylabel={Acceleration rate},
%    xlabel={Number of GPUs}
%    ]
%    \addplot[thick,mark=square,draw=red!90!black] table[x=thread,y=DCRBR]{\mydata};
%
%    \addplot[color=black,dashed] table[x=thread,y=idea]{\mydata};
%    \legend{SENG}
%\end{axis}
%    \end{tikzpicture}
%    }
%    \subfigure[Per Epoch Time]{
%      \begin{tikzpicture}[scale=0.2]
%    \pgfplotstableread[row sep=\\]{
%thread	DCRBR  idea\\
%4	1  1\\
%8	1.8311  2 \\
%16	3.3626   4\\
%32	7.4644  8\\
%}\mydata
%
%\begin{axis}[
%    grid,
%    width=\textwidth,
%    height=0.66\textwidth,
%    legend pos=north west,
%    xtick=data,
%    ylabel={Acceleration rate},
%    xlabel={Number of GPUs}
%    ]
%    \addplot[thick,mark=square,draw=red!90!black] table[x=thread,y=DCRBR]{\mydata};
%
%    \addplot[color=black,dashed] table[x=thread,y=idea]{\mydata};
%    \legend{SENG}
%\end{axis}i
%    \end{tikzpicture}
%    }
%    \vspace{-2ex}
%    \caption{Acceleration Rate in Terms of Different Indexes.}
%    \vspace{-2ex}
%\end{figure}
\section{Introduction}
%1. Deep Learning problems arise in practice. Emphasis the importance.
Deep learning
%\cite{lecun2015deep, Goodfellow-et-al-2016, Sch15}
makes a breakthrough and holds great promise in many applications, e.g., machine translation, self-driving and etc. Developing practical deep learning optimization methods is an urgent need from end users.

The goal of deep learning is to find a fair good decision variable $\theta\in\Rn^n$ so that the output of the network $f(x,\theta) \in \Rn^m$ matches the true target $y$. Specifically, for a given dataset $\{x_i,y_i\}_{i=1}^N$, we consider the following empirical risk minimization problem:
\be
\label{finite-sum}
\min_{\theta \in \Rn^n}\Psi(\theta) = \frac{1}{N} \sum_{i=1}^N \psi_i(\theta)  = \frac{1}{N} \sum_{i=1}^N  \psi(y_i, x_i,\theta),
%=  \frac{1}{N} \sum_{i=1}^N  -\log\left(p_i(\theta)\right) ,
 \ee
where $\psi_i(\theta)= \psi(y_i, x_i,\theta)$ is the loss function and it is common to use the negative log probability loss $\psi(y_i,x_i,\theta)=-\log \left (p \left (y_i| f\left(x_i,\theta \right) \right ) \right )$, e.g., the mean squared error (MSE) or the CrossEntropy loss.
% $p_i(\theta)=p(y_i|x_i,\theta)$ and $p$ is a distribution learned by the current network. $\psi(x,y,\theta)= -\log (p(y|x,\theta))$:
% for short.

The basic and most popular optimization methods in deep learning are first-order type methods, such as SGD \cite{RobMon51}, Adam \cite{kingma2014adam}, and etc. They are easy to implement but suffer a slow convergence rate and generalization gap in distributed large-batch training \cite{KeskarMNST17,shallue2019measuring}. Second-order type methods enjoy better convergence properties and exhibit a good potential in distributed large-batch training \cite{9123671}, but suffer a high computational cost at each iteration. They leverage the curvature information in different ways. The natural gradient method \cite{amari1997neural} corrects the gradient according to the local KL-divergence surface.
%However, these algorithms have to compute/approximate the inverse of Hessian matrix of the neural network at each iteration. This is seldom acceptable due to high computational cost.
An online approximation to the natural gradient direction is used in the TONGA method \cite{roux2008topmoumoute}. The online Newton step algorithm \cite{hazan2007logarithmic} uses the empirical fisher information matrix (EFIM) and the authors analyze the convergence properties in the online learning setting. The Fisher Information Matrix (FIM) is integrated naturally with a practical Levenberg-Marquardt framework \cite{ren2019efficient} and the direction can be economically computed by using the Sherman-Morrison-Woodbury (SMW) formula.
%KFAC \cite{martens2015optimizing} and Shampoo \cite{gupta2018shampoo} have been proven successful in large-scale applications in \cite{35epochsKFAC} and \cite{anil2020second}.
The KFAC method \cite{martens2015optimizing} based on independence assumptions approximates the FIM by decomposing the large matrix into a Kronecker product between two smaller matrices each layer. A recursive block-diagonal approximation to the Gauss-Newton matrix is studied in \cite{practicalGN} and each block is Kronecker factored.
% and can be computed by a single backward pass.
% By using fast Hessian matrix-vector products, the Hessian-free method \cite{HF-DL} uses the conjugate-gradient method to compute the corresponding direction. The CURVEBALL method  \cite{CURVEBALL} takes the curvature information with two additional forward-mode automatic differentiation operations over the network.

Theoretical understanding of the second-order type methods for deep learning problems focuses on the natural gradient descent (NGD) methods. The authors in \cite{Bernacchia2018ExactNG} consider the deep linear networks and show the fast convergence of NGD. The properties of NGD for both shallow and deep nonlinear networks in the NTK regime are shown in \cite{zhang2019fast,Cai2019AGM,Karakida2020UnderstandingAF}.

In this paper, we develop a novel \textbf{Sketchy Empirical Natural Gradient (SENG)} method.
The EFIM is usually low-rank and thus the direction lies in a small subspace. However, the cost is not tractable due to the high dimensionality. Our SENG method utilizes randomized techniques to reduce the computational complexity and memory requirement. By using the SMW formula, it is easy to know that the search direction is actually a linear combination of the subsampled gradients where the coefficients are determined by a regularized least squares (LS) subproblem. For layers with a reasonable dimension, we construct a much smaller subproblem by sketching on the subsampled gradients. Otherwise, since the gradient is a vectorization of the product between two matrices, we first take low-rank approximations to these matrices and then use randomized algorithms to approximate the expensive operations. We further extend SENG to the distributed setting and propose suitable strategies to reduce both the communication and computational cost for extremely large-scale applications. Global convergence is established under some mild assumptions and  the linear convergence rate is proved for the NTK case. Numerical comparisons with the state-of-the-art methods demonstrate the competitiveness of our method on a few typical neural network architectures and datasets. On the task of training ResNet50 on ImageNet-1k, we show great improvement over the well-tuned SGD (with momentum) method. Experiments on large-batch training are investigated to show the good scaling efficiency and the great potential in practice.

 %Due to the curse of
%    dimensionality, calculating the corresponding direction is not tractable.
%The corresponding direction is located in a low-rank subspace and can be obtained by solving a high dimensional regularized least square problem.
%To make the method practical, we propose several randomized techniques for different neural network structures. %\begin{itemize}
%\item
%We propose a novel sketchy empirical natural gradient framework for large-scale deep learning problems and a global convergence is established under some standard assumptions.

\section{The Empirical Fisher Information Matrix}
%\subsection{Fisher Information Matrix}
%By using the chain-rule, the Hessian matrix of the objective function in (\ref{finite-sum}) can be derived as follows:
%\be\label{Hessian}
%\nabla^2 \Psi(\theta)=\frac{1}{N} \sum_{i=1}^N J_f^i  \nabla_f^2 \psi_i(\theta)(J_f^i)^\top+\frac{1}{N} \sum_{i=1}^N\sum_{k=1}^m \nabla_{f_i^k}
%\psi_i(\theta)\nabla_\theta^2f_i^k(\theta),
%\ee
%where $J_f^i=\nabla_\theta f(x_i,\theta)\in \mbR^{n\times m}$,
%$f_i^k(\theta)$ is the $k$-th component of $f_i(\theta)$ and $f_i(\theta):= f(x_i,\theta)$ .
%The first part in (\ref{Hessian}) is also called the generalized Gauss-Newton (GGN) matrix,
%which is positive semi-definite (PSD) if the loss function is convex.

The FIM of the loss in \eqref{finite-sum} is based on the distribution learned by the neural network \cite{martens2015optimizing, martens2020new} and is defined by:
%the Hessian matrix can be written as:
%\be
%\label{Hessian-2}
%\nabla^2 \Psi(\theta)=
%% \frac{1}{N}\sum_{i=1}^N \frac{\partial \psi_i(\theta)}{\partial \theta}\left( \frac{\partial \psi_i(\theta)}{\partial \theta}\right)^\top   -  \frac{1}{N}\sum_{i=1}^N \frac{1}{p(y_i\mid x_i,\theta)} \frac{\partial^2 p(y_i\mid x_i,\theta)}{\partial^2 \theta}.
%   \frac{1}{N}\sum_{i=1}^N  \nabla \psi_i(\theta)\nabla \psi_i(\theta)^\top -   \frac{1}{N}\sum_{i=1}^N \frac{1}{p_i(\theta)} \nabla^2 p_i(\theta).
%\ee
$
\frac{1}{N} \sum_{i=1}^N \Expe_{y\sim p(y|x_i,\theta)} \nabla_{\theta} \psi(y,x_i,\theta)\nabla_{\theta} \psi(y,x_i,\theta)^\top,
$
where the distribution $p(y|x,\theta)$ coincides with that used in the loss function. The subsampled FIM is also considered in the KFAC method \cite{martens2015optimizing} and is defined as
\[
\frac{1}{N} \sum_{i=1}^N \frac{1}{m_i}\sum_{j=1}^{m_i} \nabla_{\theta} \psi(y_i^j,x_i,\theta)\nabla_{\theta} \psi(y_i^j,x_i,\theta)^\top,
\]
where $y_i^j \sim p(y |x_i,\theta)$ and $m_i$ is the number of the samples for $x_i$. Computing the subsampled FIM needs more backward passes. Instead, we consider the EFIM and use its low-rank subsampled variant as our curvature matrix.  Given a mini-batch $S \subset \{1,2,\dots,N\}$ with a sample size $\varrho= |S|$, the subsampled EFIM can be represented as follows:
\be
\label{mb-EFIM}
M_S (\theta)=
 \frac{1}{\varrho}\sum_{i \in S}
%  \frac{\partial \psi_i(\theta)}{\partial \theta}\left( \frac{\partial \psi_i(\theta)}{\partial \theta}\right)^\top  ,
 \nabla \psi_i(\theta)\nabla \psi_i(\theta)^\top.
\ee
 The subsampled EFIM (\ref{mb-EFIM}) is a summation of a few rank-one matrices and is low-rank if $n\gg \varrho$. In practice, when the over-parameterized neural networks ($n \gg N$) are used, the deterministic EFIM ($\rho = N$) is still low-rank. Another important motivation is that the EFIM is a part of the Hessian matrix in certain cases e.g., for the negative log probability loss, 
$
 \nabla^2 \Psi(\theta) = \frac{1}{N}\sum_{i=1}^N\left ( \nabla \psi_i(\theta)\nabla \psi_i(\theta)^\top -   \frac{1}{p \left (y_i| f\left(x_i,\theta \right) \right)} \frac{\partial^2 p \left (y_i| f\left(x_i,\theta \right) \right )}{\partial \theta^2}\right ).$
% where $(f_i)_j$ is the $j$-th element of $f(x_i,\theta).$
% Note that the FIM is not low-rank except in some special case.
% Our method also applies to FIM when it is easy to get.
Considering a neural network with $L$ layers, the gradient with respect to (w.r.t.) the layer $l$ for a single sample $\{x_i, y_i\}$ can be obtained by the back-propagation process and written as a vectorization of matrix-matrix multiplication \cite{sun2019optimization}
 \be \label{grad-back}u_i^l(\theta) = \ovec( \hat G^l_i (\theta)(\hat A^l_i(\theta))^\top),\ee
  where $\hat{G}_i^l(\theta) \in \mathbb{R}^{n^l_G \times \kappa^l}$, $\hat{A}_i^l(\theta) \in \mathbb{R}^{n_A^l \times \kappa^l}$, $n^l = n^l_G\cdot n^l_A$ and $n^l$ is the number of parameters in the $l$-th layer. Note that $\hat G^l_i (\theta)$ and $\hat A^l_i(\theta)$ are computed by the backward and forward process, respectively. Hence, the per-sample gradient is a concatenation of $L$ sub-vectors:
\begin{align*} &\nabla \psi_i (\theta) : =u_i(\theta)\\ &= [(u_i^1(\theta))^\top,\dots, (u_i^l(\theta))^\top, \dots, (u_i^L(\theta))^\top]^\top \in \Rn^n .\end{align*}

% We denote the collection of gradients with respect to the sample set $S$ as $U_S (\theta) =\frac{1}{\sqrt{\varrho}}[u_1 (\theta), u_2 (\theta), \ldots, u_{\varrho} (\theta)] \in \Rn^{n\times \varrho}$ and  the collection of gradients with respect to $l$-th layer as $U^l_S (\theta) =\frac{1}{\sqrt{\varrho}} [u_1^l (\theta) ,u_2^l (\theta), \dots, u_{\varrho}^l (\theta)]\in \Rn^{n^l\times \varrho}$.
%
 Hence, the subsampled EFIM matrix $M_S (\theta)$ and its block diagonal part can be written as:
\be
\label{mb-EFIM2}
\begin{aligned}
M_{S}  (\theta)&= \frac{1}{\varrho} \sum_{i \in S} u_i (\theta) u_i (\theta)^\top = U_S (\theta)U_S (\theta)^\top, \\
%g_S (\theta)& =  \frac{1}{\varrho}\sum_{i\in S} u_i  (\theta).
\end{aligned}
\ee
and
\be
\label{mb-EFIM2}
\begin{aligned}
M^l_{S}  (\theta)&= \frac{1}{\varrho} \sum_{i \in S} u^l_i (\theta) u^l_i (\theta)^\top = U^l_S (\theta)U^l_S (\theta)^\top,
%&= \frac{1}{\varrho} \sum_{i \in S} \ovec( \hat G^l_i (\theta)(\hat A^l_i(\theta))^\top)
\end{aligned}
\ee
where $U_S (\theta) =\frac{1}{\sqrt{\varrho}}[u_1 (\theta), u_2 (\theta), \ldots, u_{\varrho} (\theta)] \in \Rn^{n\times \varrho}$
%To obtain the $M_S$, we can reuse per-sample gradients that is used to compute mini-batch gradients.
and $U^l_S (\theta) =\frac{1}{\sqrt{\varrho}} [u_1^l (\theta) ,u_2^l (\theta), \dots, u_{\varrho}^l (\theta)]\in \Rn^{n^l\times \varrho}$.

Note that the subscript $S$ and $\theta$ will be dropped if no confusion can arise. For example, we denote $ M_{S^k}^l(\theta_k)$ by $M_k^l$ at the point $\theta_k$. Throughout this paper, the layer number is expressed by the superscripts.
%We consider the basic subsampling scheme and other schemes can be also considered, such as \cite{johnson2013accelerating}.
%It is natural that neural network is composed of several layers. This leads to an idea that diagonal approximation to $M_S$ with each block to per layer:
%\be
%\label{mb-EFIM-block}
%M_{S}\approx   \tilde{M}_{S} :=\diag\left\{ M^1_{S},\dots,M^L_{S}\right \}= \diag\left\{  U^1_S(U^1_S)^\top,\dots,U^L_S(U^L_S)^\top \right\}.
%\ee

\section{The SENG Methods}
We first describe a second-order framework for the problem (\ref{finite-sum}). At the $k$-th iteration, a regularized quadratic minimization problem at the point $\theta_k$ is constructed as follows:
\be
\label{quadratic-model}
\min_d F_k(d)=\Psi_k + g_k^\top d + \frac{1}{2}d^\top (B_k+\lambda_k I) d,\ee
where $\Psi_k = \Psi(\theta_k)$, $g_k=g_{S_k}(\theta_k) = \frac{1}{|S_k|}\sum_{i\in S_k} u_i(\theta_k) $ is the mini-batch gradient, $B_k$ is an approximation to the Hessian matrix of $\Psi$ at $\theta_k$ and $\lambda_k$ is a regularization parameter to make $B_k + \lambda_k I$ positive definite. Note that the sample sets in the $B_k$ and $g_k$ can be different.
To reduce the computational cost, $B_k$ is designed to be block diagonal according to the network structure:\[B_k = \text{block-diag}\{ M^1_k,\dots M^L_k \}.\] Hence, $B_k$ is positive semi-definite. By solving the subproblem (\ref{quadratic-model}), we obtain $d_k :=[(d_k^1)^\top,\dots,(d_k^L)^\top]^\top$, where
\be \label{eq:dk} d_k^l =- (M_k^l+\lambda_k I)^{-1}g_k^l .\ee
Then we set $\theta_{k+1} = \theta_k + \alpha_k d_k$, where $\alpha_k$ is the step size.
Since the formulations of the directions $d_k^l$ for all layers are identical, we next only focus on a single layer by dropping the explicit layer indices and the iteration number $k$ if no confusion can arise. For example, $n^l$ and $U_k^l$ are written as $n$ and $U$ for simplicity in certain cases. %In that case, we can view the network as $1$-layer network.
\subsection{Direction in a Low-rank Subspace}
The main concern in \eqref{eq:dk} is the expensive computation of the inverse of $M$. However,
considering the low-rank structure of $M$ in \eqref{mb-EFIM2} and by the SMW formula, the direction actually is:
\be \begin{aligned}
\label{inversion} d
%&  - (\lambda I + U U^\top)^{-1}g\\
% =- \frac{1}{\lambda} g + \frac{1}{\lambda}  U\left((\lambda I + U^\top U)^{-1} (U^\top g )\right )
 = -  a g  +   aUb,
\end{aligned}
\ee
where $a=\frac{1}{\lambda}$ is a scalar and
\be \label{sol-ls-1}b=(\lambda I + U^\top U)^{-1} (U^\top g)\in \Rn^{\varrho}.\ee
Thus, the direction $d$ is located in a low-rank subspace spanned by $g$ and the column space of $U$. %If $g_{\hat{S}}$ is estimated by the same samples $S$ as $U_S$, e.g., $\hat{S} = S$, the direction will be reduced to
%\[
%%\begin{aligned}
%d = - \frac{a}{\varrho}U\textbf{1} + a Ub = - a U(\frac{1}{\varrho}\textbf{1} - b ),
%%\end{aligned}
%\]
%where $\textbf{1}\in \Rn^{\varrho}$ is the vector that all entries are 1.
%This means a second-order type direction is a proper combination of the gradient estimation.
%An interesting problem is how to choose good coefficients with a low cost.

This computation involves three basic operations: $U^\top z$, $U^\top U$ and $Uc$ for certain vectors $z$ and $c$. The main cost is the computation of the coefficients $b$, whose complexity is $O(\varrho^3 +\varrho^2 n)$. Since the batch size $\varrho$ is not large in many cases, the bottleneck in \eqref{sol-ls-1} is computing the product $U^\top U$ rather than computing the inverse of a $\varrho\times \varrho$ matrix.

The number of parameters $n$ for each layer is usually very large, see Table \ref{resnet-stat}. Therefore, both the computational cost and memory requirement of $U$ can not be ignored due to the high dimensionality. When $n$ is large and $n > (n_G + n_A ) \kappa$, e.g., for the cases IV, V, VI in Table \ref{resnet-stat}, it is better to store $\hat{G_i}$ and $\hat{A}_i$ other than the gradient $u_i$. Otherwise, e.g., for the cases I, II, III in Table \ref{resnet-stat}, we store $U$ explicitly and the computational cost can be reduced by sketching. We next present explicit and implicit methods for both cases by designing different sketching mechanisms in Sec \ref{SketchLS} and Sec \ref{ImplicitConstruction}, respectively.
\begin{table}%[t]%[!b]
\footnotesize
\centering
\caption{Statistics on a few typical layers of ResNet50. Conv means the convolutional layer while Fc means the fully-connected layer.}
%{\color{red} To be completed.}.}
{
\vspace{0.5ex}
\begin{tabular}{cccccc}
\hline \hline
Case & $n^l$ & $n_G^l$ &  $n_A^l$ & $\kappa^l$ & Type \\ \hline
I &9, 408 & 64 & 147 & 12, 544 &Conv \\
II &147, 456 & 128 & 1, 152 & 784 &Conv \\
III&524, 288 & 1, 024 & 512 & 196 &Conv \\
IV&1, 048, 576 & 2, 048 & 512 & 49 &Conv \\
V&2, 359, 296 & 512 & 4, 608 & 49 &Conv \\
VI&2, 049, 000 & 1, 000 & 2, 049 & 1 &Fc \\
\hline \hline
\end{tabular}
}
\label{resnet-stat}
\vspace{-2ex}
\end{table}

\subsection{Sketching on a Regularized LS Subproblem}
\label{SketchLS}
In this part, we use sketching  on a regularized least squares subproblem to reduce the computational cost of $b$. This is based on the key observation that the vector $b$ in (\ref{inversion}) is the solution of the following regularized LS problem:
\be\label{least-square-1}
\min_{b \in \Rn^{\varrho}} \|Ub - g\|^2 + \lambda\|b\|_2^2.
\ee
%From (\ref{least-square-1}), the direction $d$ in (\ref{inversion}) can be explained as the ``residual'' of $g$ in the space $U$ in some sense.
We use the sketching method to reduce the scale of the subproblem by denoting $\Xi = \Omega U$, $\xi = \Omega g$, where $\Omega \in \Rn^{q\times n}$ is a sketching matrix ($q\ll n$). Then, the subproblem is modified as:
\be\label{least-square-2}
\min_{\hat b\in \Rn^{\varrho}} \|\Omega U\hat b - \Omega g\|^2 + \lambda\|\hat b\|_2^2 = \|\Xi \hat b-\xi\|^2+\lambda\|\hat b\|^2.
\ee
The solution of the problem (\ref{least-square-2}) is
\be
\label{sol-ls-2}
\hat b = \left(\lambda I + \Xi^\top\Xi\right)^{-1}\Xi^\top\xi.\ee Hence, the direction is changed to:
\be
\begin{aligned}
\label{sketch-direc}
%\hat{d} &= - \frac{1}{\lambda} g + \frac{1}{\lambda}U\hat b.
\hat{d} &= -a g +aU\hat b.
   \end{aligned}
\ee
Replacing (\ref{least-square-1}) by (\ref{sketch-direc}), the complexity of calculating the coefficients is reduced from $O(\varrho^2n)$ to $O(\varrho^2 q)$.
%By analyzing the formula (\ref{sketch-direc}), we find the direction from sketching least minimization problem is :
%\be
%\begin{aligned}
%\label{sketch-direc2}
%\hat{d} =- (\lambda I  + UU^\top\Omega^\top \Omega)^{-1}g.
%   \end{aligned}
%\ee
%Compare with the primal format (\ref{inversion}), this is the result by replacing $B_k = UU^\top$ with $UU^\top\Omega^\top \Omega^\top$. $UU^\top\Omega^\top \Omega^\top$ has the same non-zero eigenvalues as $UU^\top$ and for any vector $v$, we have $v^\top UU^\top\Omega^\top \Omega^\top v\geq 0 $.

\textbf{Construction of $\Omega$.} We consider random row samplings where the rows of $\Omega_{i,:}, \; i=1,2,\dots, q,$ are sampled from \be\label{sketch-mat} \omega\leftarrow \frac{e_{i}^\top}{p_i}, i=1,2,\dots,n, \ee with/without replacement, where $\{p_j\}$ are given sampling probabilities. Two common strategies are listed below:
\begin{itemize}
\item Uniform sampling: All $p_i$ are the same and $p_i = \frac{1}{n},\; \forall i=1,2,\dots,n$.
\item Leverage score sampling: Each $p_i$ is proportional to the row norm squares $\|U_{i,:}\|_2^2$, where $U_{i,:}$ is the $i$-th row of $U$, that is, $p_i = \frac{\|U_{i,:}\|_2^2}{\sum_{i=1}^{n}\|U_{i,:}\|_2^2}$.
\end{itemize}
In fact, the kinds of the sketching methods do not have a strong influence on the performance. Moreover, sketching the matrix $U$ by row sampling is cheap.

\subsection{Implicit Computation and Storage of $U$ to Reduce Complexity}
\label{ImplicitConstruction}
Although the computational complexity is reduced by sketching, the memory consumption in (\ref{sketch-direc}) is still large in certain cases. In this part, we take advantage of the structure of the gradient to reduce the memory usage.
%The main idea is from the construction \eqref{grad-back} where the storage of $\hat{G_i}$ and $\hat{A}_i$ is smaller than that of the gradient $u_i$ if $n > (n_G + n_A ) \kappa$. %In this part, we consider the computation in one layer and drop the layer $l$ in the next.

 We first assume that each element of $U_S$ can be approximated as follows:
\be
\label{ApproximationGA}
u_i =\ovec\left( \hat{G_i}\hat{A}_i^\top\right) \approx \ovec\left(G_iA_i^\top \right)= \sum_{j=1}^r {a}_{ij} \otimes {g}_{ij} ,
\ee
where ${G}_i = [{g}_{i1}, \ldots, {g}_{ir}]\in \mathbb{R}^{n_G \times r}$, ${A}_i = [{a}_{i1}, \ldots, {a}_{ir}] \in \mathbb{R}^{n_A \times r}$ and $r \leq \kappa$. \footnote{When $\kappa$ is large enough, the approximation (\ref{ApproximationGA}) can be obtained by computing a partial SVD of $\hat G_i$ or $\hat A_i$ regarding to their sizes.}
%%, i.e.,
%%$$\hat G_i \approx N_{G_i}\Sigma_{G_i}V_{G_i},$$ where $N_{G_i}\in\Rn^{n_{G}\times r}, $ $\Sigma_{G_i}\in\Rn^{r\times r}$ and $V_{G_i}\in\Rn^{r\times \kappa}$. Hence, (\ref{ApproximationGA}) can be obtained by setting $G_i = N_{G_i}\Sigma_{G_i}$ and $A_i = \hat A_iV_{G_i}^\top$. These two matrices can be constructed in a similar fashion if a low-rank approximation to $\hat A_i$ is available. The partial SVD can be obtained by cheap randomized SVD methods. Note that the above scheme is not needed when $\kappa$ is small enough.

We next describe sketching methods to compute $U^\top z$, $U^\top U$ and $Uc$ for any vector $z,c$ by using $\{G_i,A_i\}$.
Denote \be
\begin{aligned}
\label{storeAG}\widetilde{A} &= [A_1,A_2,\dots,A_{\varrho}] \in \Rn^{n_A\times r\varrho},\\ \widetilde{G} &= [G_1,G_2,\dots,G_{\varrho}] \in \Rn^{n_G\times r\varrho}. \end{aligned}\ee
When $n_A$ or $n_G$ is large, we sample the rows of $\widetilde G$ and $\widetilde A$ with two sketching matrices $\Omega_G\in \Rn^{\zeta_G \times n_G}$ and $\Omega_A\in \Rn^{\zeta_A \times n_A}$. Hence, we obtain
\be
\label{sketchedAG}
\begin{aligned}
 \Xi_{\widetilde A }&= \Omega_A\widetilde{A} = [\Xi_{A_1},\dots,\Xi_{A_\varrho}], \\ \Xi_{\widetilde G} &= \Omega_G\widetilde{G} = [\Xi_{G_1},\dots,\Xi_{G_\varrho}],
 \end{aligned}
\ee
where $\Xi_{G_i} =  \Omega_GG_i $ and $\Xi_{A_i} = \Omega_A A_i$. When $n_A$ and $n_G$ are already small enough, we simply let $\Xi_{\widetilde A} =\widetilde A$ and $\Xi_{\widetilde G} = \widetilde G$.

\textbf{Computation of $U^\top z$}.  We sketch $\texttt{mat}(z)$ with the same sketching matrices and define $\Xi_{z} = \Omega_G\texttt{mat}(z)\Omega_A^\top$, where $\texttt{mat}(\cdot):\Rn^{n}\rightarrow \Rn^{n_G\times n_A}$. By the randomized techniques, the $i$-th element of $U^\top z$ can be approximated as:
   \begin{align}
  u_i^\top z &\approx \sum_{j=1}^r{g}_{ij}^\top\texttt{mat}(z){a}_{ij} \label{computeUtz}\\
&  \approx
  \texttt{elesum}\left( \left( \Xi_z^\top \Xi_{G_i}  \right)\odot (\Xi_{A_i})\right),   \label{implicit-sketching-2}
  \end{align}
%  \be
%  \label{Utopz}
%  \begin{aligned}
%  (u^i)^\top z &= \sum_{j=1}^r ({a}_{ij}^\top \otimes {g}_{ij}^\top)\texttt{mat}(z) =  \sum_{j=1}^r{g}_{ij}^\top\texttt{mat}(z){a}_{ij}\\
%  &=\texttt{elesum}\left( \left( \texttt{mat}(z)^\top G_i  \right)\odot A_i\right),
%  \end{aligned}
%  \ee
where $\odot$ is the Hadamard product and $\texttt{elesum}(X)=\sum_{ij} X_{ij}$. %Denote $\mathcal{A}_{\Omega_A,\Omega_G,U}(z) = [  \texttt{elesum}\left( \left( \Xi_z^\top \Xi_{G_1}  \right)\odot (\Xi_{A_1})\right),\dots,  \texttt{elesum}\left( \left( \Xi_z^\top \Xi_{G_\varrho}  \right)\odot (\Xi_{A_\varrho})\right)]^\top.$ Hence, $U^\top z \approx \mathcal{A}_{\Omega_A,\Omega_G,U}(z).$
  %  When the dimension is large, we use the sketching method to compute the expensive matrix-matrix computation in (\ref{Utopz}). By sampling the rows of $G_i$ and $A_i$ with sketching matrix $\Omega_G\in \Rn^{\zeta_G \times n_G}$ and $\Omega_A\in \Rn^{\zeta_A \times n_A}$, (\ref{Utopz}) is changed as:
%  \be
%  \label{implicit-sketching-2}
%  (u^i)^\top z\approx
%  \texttt{ele-sum}\left( \left( \Xi_z^\top \Omega_GG_i  \right)\odot (\Omega_A A_i)\right)
%  \ee
%  where $\Xi_z= \Omega_G\texttt{mat}(z)\Omega_A^\top$, $\Xi_{G_i} =  \Omega_GG_i $ and $\Xi_{A_i} = \Omega_A A_i$.

\textbf{Computation of $U^\top U$}.
%In the next, let us introduce how to compute $U^\top U$ which is crucial in (\ref{inversion}).
%Assume ${G}_i = [{g}_{i1}, \ldots, {g}_{ip}]$, ${A}_i = [{a}_{i1}, \ldots, {a}_{ip}]$, then
%\begin{align}u_i \approx \ovec({G}_i{A}_i^\top) &= \sum_{j=1}^p {a}_{ij} \otimes {g}_{ij}.\label{approx-1}
%\end{align}
Similarly, the $(i,j)$ element of $U^\top U$ is approximated as:
\begin{align}
(u_i)^\top u_j &\approx \left(\sum_{k=1}^r a_{ik} \otimes g_{ik}\right)^\top \left(\sum_{k=1}^r a_{jk}\otimes g_{jk}\right)\label{UtUapprox1} \\
%& = \texttt{elesum}\left ( (A_i^\top A_j) \odot (G_i^\top G_j)\right) \\
%= \sum_{s=1}^r\sum_{t=1}^r \left(a_{is}^\top a_{jt}\right)\cdot\left( g_{is}^\top g_{jt} \right)
%\label{UtUapprox1}\\
&\approx \texttt{elesum}\left ( (\Xi_{A_i}^\top \Xi_{A_j}) \odot (\Xi_{G_i}^\top \Xi_{G_j})\right)\label{UtUapprox2}
. \end{align}

\textbf{Computation of $Uc$}. To compute $Uc$, we first have to compute the per-sample gradient $u_i$ for all $i\in S$, multiply them with corresponding $c_i$ and finally sum them together:
\begin{align}
 Uc = \sum_{i\in S} u_ic_i \approx \sum_{i\in S} \ovec({G}_i{A}_i^\top) c_i
% = \sum_{i\in S} \ovec({G}_i \sqrt{|c_i |}\cdot \frac{c_i}{\sqrt{|c_i|}}{A}_i^\top)
 \label{approx-2}.
 \end{align}
The process (\ref{approx-2}) is expensive since it requires the computation of $\varrho$ matrix-matrix products and vectorizations as well. When the dimension $n$ is large, the computational cost is not tractable. Alternatively, we assume $A_i$ and $G_i$ are independent and approximate $Uc$ by $ \mathcal{C}_U(c)$, i.e., the product between the weighted averages of $G_i$ and $A_i$:
 \begin{align}
% Uc
% \approx
  \mathcal{C}_U(c)
 =\ovec \left( \sum_{i\in S} \sqrt{|c_i|}G_i\right) \left(\sum_{i\in S} \frac{c_i}{\sum_{i=1}^\varrho\sqrt{|c_i|}}A_i\right)^\top. \label{approx-3}
 \end{align}
Therefore, an explicit calculation and storage of $u_i$ is avoided, and only one matrix-matrix multiplication is needed.

Let $\mathcal{A}_{\Omega_A,\Omega_G,U}(z)$ and $\mathcal{B}_{\Omega_A,\Omega_G, U}$ be the approximation of $U^\top z$ by \eqref{implicit-sketching-2} and $U^\top U$ by \eqref{UtUapprox2}. Combining them with (\ref{approx-3}), the direction can be obtained as follows:%Combining (\ref{implicit-sketching-2}), (\ref{UtUapprox2}) and (\ref{approx-3}), the direction can be obtained as follows:
\be\label{implicit-dir} \hat{d} =-\frac{1}{\lambda} g + \mathcal{C}_U(\hat{b}),\ee
where $\hat b =(\mathcal{B}_{\Omega_A,\Omega_G,U}+\lambda I)^{-1}\mathcal{A}_{\Omega_A,\Omega_G,U}( g) $. Note that (\ref{implicit-sketching-2}) and (\ref{UtUapprox2}) are equal to $((\Omega_A \otimes \Omega_G)u_i)^\top((\Omega_A \otimes \Omega_G) z)$ and $((\Omega_A \otimes \Omega_G)u_i)^\top ((\Omega_A \otimes \Omega_G)u_j)$, respectively. Therefore, the computation of $\hat b $ here can be seen as a special case of (\ref{sol-ls-2}) by choosing $\Omega = (\Omega_A \otimes \Omega_G).$

\LinesNumberedHidden
\begin{algorithm2e}
%\small
\caption{The Computation of the Direction}
\label{alg:dir-compute}
\lnlset{alg:dir-compute}{1}{\textbf{INPUT:}} Curvature matrix update frequency $T$, threshold $\mathcal{T}$ and regularization $\lambda_k$.\\
\For{\text{layer} $l = 0,1,..., L$}{
\lIf{$n^l < \mathcal{T}$ }{\\
 \quad  \lIf{ $k$ mod T $=0$}{\\
   \qquad construct $U_k^l$ based on the sample set $\tilde{S}_k$
  }
 \quad \lElse{ \\
   \qquad set $U_k^l = U_{k-1}^l$
 }
  \quad Construct the sketching matrix $\Omega^l_k$ by (\ref{sketch-mat});\\
  \quad Solve the sketched least squares problem(\ref{least-square-2}); \\
  \quad Set $\hat d_k^l$ by (\ref{sketch-direc})
}
\lElse{ \\
  \quad \lIf{ $k$ mod T $=0$}{\\
    \qquad update $\widetilde A_{k}^l, \widetilde G_k^l$ by the set $\tilde{S}_k$ and (\ref{storeAG})}
\quad \lElse{\\
 \qquad set $\widetilde A_k^l = \widetilde A_{k-1}^l$ and $\widetilde G_k^l=\widetilde G_{k-1}^l$}
\quad {Construct $(\Omega_G)^l_k$ and $(\Omega_A)^l_k$ by (\ref{sketch-mat})};\\
\quad Compute $\hat d_k^l$ by (\ref{implicit-dir})
}
} % For
\lnlset{alg:prefim}{7}{\textbf{OUTPUT:} $\hat d_k :=[(\hat d_k^1)^\top,\dots,(\hat d_k^L)^\top]^\top$.}
\end{algorithm2e}
We summarize the computation of the direction $\hat d$ in Algorithm \ref{alg:dir-compute} and SENG in Algorithm \ref{alg:SENG}, respectively.

 \LinesNumberedHidden
\begin{algorithm2e}[ht]
%\small
\caption{The SENG Methods}
\label{alg:SENG}
\lnlset{alg:prefim}{1}{\textbf{INPUT:}} Initial parameter $\theta_1$, step size $\{\alpha_k\}$ and regularization $\{\lambda_k\}.$\\
\For{$k = 1,..., T$}{
\lnlset{alg:prefim}{2}{Choose the samples $S_k$ and compute $g_k$;} \\
%\lnlset{alg:prefim}{3}{Construct $U_{S_k}$ }; \\
\lnlset{alg:prefim}{3}{Compute the direction $\hat d_k$ by Algorithm \ref{alg:dir-compute};\\ }
\lnlset{alg:prefim}{4}{Set $\theta_{k+1} = \theta_k+\alpha_k  \hat d_k$;}
}
\lnlset{alg:prefim}{5}{\textbf{OUTPUT:}} $\theta_{T+1}$.
\end{algorithm2e}

\subsection{Computational Cost and Memory Consumption}
\label{computation-and-storage}
In this part, we summarize computational cost and memory consumption of our methods in Table \ref{ComStoCost}. We can observe that the randomized methods reduce both the computational cost and memory usage.

The SENG method avoids the inversion of high dimensional matrices. The size of the matrices equals the batch size. In practice, this number is often set to be 32, 64 or 256, which means that the computational cost of the matrix inversion is not a bottleneck. Instead, the matrix-matrix multiplications are required each iteration, but the cost is alleviated by our proposed sketching strategies. Note that the matrix inversion is the main computational cost in KFAC method. For example, the sizes of matrices to be inverted are 4, 608 and 512 for the case V in Table \ref{resnet-stat}. Since the cost of matrices multiplications is usually smaller than that of the matrix inversion, generally speaking, the SENG methods can show greater advantages in large neural networks where the matrix inversion takes up most of the computational time.

\section{Distributed SENG}
In this section, we extend our methods to the distributed setting. Assume that $\mathcal{M}$ parallel workers are available and the samples $S_k$ are allocated to the $\mathcal{M}$ workers evenly, that is, $S_k = [S_{k,1},S_{k,2},\dots,S_{k,\mathcal{M}}],$
where $S_{k,i}$ is the samples in the $i$-th worker. The corresponding mini-batch gradient and the collection of gradients are also computed and stored in different workers accordingly
\[
\begin{aligned}
g_k =\frac{1}{\mathcal{M}}\sum_{i=1}^{\mathcal{M}} g_{S_{k,i}}, \ \qquad
 U_k = [U_{S_{k,1}},\dots,U_{S_{k,\mathcal{M}}}].
\end{aligned}
\]
Then, the direction in (\ref{inversion}) can be rewritten as follows
\be
\label{inversion-dist}
\begin{aligned}
d_k
%&= -a_kg_{k} + a_kU_kb_k\\
 & = -a_k \frac{1}{\mathcal{M}}\sum_{i=1}^{\mathcal{M}} g_{S_{k,i}}+ a_k \frac{1}{\mathcal{M}}\sum_{i=1}^{\mathcal{M}} U_{S_{k,i}} b_{k,i}.
\end{aligned}
\ee
where $a_k=\frac{1}{\lambda_k}$, $b_k = [b_{k,1}^\top,\dots,b_{k,\mathcal{M}}^\top]^\top,\; b_{k,i}\in \Rn^{|S_{k,i}|}$.
%\be\label{sol-ls-1-dist}
%\begin{aligned}
%&b_k=(\lambda_k I + U_k^\top U_k)^{-1} (U_k^\top g_{k}). \\
%\end{aligned}
%\ee
%\be\label{sol-ls-12-dist}
%\begin{aligned}
%&(\lambda_k I + U_k^\top U_k)^{-1}  \\
%&= \left [ \begin{matrix} &\lambda_k I + U_{S_{k,1}}^\top U_{S_{k,1}},&\dots, &\lambda_k I + U_{S_{k,1}}^\top U_{S_{k,\mathcal{N}}}\\
%&\vdots &\vdots &\vdots\\
%&\lambda_k I + U_{S_{k,\mathcal{N}}}^\top U_{S_{k,1}},&\dots, &\lambda_k I + U_{S_{k,\mathcal{N}}}^\top U_{S_{k,\mathcal{N}}}\\
%\end{matrix} \right]^{-1}
%%\left [\begin{matrix}
%%&U_{S_{k,1}}^\top g_k\\
%%&\vdots\\
%%&U_{S_{k,\mathcal{N}}}^\top g_k
%%\end{matrix}\right ] .
%\end{aligned}
%\ee
%
%\be\label{sol-ls-13-dist}
%\begin{aligned}
%&U_k^\top g_{k} =
%%&= \left [ \begin{matrix} &\lambda_k I + U_{S_{k,1}}^\top U_{S_{k,1}},&\dots, &\lambda_k I + U_{S_{k,1}}^\top U_{S_{k,\mathcal{N}}}\\
%%&\vdots &\vdots &\vdots\\
%%&\lambda_k I + U_{S_{k,\mathcal{N}}}^\top U_{S_{k,1}},&\dots, &\lambda_k I + U_{S_{k,\mathcal{N}}}^\top U_{S_{k,\mathcal{N}}}\\
%%\end{matrix} \right]^{-1}
%\left [\begin{matrix}
%&U_{S_{k,1}}^\top g_k\\
%&\vdots\\
%&U_{S_{k,\mathcal{N}}}^\top g_k
%\end{matrix}\right ] .
%\end{aligned}
%\ee

To compute $U_k^\top U_k$ in (\ref{sol-ls-1}), we need compute $U_{S_{k,i}}^\top U_{S_{k,j}}$ for all $i,j = 1,\dots, \mathcal{M}$. However, since $U_{S_{k,j}}$ is stored by different workers, extensive communication cost is required among them. In the next, we use the block approximation to overcome this difficulty.

%Note that the sketchy methods stated in Sec \ref{SketchLS}-\ref{ImplicitConstruction} are also used in this setting.
%\subsection{SENG-Dist-I}
%Let us introduce a new notation which emphasizes the role of the samples
%\be d(S) = -ag_{S} + aU_S(\lambda I + U_S^\top U_S)^{-1}(U_S^\top g_{S}) \approx \text {Approximation methods in Sec. \ref{SketchLS}-\ref{ImplicitConstruction}}.\ee
%Therefore, $d_k$ in (\ref{inversion-dist}) can be also represented as $d_k = d(S_k)$. Notice that the computation of $d(S_{k,i})$ can be computed efficiently at each node independently and therefore do not need extra communication. Simultaneously, we can compute direction $d(S_{k,i})$ based on the samples $S_{k,i}$ at each node. The first proposed method is simple to aggregate $d(S_{k,i})$ by $\text{All-Reduce}$ operation as follows:
%\be \label{direc-dist-1}d_{k,\text{dist-I}} = \frac{1}{\mathcal{N}}\sum_{i=1}^\mathcal{N}d(S_{k,i}).\ee
%The communication cost of calculating (\ref{direc-dist-1}) is small, each iteration, only a tensor of the same size as the gradient is synchronized among the nodes.
%\subsection{Algorithm Description}

\subsection{Distributed SENG Algorithm}
\textbf{Block Diagonal Approximation to $U^\top U.$}

A direct idea is to use a diagonal approximation of $U_k^\top U_k$ to calculate the components $b$ in (\ref{sol-ls-1}). Specifically, $(\lambda_k I + U_k^\top U_k)^{-1} $ is approximated by a block-diagonal matrix with $\mathcal{M}$ blocks and
 \be\label{sol-ls-1-dist}
\begin{aligned}
b_k
%&=(\lambda_k I + U_k^\top U_k)^{-1} (U_k^\top g_{k}) \\
\approx \text{block-approx}\left \{ (\lambda_k I + U_k^\top U_k)^{-1}\right \} (U_k^\top g_{k})
%& = \left [\dots; (\lambda_k I + U_{S_{k,i}}^\top U_{S_{k,i}})^{-1} (U_{S_{k,i}}^\top g_{k});\dots \right]
:=\hat b_k.
%\left [ \begin{matrix} &\lambda_k I + U_{S_{k,1}}^\top U_{S_{k,1}},&\dots, &0\\
%&\vdots &\dots ,&\vdots\\
%&0,&\dots, &\lambda_k I + U_{S_{k,\mathcal{N}}}^\top U_{S_{k,\mathcal{N}}}\\
%\end{matrix} \right]^{-1}
%\left [\begin{matrix}
%&U_{S_{k,1}}^\top g_k\\
%&\vdots\\
%&U_{S_{k,\mathcal{N}}}^\top g_k
%\end{matrix}\right ] := \hat b_k(g_k) .
\end{aligned}
\ee

Since $\hat b_{k,i}$ only relates to $U_{k,i}$ and $g_k$, once $g_k$ is available, we can compute $ \hat b_k$ and further calculate $ U_{S_{k,i}} \hat b_{k,i}$ simultaneously. Therefore, the direction can be obtained by averaging (All-Reduce) them among all workers by \eqref{inversion-dist}. The extra communication traffic is one tensor with the same size as the gradient and this synchronization is done separately from that of the gradient.

\begin{table}
\scriptsize
\centering
\caption{A summary of the computational and memory complexity.}
\begin{tabular}{c|ccc|ccc}
\hline \hline
&\multicolumn{3}{c|}{Computational Cost} &\multicolumn{3}{c}{Memory Consumption} \\
 \hline
% Original Direction Computation (\ref{eq:dk})&\multicolumn{3}{c|}{$n^3$} &\multicolumn{3}{c}{$np$} \\
Low-rank Computation (\ref{inversion}) &\multicolumn{3}{c|}{$\varrho^3 +\varrho^2 n$} &\multicolumn{3}{c}{$\varrho n$} \\
%\hline
Original LS (\ref{least-square-1})&\multicolumn{3}{c|}{$\varrho^2 n$} &\multicolumn{3}{c}{$\varrho n$} \\
Sketchy LS (\ref{least-square-2})&\multicolumn{3}{c|}{$\varrho^2 q$} &\multicolumn{3}{c}{$\varrho n$} \\
%\hline
$U^\top z$ &\multicolumn{3}{c|}{$\varrho n$}  & \multicolumn{3}{c}{$\varrho n$}  \\
$U^\top z$ (\ref{computeUtz})&\multicolumn{3}{c|}{$n_An_G\varrho r $} &\multicolumn{3}{c}{$(n_A+n_G)r\varrho$} \\
Randomized $U^\top z$ (\ref{implicit-sketching-2})&\multicolumn{3}{c|}{$\zeta_A\zeta_G \varrho r$} &\multicolumn{3}{c}{$(\zeta_A+\zeta_G)r\varrho$} \\
%\hline
$U^\top U$&\multicolumn{3}{c|}{$\varrho^2 n$} &\multicolumn{3}{c}{$\varrho n$} \\
 $U^\top U$ (\ref{UtUapprox1})&\multicolumn{3}{c|}{$r^2\varrho^2(n_{ A} +n_{ G})$} &\multicolumn{3}{c}{$(n_A+n_G)r\varrho$} \\
Randomized $U^\top U$ (\ref{UtUapprox2})&\multicolumn{3}{c|}{$r^2\varrho^2(\zeta_{ A} +\zeta_{ G})$} &\multicolumn{3}{c}{$(\zeta_A+\zeta_G)r\varrho$} \\
%\hline
$Uc$&\multicolumn{3}{c|}{$\varrho n$} &\multicolumn{3}{c}{$\varrho n$} \\
$Uc$ (\ref{approx-2})&\multicolumn{3}{c|}{$n_An_Gr\varrho$} &\multicolumn{3}{c}{$(n_A+n_G)r\varrho$} \\
 $Uc$ (\ref{approx-3})&\multicolumn{3}{c|}{$n_An_Gr$} &\multicolumn{3}{c}{$(n_A+n_G)r\varrho$} \\
 \hline \hline
\end{tabular}
\label{ComStoCost}
\vspace{-2ex}
\end{table}

We further consider a distributed variant which has the same communication cost and the extra tensor can be synchronized simultaneously with the gradient. The coefficient $b_k$  is approximated by using the gradient $g_{k-1}$ in the last step as:
\be
\label{Dist-syn-2}
\begin{aligned}
%\hat g_{k-1} &\approx (1-\zeta_k)\hat g_{k-2} + \zeta_k g_{k-1}\\
b_k \approx \text{block-approx}\left \{ (\lambda_k I + U_k^\top U_k)^{-1}\right \} (U_k^\top  g_{k-1}) := \tilde b_k.
\end{aligned}
\ee

The sketching method presented in Sec \ref{SketchLS} can be used naturally in the above distributed algorithms by replacing $U_{S_{k,i}}$ by $\Omega_k U_{S_{k,i}}$ in (\ref{sol-ls-1-dist}) and \eqref{Dist-syn-2}. If we use the implicit computation and storage of $U_k$ in Sec \ref{ImplicitConstruction}, the diagonal approximation to $U_k^\top U_k$ can also be applied to (\ref{UtUapprox2}).

\textbf{Computation of $U^\top z$}
%
%Note that the $g_k$ in \eqref{sol-ls-1-dist} or $g_{k-1}$ in \eqref{Dist-syn-2} has already been synchronized among different nodes before computing $d_k$. The operation $U^\top z$
can be the same way as that in Sec \ref{SketchLS} or Sec \ref{ImplicitConstruction}.

\textbf{Computation of $Uc$}

For the layers that use the explicit mechanisms, the operation $Uc$ is the same as (\ref{inversion-dist}). Otherwise, we use a new mechanism to overcome the communication in the summation of the absolute values $c_k$ presented in \eqref{approx-3}. We first compute $(G^c)_{i} = \sum_{j\in S_{k,i}} \sqrt{|c_j|}G_j$ and $(A^c)_i = \sum_{j\in S_{k,i}} \frac{c_j}{\sum_{i=1}^{|S_{k,i}|}\sqrt{|c_j|}}A_j$ in different workers, then synchronize them and approximate $Uc$ as:
 \begin{align}
 &Uc  \approx \ovec \left( \sum_{i =1 }^\mathcal{M} (G^c)_i\right) \left(\sum_{i =1 }^\mathcal{M} (A^c)_i\right)^\top.
 \end{align}
%However, the computation $Uc$ in (\ref{approx-3}) need to be adjusted to reduce the synchronous cost. We consider the following approximation:
%\begin{align}
%& Uc  \approx \mathcal{C}_U(c) \\ &\approx \sum_{i=1}^{\mathcal{N}} \ovec \left( \sum_{i\in S_{k,i}} \sqrt{|c_i|}G_i\right) \left(\sum_{i\in S_{k,i}} \frac{c_i}{\sum_{i=1}^{|S_{k,i}|}\sqrt{|c_i|}}A_i\right)^\top.
% \end{align}
% In this way, the size of synchronized tensor is also the same as the gradient.

% \subsection{SENG-Dist-III}
%  \be\label{sol-ls-1-dist}
%\begin{aligned}
%\hat b_k(\widetilde g_k) &=(\lambda_k I + U_k^\top U_k)^{-1} (U_k^\top \widetilde g_{k}), \\
%\widetilde g_k &= m_k \widetilde g_{k-1} + (1-m_k) g_{k-1}.\\
%\end{aligned}
%\ee

\section{Convergence Analysis}
The convergence analysis of SENG is established in this section. We first prove that for the general objective function $\Psi(\theta)$, the algorithm converges to the stationary point globally. Furthermore, it is shown that the SENG method can converge to the optimal solution linearly under some mind conditions in the fully-connected neural network.
\subsection{Global Convergence}
In this part, we show the global convergence when $g_k$ is the unbiased mini-batch gradient. We assume the directions for all layers are obtained by the sketched subproblem (\ref{least-square-2}) in Sec \ref{SketchLS}. Since the update rules for all layers are identical, we only consider one layer and drop the layer indices. The main idea of our proof is to first estimate the error between $d_k$ (\ref{inversion}) and $\hat d_k$ (\ref{sol-ls-2}), and the descent of the function values, then balance them by choosing a suitable step size. We give some necessary assumptions below.

\begin{assumption} Let $\eta_k,\epsilon_k \in(0,1)$. Let $v$ be any fixed vector and $N_k \in \Rn^{n\times \rho_k}$ be an orthogonal basis for the column span of $U_k$, where $\rho_k = \text{rank}(U_k).$ Let $\Omega_k \in \Rn^{q_k\times n}$ be a sketching matrix, where the sample size $q_k$ depends on $\eta_k$, $\epsilon_k$ and $\delta_k$. The following two assumptions hold for all $k$ with a probability $1-\delta_k$:
\begin{itemize}
\item[A.1)] $\|N_k^\top \Omega_k^\top \Omega_k N_k- I\|_2 \leq \eta_k,$
\item[A.2)] $\|N_k^\top \Omega_k^\top \Omega_k v - N_k^\top v\|^2_2 \leq \epsilon_k \|v\|_2^2.$
\end{itemize}
\end{assumption}
Assumptions A.1-A.2 are called subspace embedding property and matrix multiplication property, respectively. They are standard in related sketching methods \cite{10.5555/3122009.3242075,10.5555/2946645.2946694}. When the sample size $q_k$ is large enough, Assumptions A.1 and A.2 will be satisfied. Throughout the paper, we assume the sketching matrices are independent from each other and from the stochastic gradients. The sequence $\{\theta_k\}$ do not affect the stochasticity of $\mathbf{\Omega} = \{\Omega_k\}_{k=1}^\infty$.
\begin{assumption}
\begin{itemize}
\item[B.1]$\Psi$ is continuously differentiable on $\Rn^n$ and is bounded from
    below. The gradient $\nabla \Psi$ is Lipschitz continuous on $\Rn^n$ with modulus $L_\Psi \geq 1$.
\item[B.2] There exists positive constants $h_1,h_2$ such that the matrix holds:
$
h_1I \preceq (B_k+\lambda_k I) \preceq h_2 I
$
for all k.
\item[B.3] $B_k$ and $g_k$ are independent for any iteration $k$. In addition, it holds almost surely that the stochastic gradient is unbiased, i.e., $\Expe[g_k| {\theta_{k-1},\dots,\theta_0}, \mathbf{\Omega}] = \nabla \Psi(\theta_k)$ and the variance of the stochastic gradient is bounded:
\[
 \Expe[\|g_k -\nabla \Psi(\theta_k)
 \|^2| {\theta_{k-1},\dots,\theta_0}, \mathbf{\Omega}]  \leq  \sigma_k^2.
\]

\end{itemize}
\end{assumption}
Assumptions B.1-B.3 are common in stochastic quasi-Newton type methods \cite{ByrHanNocSin16,WanMaGolLiu17,yangSEQN}. We next summarize our analysis.

\begin{theorem}\label{thm:globalconv}
Suppose that Assumptions A.1-A.2 and B.1-B.3 are satisfied and $
 \frac{\sqrt{\epsilon_k }+\eta_k}{1-\eta_k}$ is small enough. If the step size $\{\alpha_k\}$ further satisfies $\alpha_k \leq \min\left \{\frac{1}{2L_{\Psi}},\frac{h_1^2}{2L_\Psi h_2} \right \}$,
$\sum \alpha_k = \infty$ and $\sum \alpha_k\sigma_k^2 < \infty,$ it holds
%for Algorithm \ref{alg:SENG}
\begin{equation*}
\lim_{k\rightarrow \infty} \nabla \Psi(\theta_k) = 0
\end{equation*}
with probability $\Pi_{k=0}^\infty (1-\delta_k)$.
\end{theorem}
The proof of Theorem \ref{thm:globalconv} is shown in the Appendix.
\begin{figure*}[t]
\centering
\vspace{-1ex}
\begin{tabular}{ccc}
\includegraphics[width=0.33\textwidth]{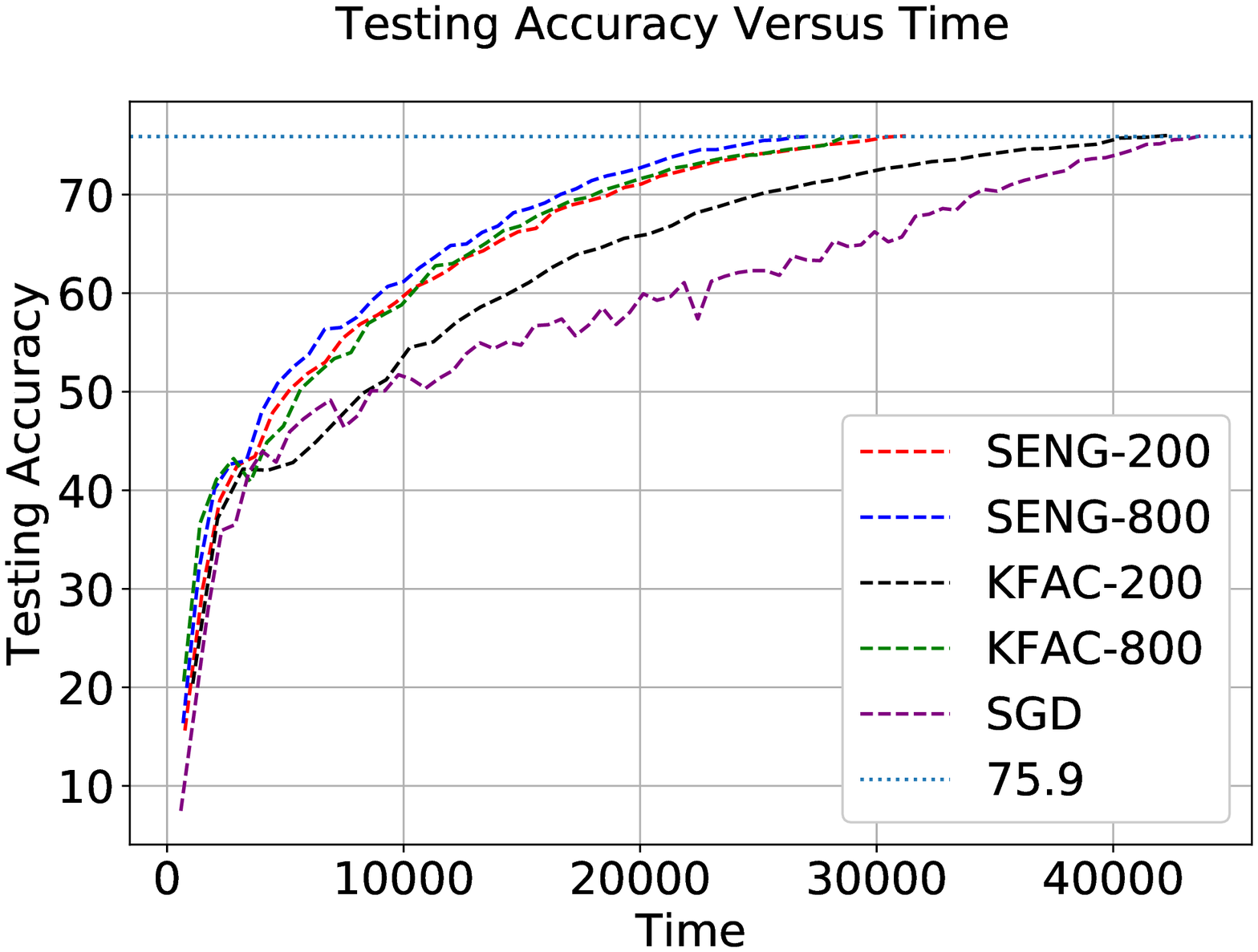}
\includegraphics[width=0.33\textwidth]{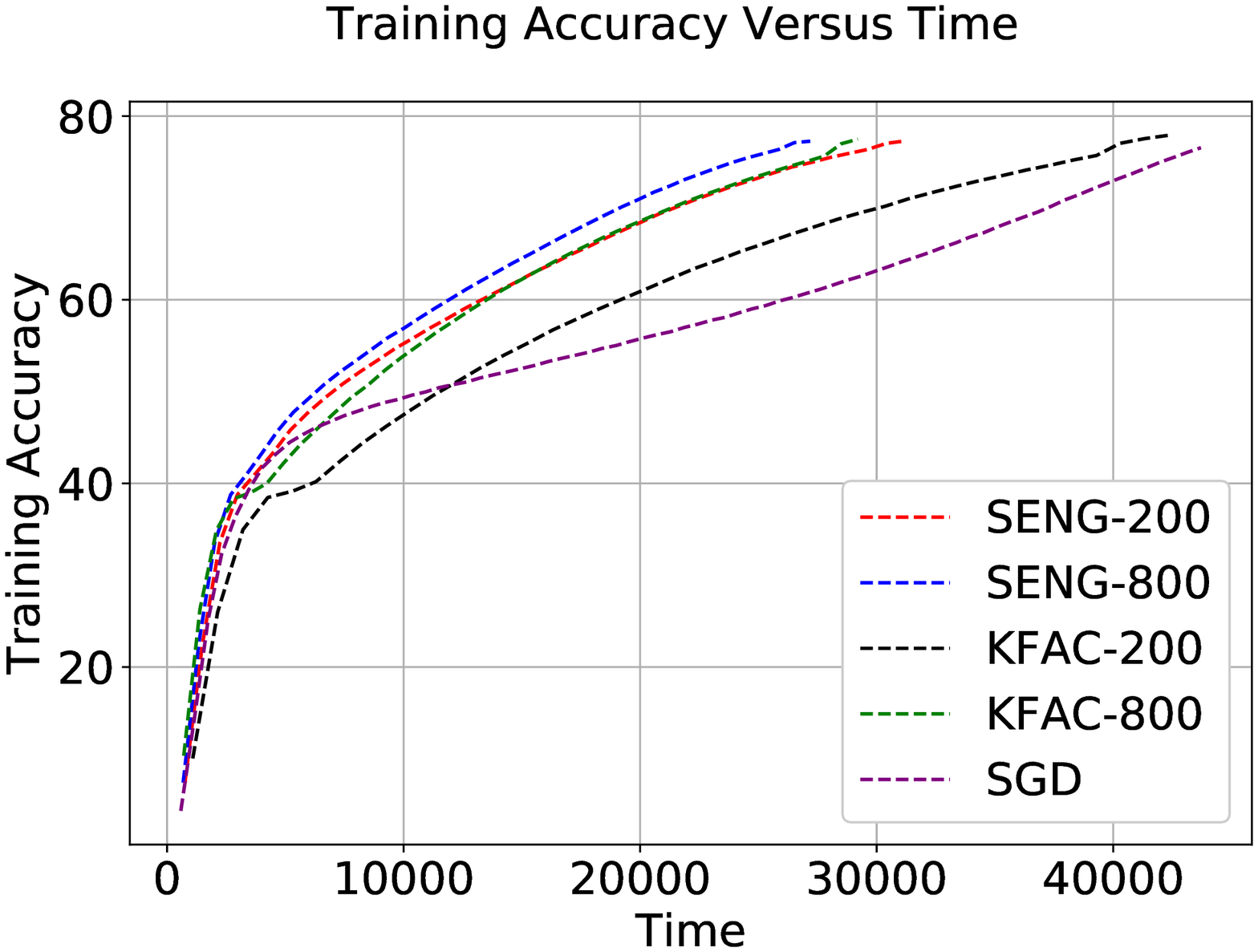}
\includegraphics[width=0.33\textwidth]{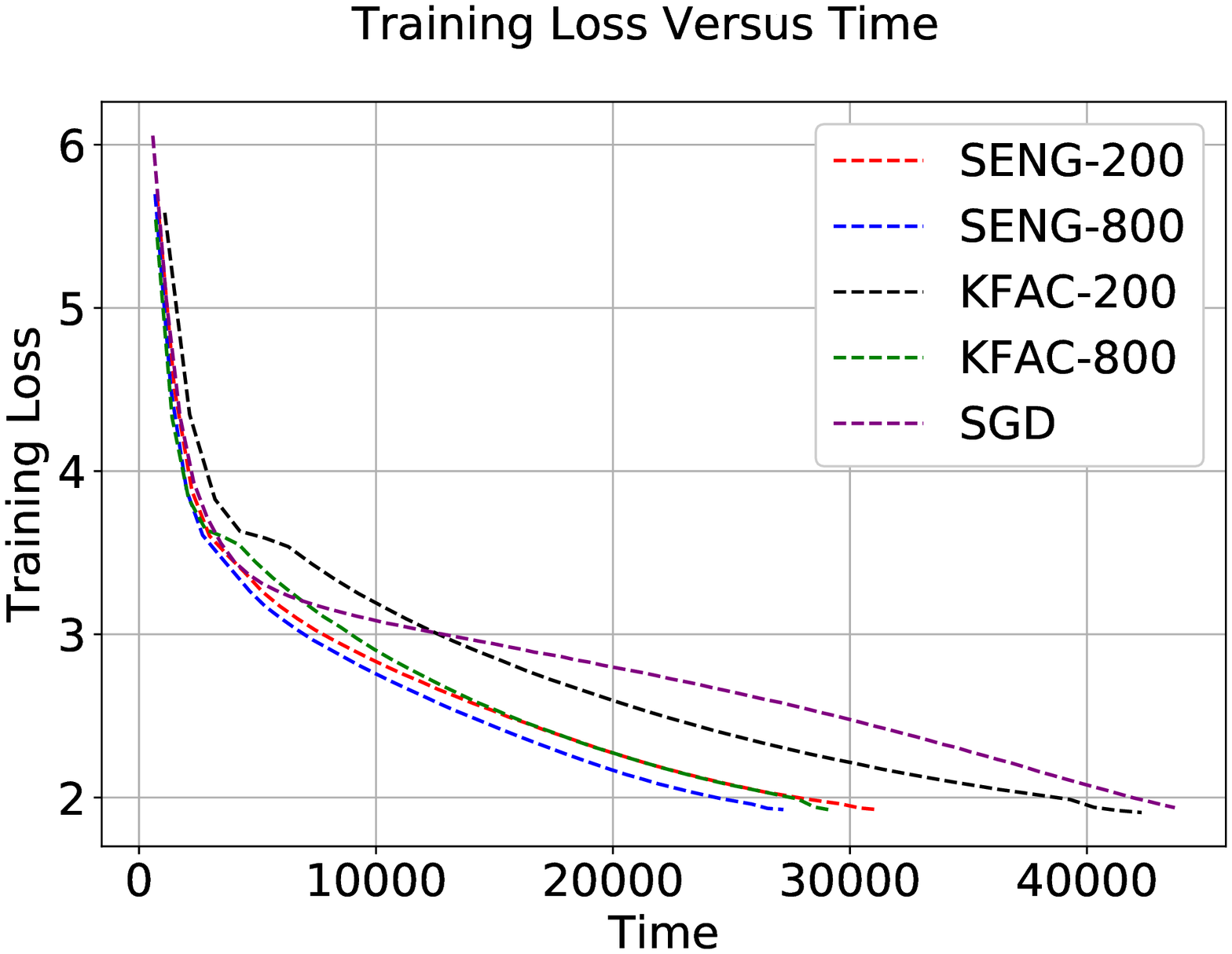}\\
\end{tabular}
\vspace{-2ex}
\caption{Numerical performance on ResNet50 on ImageNet-1k.}
\label{imagenet-time}
%Upper Three: The numerical comparison with KFAC and SGD. Lower Three: The performance between SENG variants with different matrix updating frequency. The number behind SENG is the number of frequency.
\vspace{-2ex}

\end{figure*}
\subsection{Linear Convergence in Wide Neural Networks}
In this part, we analyze SENG for over-parameterized neural networks in the NTK regime. Consider a fully-connected network with a single output \footnote{For simplicity, we consider one-dimension output and fix the second layer. However, it is easy to extend the analysis to the multi-dimensional output or the case for jointly training both layers.}:
%\begin{align*}
%f_{l}(x,\theta) = \frac{\sigma_w}{\sqrt{M_{l-1}}}W_lh_{l-1} + \sigma_b b_l, \qquad h_l = \phi(f_l),
%\end{align*}
\begin{align*}
f(x,\theta) = \frac{1}{\sqrt{\hat m}}\sum_{i=1}^{\hat{m}} a^i \phi((w^i)^\top x) = \frac{1}{\sqrt{\hat{m}}}a^\top \phi(W x)  ,
\end{align*}
%for $l=1,\dots, L,$
%where $h_0=x \in \Rn^{m_0\times 1}$ is the input, $\theta$ is the set of all learned parameters $\{W_l\in \Rn^{m_{l}\times m_{l-1}},b_l \in \Rn^{m_l}\}_{l=1}^L$, $\sigma_w$ and $\sigma_b$ are scaling factors.
where $x\in\Rn^{m_0}$ is the input, $\phi$ is the Relu activation function and $a = [a^1,\dots, a^{\hat{m}}]^\top \in \Rn^{\hat m}.$ $\theta = \ovec{(W)}\in\Rn^n$ is the set of the learned parameters, where $W = [w^1,\dots w^{\hat{m}} ] $. For a given dataset $\{x_i,y_i\}_{i=1}^N$,  we consider the case where the loss is set to be MSE:
\be
\label{MSE-NTK}
\Psi(\theta) = \frac{1}{2} \|f(\theta) - y\|_2^2,
\ee
%where $f(\theta) = [f_L(x_1,\theta)^\top,\dots, f_L(x_N,\theta)^\top]^\top \in \Rn^{m_L N}$ and $y = [y_1^\top,\dots,y_N^\top]^\top\in \Rn^{m_L N}.$
where $f(\theta) = [f(x_1,\theta)^\top,\dots, f(x_N,\theta)^\top]^\top \in \Rn^{ N}$ and $y = [y_1^\top,\dots,y_N^\top]^\top\in \Rn^{N}.$
Denote $J(\theta) $ by the collections of Jacobian matrices $\frac{\partial f(\theta)}{\partial \theta }^\top \in \Rn^{n\times N}.$ Hence, the gradient can be written as $\nabla \Psi(\theta_k) = J_k^\top (f^k-y),$ where $J_k = J(\theta_k)$ and $f^k = f(\theta_k).$
%It is not difficult to extend Theorem \ref{NGD-Theory-1} to our SENG method.
%Let us first see the iteration update of the expected version of SENG (where $U_k = J_k^\top$) for MSE loss :
%\[
%\theta_{k+1} = \theta_k - \eta J_k^\top (J_kJ_k^\top)^{-1}(J_kJ_k^\top)^{-1} J_kJ_k^\top f^k
%\]

Consider the SENG using \eqref{sol-ls-2} applied to $U_k = J_k^\top$ with $\alpha_k \equiv \alpha$. We next briefly explain a key step of our proof by deriving an equivalent format of the direction. The update rule is $
\theta_{k+1} = \theta_k - \alpha d_k,
$
where $d_k = \frac{1}{\lambda_k}J_k^\top\left ( I -  (\lambda_k I + \widetilde M_k )^{-1}\widetilde M_k \right )( f^k-y)$ and $\widetilde M_k = J_k\Omega_k^\top \Omega_k J_k^\top.$
Assume  $\widetilde M_k$ is invertible and its eigenvalue decomposition is $ Q_k \widetilde \Sigma_k Q_k$ where $Q_k$ is orthogonal and $\widetilde \Sigma_k$ is diagonal. Then, we obtain: \vspace{-1ex}
\be
\begin{aligned}
d_k
%& \frac{1}{\lambda_k}J_k^\top\left ( I -  (\lambda_k I + J_k\Omega_k \Omega_k^\top J_k^\top)^{-1}J_k\Omega_k \Omega_k^\top  J_k^\top \right ) f^k \\
=& \frac{1}{\lambda_k}J_k^\top \left (Q_k \left( I - (\lambda_k I + \widetilde \Sigma_k)^{-1}  \widetilde \Sigma_k \right ) Q_k^\top \right)( f^k-y)\\
= & J_k^\top \left (Q_k ( \lambda_k I + \widetilde \Sigma_k)^{-1}  Q_k^\top \right) (f^k-y),
\end{aligned}
\vspace{-1ex}
\ee
where the second equality uses the fact that $\widetilde \Sigma_k$ is a diagonal matrix.
The difference of SENG from NGD can be seen clearer by letting $\lambda_k \rightarrow 0$. Since $\widetilde M_k$ is positive definite due to over-parameterization, each diagonal entry of $\widetilde \Sigma_k$ is positive and we have: \vspace{-1ex}
$$d_k = J_k^\top \left (Q_k \widetilde \Sigma_k^{-1}  Q_k^\top \right)(f^k-y) = J_k^\top \widetilde M_k^{-1}(f^k-y).$$
Note that the direction of NGD is $J_k^\top (J_kJ_k^\top)^{-1}(f^k-y)$, see \cite{zhang2019fast}. Hence, the main difference is the occurrence of $\Omega_k^\top \Omega_k$ in $\widetilde M_k$. 

 Our convergence is established in the next theorem.
%The update of SENG in this case becomes:
%\[
%\begin{aligned}
%\theta_{k+1}
%& =  \theta_k - \alpha J_k^\top (J_k\Omega_k^\top \Omega_kJ_k^\top)^{-1}(f^k-y).
%\end{aligned}
%\]
\begin{theorem}
\label{ntk-theorem}
Assume that (1) the initialization $ w_0^i \sim \mathcal{N}(0,\nu^2 I)$, $a_0^i \sim \text{unif}\left(\{-1,+1\}\right)$ for $i=1,\dots,\hat m$; (2) the sketching matrices $\{\Omega_k\}$ are independent from the initialization; (3) each $\|x_i\|_2 = 1$, $|y_i| =\mathcal{O}(1)$ and $x_i \neq x_{i'}, \ \forall i\neq i' $. Then, under the Assumptions A.1-A.2, if $\eta_k$, $\epsilon_k$ are small enough and $\hat m=\Omega\left( \frac{n^4}{\nu^2\lambda_0^4\delta^3}\right)$, it holds with a constant $\zeta \in (0,1)$,
$$ \|f^{k} - y\|_2^2 \leq \zeta^k \|f^0 - y\|_2^2$$
 with probability $(1-\delta)\Pi_{k=0}^\infty (1-\delta_k)$ over the random initialization and the sketching.
\end{theorem}
The proof of Theorem \ref{ntk-theorem} is shown in the Appendix. Here, $\lambda_0$ is the smallest eigenvalue of limiting Gram matrix. $\delta_0$ and $\delta_k$ are related to the weight initialization and sketching, respectively. The proof follows from \cite{zhang2019fast} by bounding two errors. One error is the difference between the natural gradient flow and the NGD sequence while the other is the difference between the sequence of NGD and that of the SENG. Theoretical understanding of SENG can be extended to the multi-layer fully-connected networks case similar to that in \cite{Karakida2020UnderstandingAF}.

\section{Numerical Experiments}

\label{resnet18-cifar10}
In this part, we report the numerical results of SENG and make comparisons with the state-of-the-art methods. The performance is shown on the classical neural networks ``ResNet18'' and ``VGG16\_bn'' with three commonly used datasets CIFAR10, CIFAR100 and SVHN. Furthermore, we consider the ResNet50 on ImageNet-1k classification problem and show the advantages over the SGD (with momentum) and KFAC.
% with an average of 5 different runs on the CIFAR10 dataset with the ResNet18 \cite{he2016deep} .  For a fair comparison, the batch size and the momentum are set to be 256 and 0.9, and we decrease the learning rate by 0.1 every 30 epochs for all methods. For KFAC and our methods, the frequency of updating the curvature matrix, the initial learning rate and the regularization parameters $\lambda_k$
%are set to be $50$, $0.1$ and $0.8 \times 0.6^{\lfloor
%\mathrm{epoch}/30 \rfloor}$, respectively.
Our codes are implemented in PyTorch. We run ResNet18 and VGG16\_bn experiments on one Tesla V100 GPU and ResNet50 on multiple Tesla V100 GPUs.

\begin{table*}%[t]%[!b]
\centering
\caption{Comparison of SENG on different batch sizes. We terminate the training when the top-1 testing accuracy achieves 75.9\%. BS means batch size. TT means total time. TpE means time per epoch. Scaling efficiency (SE) for each line is $\frac{512}{BS(\cdot)} \times \frac{TT(512)} { TT(\cdot)}$ or $\frac{512}{BS(\cdot)} \times \frac{TpE(512)} { TpE(\cdot)}$. }
%{\color{red} To be completed.}.}
{
%\vspace{ex}
\begin{tabular}{c|cccccc}
\hline \hline
 BS & \#GPUs & \# Epochs & Total Time (TT)  & SE (TT) & Time Per Epoch (TpE) & SE (TpE) \\ \hline
 512 & 4 & 41 &  371.5 min &  1 &542.22 s & 1  \\
1024 & 8 & 41 &  202.4 min &  0.92 & 296.12 s & 0.92 \\
 2048 & 16 &41 &  103.2 min &  0.90 &151.07 s & 0.90\\
4096 & 32 &41 & \textbf{49.7 min}  &0.93  &72.66 s &  0.93\\
 \hline \hline
\end{tabular}
}
\label{ImageNet-BS}
\vspace{-2ex}
\end{table*}

%For each task, the results of different methods are based on the same environment, i.e., both the same softwares and the same hardwares.
%When the matrix $B_k$ is set to be the trivial subsampled EFIM matrix (\ref{mb-EFIM}), we name it empirical natural gradient (ENG). {\color{red}The description of numerical results wait for the results from Peng Cheng Lab. I will update this afternoon 28/05/2020. }

\subsection{ResNet18 \& VGG16\_bn}

We demonstrate the comparison results of six tasks in this part. The compared methods are well-tuned by a grid search and the details can be found in the Appendix. We terminate the algorithms once their top-1 testing accuracy attains the given baseline. The related statistics are averages of three independent runs and shown in Table \ref{statistics-resnetvgg}.
We further show the the changes of testing accuracy and other statistics versus training time in the Appendix.
\vspace{-2ex}
\begin{table}[H]%[t]%[!b]
%\vspace{-2ex}
\scriptsize
\centering
\caption{ Comparison with SGD, ADAM and KFAC on six tasks over three independent runs. ``N'' means the methods can not attain the given accuracy and the attained best testing accuracy is reported next to it. The time are in seconds. }
%{\color{red} To be completed.}.}
{
\begin{tabular}{r|cccc}
\hline \hline
 & \multicolumn{4}{c}{VGG16\_bn} \\ \hline
& SENG &SGD & ADAM & KFAC \\ \hline
CIFAR10 TimeTo92\% & {\textbf{943.12}} & 1034.82 & N/89.4\% & 6009.86\\
Time/Epoch & 17.49 & 14.58 & 16.13 & 113.37\\ \hline
CIFAR100 TimeTo70\% & {\textbf{1088.16}} & 1168.49 & N/63.0\% & 5652.81\\
Time/Epoch & 18.14 & 15.16 & 16.50 & 113.13\\ \hline
SVHN TimeTo95\% & {\textbf{515.06}} & N/75.5\% & N/94.97\% & 7321.62\\
Time/Epoch & 24.43 & 20.30 & 22.56 & 162.57\\ \hline
 & \multicolumn{4}{c}{ResNet18} \\ \hline
& SENG &SGD & ADAM & KFAC \\ \hline
CIFAR10 TimeTo94\% & {\textbf{940.72}} & 1083.60 & N/91.3\% & 1040.24\\
Time/Epoch & 16.48 & 15.04 & 15.39 & 19.67\\ \hline
CIFAR100 TimeTo76\% &{\textbf{ 952.26}} & N/75.0\% & N/68.4\% & 1001.44\\
Time/Epoch & 16.70 & 14.99 & 15.14 & 19.27\\ \hline
SVHN TimeTo96\% & {\textbf{685.29}} & 1091.78 & N/94.7\% & 1103.17\\
Time/Epoch & 22.90 & 19.88 & 20.21 & 25.65\\ \hline \hline
\end{tabular}
}
\label{statistics-resnetvgg}
\vspace{-4ex}
\end{table}

The performance of SENG is the best in all six tasks.
Compared with first-order type methods, our SENG method takes fewer steps without too much overhead. The advantage of SENG in terms of the computational time over KFAC is also remarkable. For example, in CIFAR10 dataset with VGG16\_bn, the time per epoch of SENG is 18 seconds while that of KFAC is 113.37 seconds. There exists large fully linear layers in the end of VGG16\_bn, so that the cost of matrix inversions in the KFAC method dominates. By constrast, since the size of every matrix to be inverted in SENG is equal to the batch size, its computational time can be controlled.

\begin{table}%[t]%[!b]
%\scriptsize
\centering
\caption{Detailed Statistics on ResNet50 on ImageNet-1k when the top-1 testing accuracy achieves 75.9\%. The numbers that follows SENG and KFAC are the number of matrix update frequency.}
%{\color{red} To be completed.}.}
{
\vspace{0.5ex}
\begin{tabular}{c|ccc}
\hline \hline
&\# Epochs & Total Time  &Time Per Epoch \\ \hline
   SENG-800 & { \textbf{41}} &{  \textbf{27, 190 s}} & \textbf{ 663.17 s} \\ \hline
    SGD & 76 & 43, 707 s & { \textbf{575.09 s}} \\ \hline
  SENG-200 & 41 & 31, 224 s&761.56 s \\ \hline
  KFAC-800 & 42  & 29, 204 s &712.29 s\\ \hline
  KFAC-200 & 42  & 42, 307 s  & 1007.31 s\\ \hline
  \hline
\end{tabular}
}
\label{ImageNet-stat}
\vspace{-4ex}
\end{table}
\subsection{ResNet50 on ImageNet-1k}
The training of ResNet50 on ImageNet-1k dataset is one of the base experiments in MLPerf \cite{mattson2019mlperf}.
%The numerical performance on the task is an important criterion to evaluate the new proposed algorithms.
We compare SENG with KFAC and SGD, and terminate the training process of all the three methods once the Top-1 testing accuracy equals or exceeds 75.9\% as in the MLPerf requirement. The comparison with ADAM is not reported because it does not perform well in our numerical results.

We report the testing accuracy, training accuracy and training loss versus training time in Figure \ref{imagenet-time}. The batch size is chosen to be $256$ for all the three methods and does not change during the training process.  As shown in Figure \ref{imagenet-time}, SENG performs best in the total training time and only takes 41 epochs to get 75.9\% Top-1 testing accuracy. Detailed statistics can be found in Table \ref{ImageNet-stat}. We can see the time per epoch of SENG is close to that of SGD while the number of epoch is much smaller. SENG is faster than KFAC using the same number of matrix update frequency because SENG does not require expensive matrix inversions. Note that the SGD reported here uses the cosine learning rate and its Top-1 testing accuracy can exceed 75.9\% within 76 epochs, which is a well-tuned version. Detailed tuning strategies for SENG and KFAC can be found in the Appendix.
%The performance of SENG can be better if a suitable tuning strategy can be found and the implementation is further optimized.

\subsection{Scaling Efficiency}
It has been widely known that large batch training will lead to performance degradation. In this part, we investigate the scalability  of our proposed SENG. We start by running the codes with 4, 8 and 16 GPUs on one node, and then run the distributed version on 32 GPUs across 2 nodes. The results are shown in Table \ref{ImageNet-BS}. We can see that the GPU scaling efficiency is over 90\% in terms of both total time and  time per epoch.  The results show the great potential of SENG in the distributed large-batch training in practice.

It is noticed that SENG with the batch size 4, 096 can attain top-1 testing accuracy 75.9\% within 41 epochs and takes 49.7 minutes. The number of epoch is the same on all batch sizes, which illustrates the effectiveness of SENG with large batch.
Compared with the results reported in \cite{9123671} where SP-NGD with a batch size 4, 096 but on 128 Tesla V100 GPUs takes 32.5 minutes to top-1 testing accuracy 74.8\%, our results are also reasonable. With more computational resources, it is expected that SENG can attain the given testing accuracy within less training time. The detailed hyper-parameters are reported in the Appendix.
\section{Conclusion}
In this paper, we develop efficient sketching techniques for the empirical natural gradient method for deep learning problems.  Since the EFIM is usually low-rank, the corresponding direction is actually a linear combination of the subsampled gradients based on the SMW formula. For layers whose number of parameters is not huge,  we construct a much smaller least squares problem by sketching on the subsampled gradients. Otherwise, the quantities in the SMW formula is computed by using the matrix-matrix  representation of the gradients. We first approximate them by low-rank matrices, then use sketching methods to compute the expensive parts. Global convergence is guaranteed under some standard assumptions and a fast linear convergence is analyzed in the NTK regime. Our numerical results show that the empirical natural gradient method with randomized techniques can be quite competitive with the state-of-the-art methods such as SGD and KFAC. Experiments on the distributed large-batch training illustrate that the scaling efficiency of SENG is quite promising.

\textbf{Acknowledgments}  M. Yang, D. Xu and Z. Wen are supported in part by Key-Area Research and Development Program of Guangdong Province (No.2019B121204008),  the NSFC grants 11831002 and Beijing Academy of Artificial Intelligence.
\nocite{langley00}

\bibliography{reference}

\begin{thebibliography}{24}
\providecommand{\natexlab}[1]{#1}
\providecommand{\url}[1]{\texttt{#1}}
\expandafter\ifx\csname urlstyle\endcsname\relax
  \providecommand{\doi}[1]{doi: #1}\else
  \providecommand{\doi}{doi: \begingroup \urlstyle{rm}\Url}\fi

\bibitem[Amari(1997)]{amari1997neural}
Amari, S.-i.
\newblock Neural learning in structured parameter spaces-natural riemannian
  gradient.
\newblock In \emph{Advances in neural information processing systems}, pp.\
  127--133, 1997.

\bibitem[Bernacchia et~al.(2018)Bernacchia, Lengyel, and
  Hennequin]{Bernacchia2018ExactNG}
Bernacchia, A., Lengyel, M., and Hennequin, G.
\newblock Exact natural gradient in deep linear networks and its application to
  the nonlinear case.
\newblock In \emph{NeurIPS}, 2018.

\bibitem[Botev et~al.(2017)Botev, Ritter, and Barber]{practicalGN}
Botev, A., Ritter, H., and Barber, D.
\newblock Practical {G}auss-{N}ewton optimisation for deep learning.
\newblock In \emph{International Conference on Machine Learning}, pp.\
  557--565, 2017.

\bibitem[Byrd et~al.(2016)Byrd, Hansen, Nocedal, and Singer]{ByrHanNocSin16}
Byrd, R.~H., Hansen, S.~L., Nocedal, J., and Singer, Y.
\newblock A stochastic quasi-{N}ewton method for large-scale optimization.
\newblock \emph{SIAM Journal on Optimization}, 26\penalty0 (2):\penalty0
  1008--1031, 2016.

\bibitem[Cai et~al.(2019)Cai, Gao, Hou, Chen, Wang, He, Zhang, and
  Wang]{Cai2019AGM}
Cai, T., Gao, R., Hou, J., Chen, S., Wang, D., He, D., Zhang, Z., and Wang, L.
\newblock A gram-gauss-newton method learning overparameterized deep neural
  networks for regression problems.
\newblock \emph{ArXiv}, abs/1905.11675, 2019.

\bibitem[Goyal et~al.(2017)Goyal, Doll{\'a}r, Girshick, Noordhuis, Wesolowski,
  Kyrola, Tulloch, Jia, and He]{goyal2017accurate}
Goyal, P., Doll{\'a}r, P., Girshick, R., Noordhuis, P., Wesolowski, L., Kyrola,
  A., Tulloch, A., Jia, Y., and He, K.
\newblock Accurate, large minibatch {SGD}: Training {I}mage{N}et in 1 hour.
\newblock ArXiv:1706.02677, 2017.

\bibitem[Hazan et~al.(2007)Hazan, Agarwal, and Kale]{hazan2007logarithmic}
Hazan, E., Agarwal, A., and Kale, S.
\newblock Logarithmic regret algorithms for online convex optimization.
\newblock \emph{Machine Learning}, 69\penalty0 (2-3):\penalty0 169--192, 2007.

\bibitem[Karakida \& Osawa(2020)Karakida and
  Osawa]{Karakida2020UnderstandingAF}
Karakida, R. and Osawa, K.
\newblock Understanding approximate fisher information for fast convergence of
  natural gradient descent in wide neural networks.
\newblock 2020.

\bibitem[Keskar et~al.(2017)Keskar, Mudigere, Nocedal, Smelyanskiy, and
  Tang]{KeskarMNST17}
Keskar, N.~S., Mudigere, D., Nocedal, J., Smelyanskiy, M., and Tang, P. T.~P.
\newblock On large-batch training for deep learning: Generalization gap and
  sharp minima.
\newblock In \emph{5th International Conference on Learning Representations,
  {ICLR} 2017, Toulon, France, April 24-26, 2017, Conference Track
  Proceedings}. OpenReview.net, 2017.
\newblock URL \url{https://openreview.net/forum?id=H1oyRlYgg}.

\bibitem[Kingma \& Ba(2014)Kingma and Ba]{kingma2014adam}
Kingma, D.~P. and Ba, J.
\newblock Adam: A method for stochastic optimization.
\newblock \emph{arXiv preprint arXiv:1412.6980}, 2014.

\bibitem[Martens(2020)]{martens2020new}
Martens, J.
\newblock New insights and perspectives on the natural gradient method.
\newblock \emph{Journal of Machine Learning Research}, 21\penalty0
  (146):\penalty0 1--76, 2020.

\bibitem[Martens \& Grosse(2015)Martens and Grosse]{martens2015optimizing}
Martens, J. and Grosse, R.
\newblock Optimizing neural networks with kronecker-factored approximate
  curvature.
\newblock In \emph{International conference on machine learning}, pp.\
  2408--2417, 2015.

\bibitem[Mattson et~al.(2019)Mattson, Cheng, Coleman, Diamos, Micikevicius,
  Patterson, Tang, Wei, Bailis, Bittorf, et~al.]{mattson2019mlperf}
Mattson, P., Cheng, C., Coleman, C., Diamos, G., Micikevicius, P., Patterson,
  D., Tang, H., Wei, G.-Y., Bailis, P., Bittorf, V., et~al.
\newblock Mlperf training benchmark.
\newblock \emph{arXiv preprint arXiv:1910.01500}, 2019.

\bibitem[{Osawa} et~al.(2020){Osawa}, {Tsuji}, {Ueno}, {Naruse}, {Foo}, and
  {Yokota}]{9123671}
{Osawa}, K., {Tsuji}, Y., {Ueno}, Y., {Naruse}, A., {Foo}, C.~S., and {Yokota},
  R.
\newblock Scalable and practical natural gradient for large-scale deep
  learning.
\newblock \emph{IEEE Transactions on Pattern Analysis and Machine
  Intelligence}, pp.\  1--1, 2020.
\newblock \doi{10.1109/TPAMI.2020.3004354}.

\bibitem[Ren \& Goldfarb(2019)Ren and Goldfarb]{ren2019efficient}
Ren, Y. and Goldfarb, D.
\newblock Efficient subsampled gauss-newton and natural gradient methods for
  training neural networks.
\newblock \emph{arXiv preprint arXiv:1906.02353}, 2019.

\bibitem[Robbins \& Monro(1951)Robbins and Monro]{RobMon51}
Robbins, H. and Monro, S.
\newblock A stochastic approximation method.
\newblock \emph{Ann. Math. Stat.}, 22:\penalty0 400--407, 1951.
\newblock ISSN 0003-4851.

\bibitem[Roux et~al.(2008)Roux, Manzagol, and Bengio]{roux2008topmoumoute}
Roux, N.~L., Manzagol, P.-A., and Bengio, Y.
\newblock Topmoumoute online natural gradient algorithm.
\newblock In \emph{Advances in neural information processing systems}, pp.\
  849--856, 2008.

\bibitem[Shallue et~al.(2019)Shallue, Lee, Antognini, Sohl-Dickstein, Frostig,
  and Dahl]{shallue2019measuring}
Shallue, C.~J., Lee, J., Antognini, J., Sohl-Dickstein, J., Frostig, R., and
  Dahl, G.~E.
\newblock Measuring the effects of data parallelism on neural network training.
\newblock \emph{Journal of Machine Learning Research}, 20:\penalty0 1--49,
  2019.

\bibitem[Sun(2019)]{sun2019optimization}
Sun, R.
\newblock Optimization for deep learning: theory and algorithms.
\newblock \emph{arXiv preprint arXiv:1912.08957}, 2019.

\bibitem[Wang et~al.(2016)Wang, Luo, and Zhang]{10.5555/2946645.2946694}
Wang, S., Luo, L., and Zhang, Z.
\newblock Spsd matrix approximation vis column selection: Theories, algorithms,
  and extensions.
\newblock \emph{J. Mach. Learn. Res.}, 17\penalty0 (1):\penalty0 1697–1745,
  January 2016.
\newblock ISSN 1532-4435.

\bibitem[Wang et~al.(2017{\natexlab{a}})Wang, Gittens, and
  Mahoney]{10.5555/3122009.3242075}
Wang, S., Gittens, A., and Mahoney, M.~W.
\newblock Sketched ridge regression: Optimization perspective, statistical
  perspective, and model averaging.
\newblock \emph{J. Mach. Learn. Res.}, 18\penalty0 (1):\penalty0 8039–8088,
  January 2017{\natexlab{a}}.
\newblock ISSN 1532-4435.

\bibitem[Wang et~al.(2017{\natexlab{b}})Wang, Ma, Goldfarb, and
  Liu]{WanMaGolLiu17}
Wang, X., Ma, S., Goldfarb, D., and Liu, W.
\newblock Stochastic {Q}uasi-{N}ewton {M}ethods for {N}onconvex {S}tochastic
  {O}ptimization.
\newblock \emph{SIAM Journal on Optimization}, 27\penalty0 (2):\penalty0
  927--956, 2017{\natexlab{b}}.

\bibitem[Yang et~al.(2019)Yang, Milzarek, Wen, and Zhang]{yangSEQN}
Yang, M., Milzarek, A., Wen, Z., and Zhang, T.
\newblock A stochastic extra-step quasi-newton method for nonsmooth nonconvex
  optimization.
\newblock \emph{ArXiv:1910.09373}, 2019.

\bibitem[Zhang et~al.(2019)Zhang, Martens, and Grosse]{zhang2019fast}
Zhang, G., Martens, J., and Grosse, R.~B.
\newblock Fast convergence of natural gradient descent for over-parameterized
  neural networks.
\newblock In \emph{Advances in Neural Information Processing Systems}, pp.\
  8082--8093, 2019.

\end{thebibliography}
\bibliographystyle{icml2021}

%\iffalse
\appendix
\onecolumn

\section{Implementation Details}
The statistics of the datasets used in Section 6 are listed in Table \ref{SENG:datasets}.
\begin{table}[H]
\centering
\begin{tabular}{ccccc}%{|p{10ex}p{8ex}p{8ex}|}
\toprule
Dataset & \# Training Set & \# Testing Set   \\ %&sp. (\%) \\
\midrule
CIFAR10 	& 50,000		& 10,000  \\ %	& 99.84 \\
CIFAR100	&50,000		& 10,000  \\
SVHN &73,257 & 26,032 \\
ImageNet-1k &1,281,167 & 50,000\\
\bottomrule
\end{tabular}
 \caption{The Datasets Information.}
\label{SENG:datasets}
\end{table}
We use the official implementation of VGG16\_bn, ResNet18 and ResNet50 (also known as ResNet50 v1.5) in PyTorch. The detailed network structures can be found in the websites: \url{https://pytorch.org/docs/stable/_modules/torchvision/models/resnet.html} and \url{https://pytorch.org/docs/stable/_modules/torchvision/models/vgg.html}.

We next describe the tuning schemes of hyper-parameters. The learning rate is very important for performance and we mainly consider the following  two schemes.
\begin{itemize}
\item cosine: Given the max\_epoch and the initial learning rate $\alpha_0$, the learning rate $\alpha_k$ at the $k$-th epoch is changed as:
\[ \alpha_k = 0.001 + 0.5 * (\alpha_0 - 0.001) * (1 + \text{cos(epoch\_k / max\_epoch} * \pi)).\]
\item exp: Given the max\_epoch, decay rate $\mathcal{\widetilde D}$ and the initial learning rate $\alpha_0$, the learning rate $\alpha_k$  at the $k$-th epoch is changed as:
\[ \alpha_k = \alpha_0 * (1 - \text{epoch\_k/max\_epoch)}^{\mathcal{\widetilde D}}.\]
\end{itemize}
\begin{itemize}
\item For ResNet18 and VGG16 on the three datasets, i.e., CIFAR10, CIFAR100 and SVHN:
\begin{itemize}
\item  Adam
\begin{itemize}
\item The initial learning rate is chosen from \{1e-4, 2e-4, 5e-4, 1e-3, 2e-3, 5e-3,1e-2, 2e-2, 5e-2, 1e-1\}.
\item The parameters $\beta_1$ and $\beta_2$ are selected in \{0.9,0.99\} and $\{0.99,0.999\}$, respectively.
\item The weight decay is chosen from \{5e-4, 2e-4,1e-4\}.
\item The perturbation value $\epsilon$ is 1e-8.
\end{itemize}
\item SGD (with momentum): we use the best results from the cosine and exp schemes.

\begin{itemize}
\item For the cosine scheme, the hyper-parameters is tuning as follows:
\begin{itemize}
\item The initial learning rate is from \{1e-3, 5e-3, 1e-2, 5e-2, 1e-1\}.
\item The max\_epoch is tuned from \{85, 90\}.
\item The weight decay is chosen from \{5e-4, 2e-4,1e-4\}.
\item The momentum is set to be 0.9.
\end{itemize}
\item For the exp scheme, the hyper-parameters is tuning as follows:
\begin{itemize}
\item The initial learning rate in from \{1e-3, 5e-3, 1e-2, 5e-2, 1e-1\}.
\item The max\_epoch is tuned from \{70, 75, 80\}.
\item The decay rate is tuned from \{4, 5, 6\}.
\item The weight decay is tuned from \{5e-4, 2e-4,1e-4\}.
\item The momentum is set to be 0.9.
\end{itemize}
\end{itemize}
\item KFAC and SENG use the same grid search strategies as SGD. In addition, the damping parameter for both methods is chosen from \{1.5, 2.0, 2.5\} for VGG16 and from \{ 0.8, 1.0,1.2\} for ResNet18. KFAC updates the covariance matrix to be inverted every 200 iterations. Similarly, SENG updates the matrix $U$ at the same frequency. For convenience of notations, we call them the matrix update frequency for both methods. They are critical for the performance since the related operations are expensive.
\end{itemize}
\item For ResNet50 on ImageNet-1k, we use the linear warmup strategy  \cite{goyal2017accurate} in the first 5 epochs for SGD, KFAC and SENG, then use the cosine or exp learning rate strategy.
\begin{itemize}
\item SGD refers to a well-tuned cosine learning rate strategy in the website \footnote{\url{https://gitee.com/mindspore/mindspore/blob/r0.7/model_zoo/official/cv/resnet/src/lr_generator.py}} and the result reported here achieves top-1 testing accuracy 75.9\% within 76 epochs. This is better than the result in terms of epoch in \cite{goyal2017accurate} where they use the diminishing learning rate strategy and need nearly 90 epochs.
\item KFAC uses the exp strategy after the first 5 epochs. The initial learning rate is from \{0.05, 0.1, 0.15\} and the damping is from  \{0.05, 0.1, 0.15\} and we report the best results among them.
\item SENG uses the exp strategy after the first 5 epochs. The initial learning rate is 0.145 and the damping is 0.17.
\item Note that both SENG and KFAC are not sensitive to damping and initial learning rate. The weight decay for SENG and KFAC are chosen the best from \{5e-4, 3e-4, 2e-4, 1e-4\}. We also consider the cosine strategy for KFAC and SENG, but it does not work well.
\end{itemize}
\item For large-batch training, the detailed hyper-parameters for each batch-size are listed in Table \ref{large-bs-table}.
\begin{table}[H]
\centering
\begin{tabular}{cccccc}%{|p{10ex}p{8ex}p{8ex}|}
\toprule
Batch Size & $\alpha_{\text{warmup}}$ & $\alpha_0 $ & decay rate &max\_epoch & damping $\lambda$ \\ %&sp. (\%) \\
\midrule
512 	& 0.01		& 0.3&6  &60 & 0.17$\cdot (0.9)^{epoch/10}$  \\ %	& 99.84 \\
1, 024	& 0.01		& 0.6&6&60 &0.17$\cdot (0.9)^{epoch/10}$ \\
2, 048 &0.2&1.2 &6&60&0.3$\cdot (0.7)^{epoch/10}$\\
4, 096 &0.2&2.2 &5&55&0.3$\cdot (0.7)^{epoch/10}$\\
\bottomrule
\end{tabular}
 \caption{Detailed hyper-parameters for different batch sizes.}
\label{large-bs-table}
\end{table}

\end{itemize}

\section{Further results in section 6}
We give the other statistics on the six tasks in Figure \ref{shallownet}.
The comparison results with KFAC and SGD in section 6.2 with respect to other criteria are listed in Figure \ref{imagenet-epoch}. The performance difference of SENG and KFAC in term of epoch is not significant since the matrix update frequency does not have a strong effect on the performance and the the changes of learning rate of both methods are very similar. The variants of SENG with different matrix update frequency are shown in Figure \ref{imagenet-freq-time}. The detailed results of SENG are reported in Table \ref{ImageNet-freq-stat}.

\begin{table}[ht]%[t]%[!b]
\centering
\caption{Statistics of different matrix update frequency variants of SENG on ResNet50/ImageNet-1k when the top-1 Testing Accuracy achieves 75.9\%.}
%{\color{red} To be completed.}.}
{
\vspace{2ex}
\begin{tabular}{c|cccc}
\hline
 Frequency (SENG) & 100 &200 & 500 & 800 \\ \hline
\# Epoch & { \textbf{40}}& 41& 42 & 41\\
Total Time  & 38, 586 s  & 31, 224 s & 29, 332 s & \textbf{27, 190 s} \\
 Time Per Epoch & 964.65 s & 761.56 s & 698.38 s  & \textbf{663.17 s} \\ \hline
\end{tabular}
}
\label{ImageNet-freq-stat}
\vspace{2ex}
\end{table}
\begin{figure}[H]
\caption{Numerical Comparison on six tasks in section 6.1.
%Upper Three: The numerical comparison with KFAC and SGD. Lower Three: The performance between SENG variants with different matrix updating frequency. The number behind SENG is the number of frequency.
}
\centering
\vspace{0.5ex}
\begin{tabular}{ccc}
\includegraphics[width=0.3\textwidth]{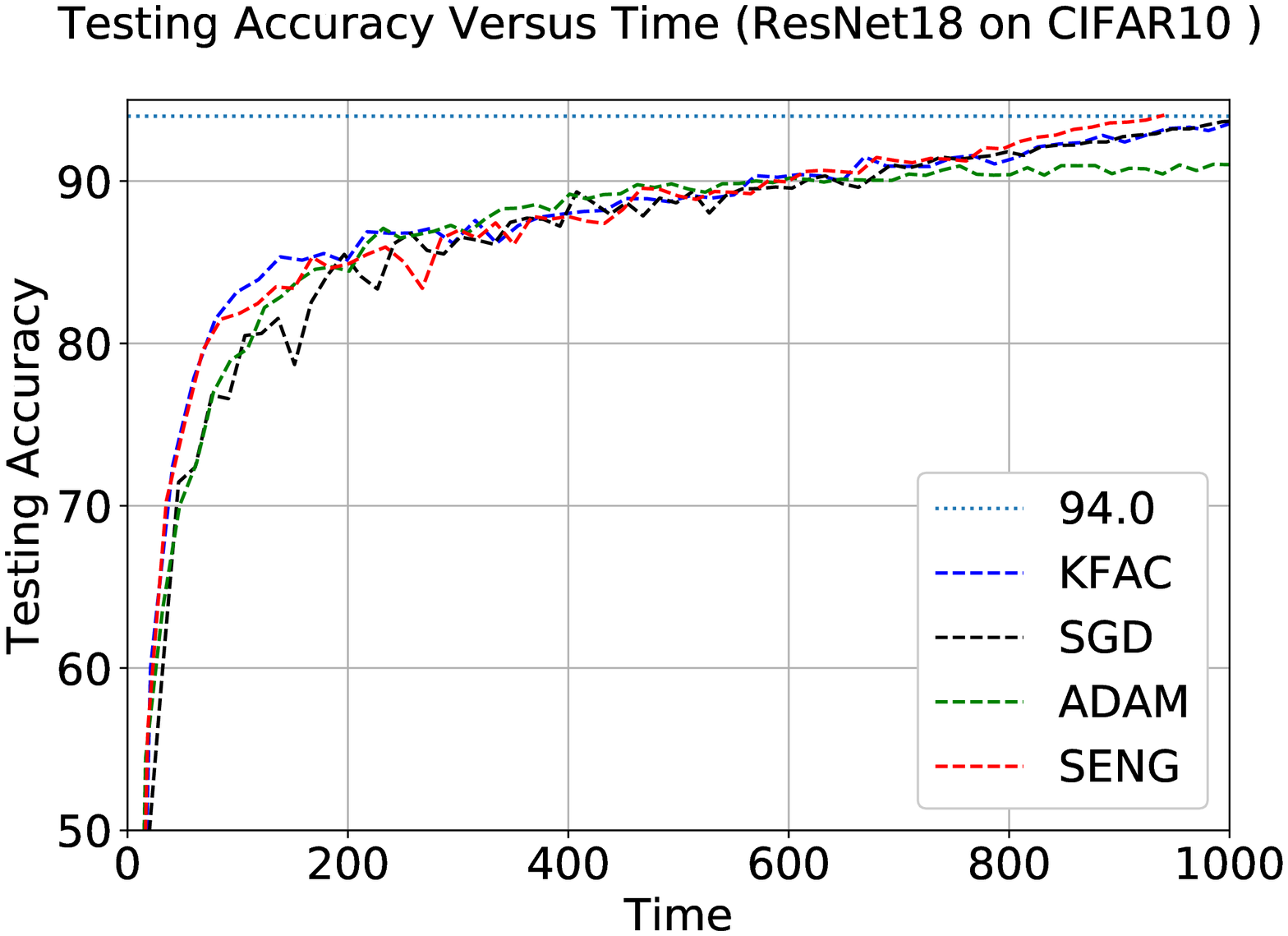}
\includegraphics[width=0.3\textwidth]{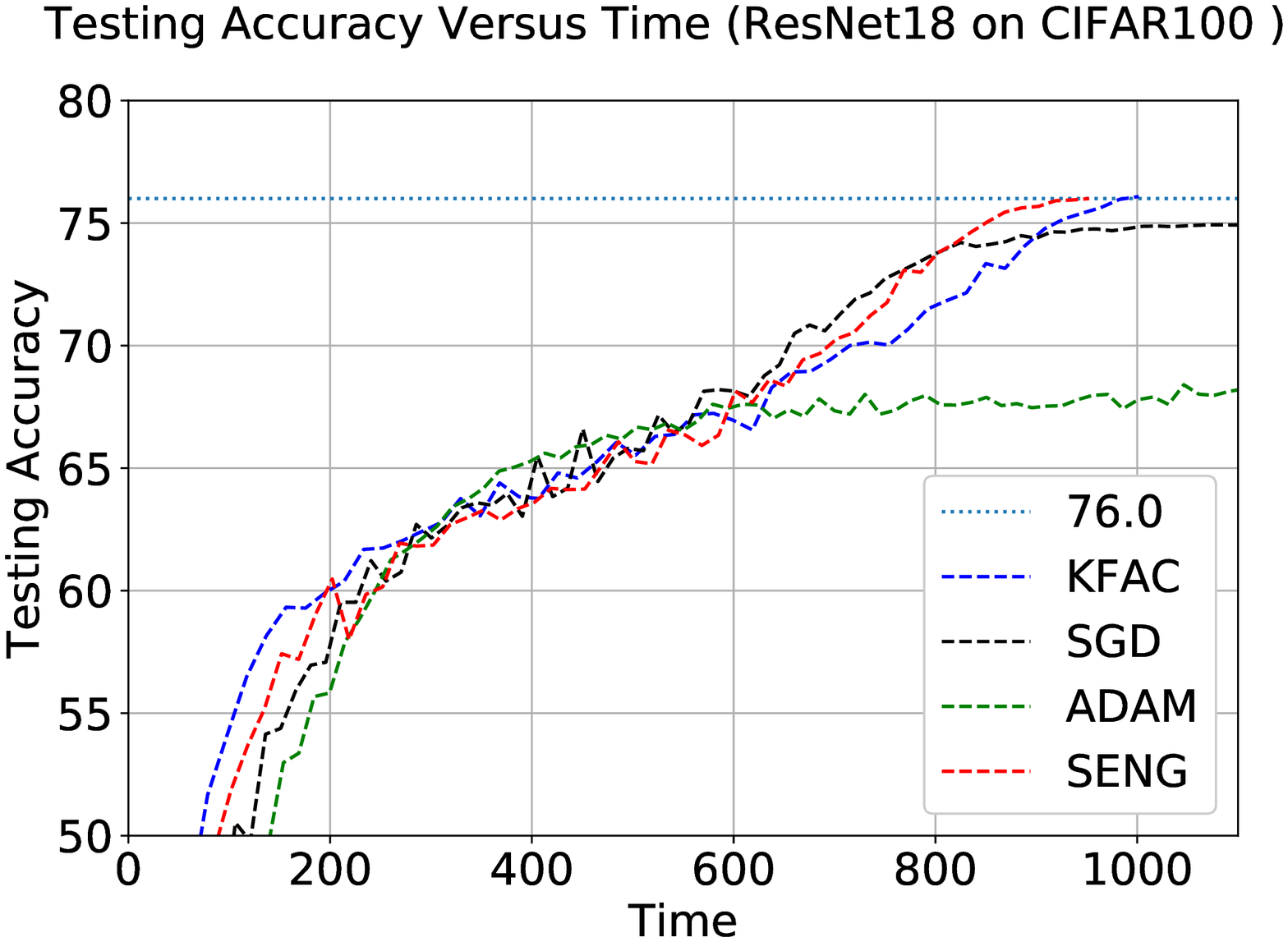}
\includegraphics[width=0.3\textwidth]{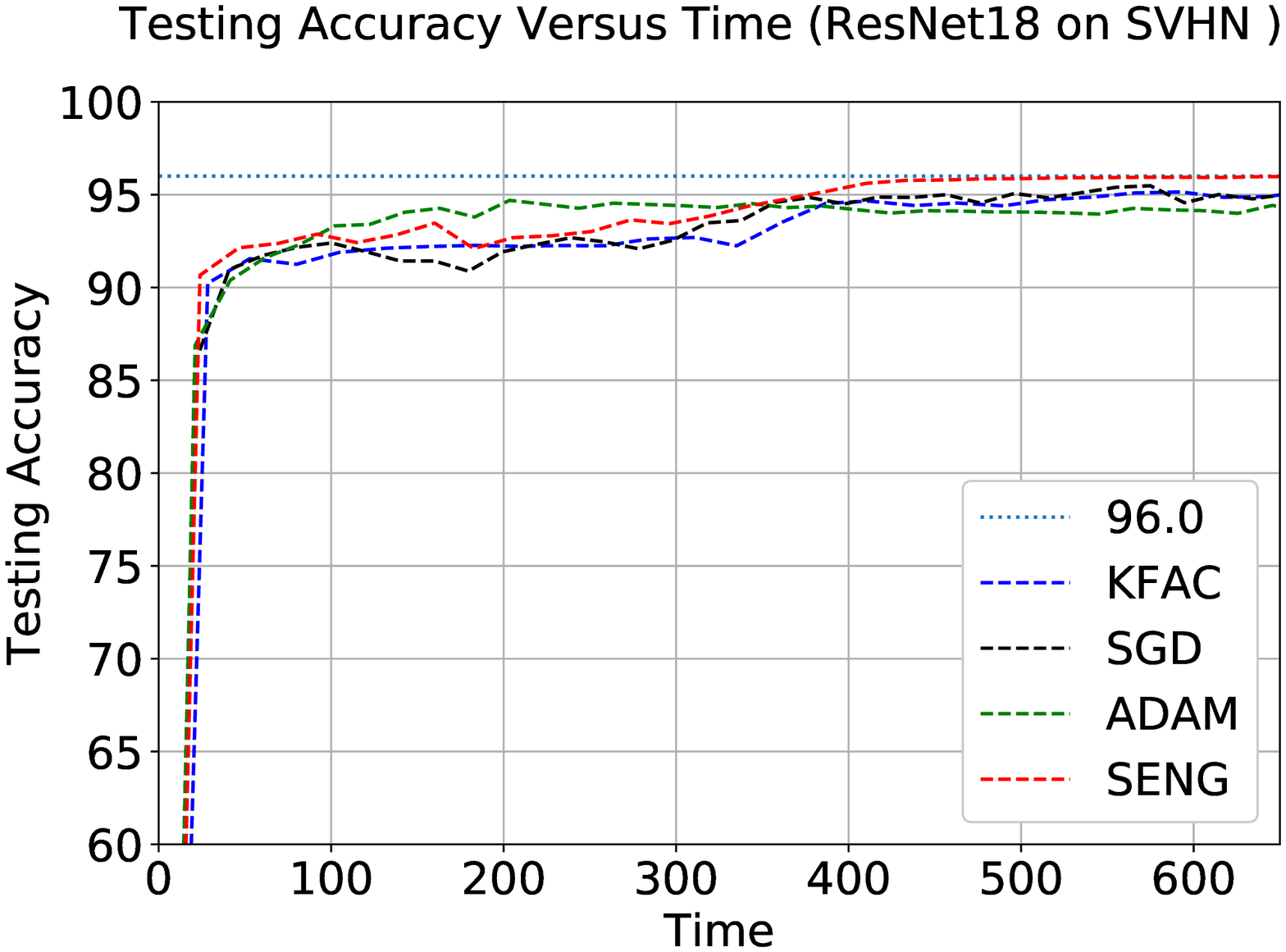}\\
\includegraphics[width=0.3\textwidth]{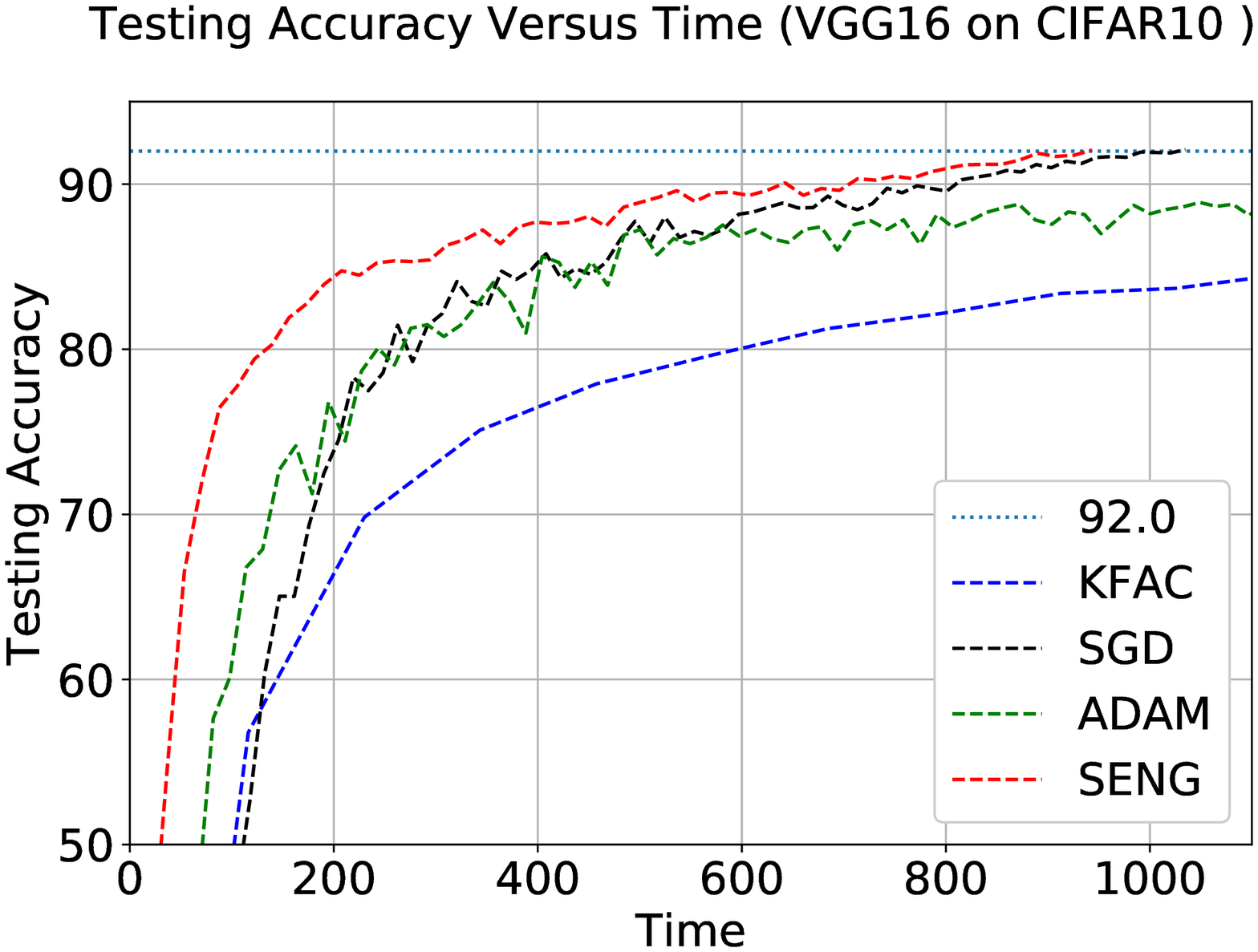}
\includegraphics[width=0.3\textwidth]{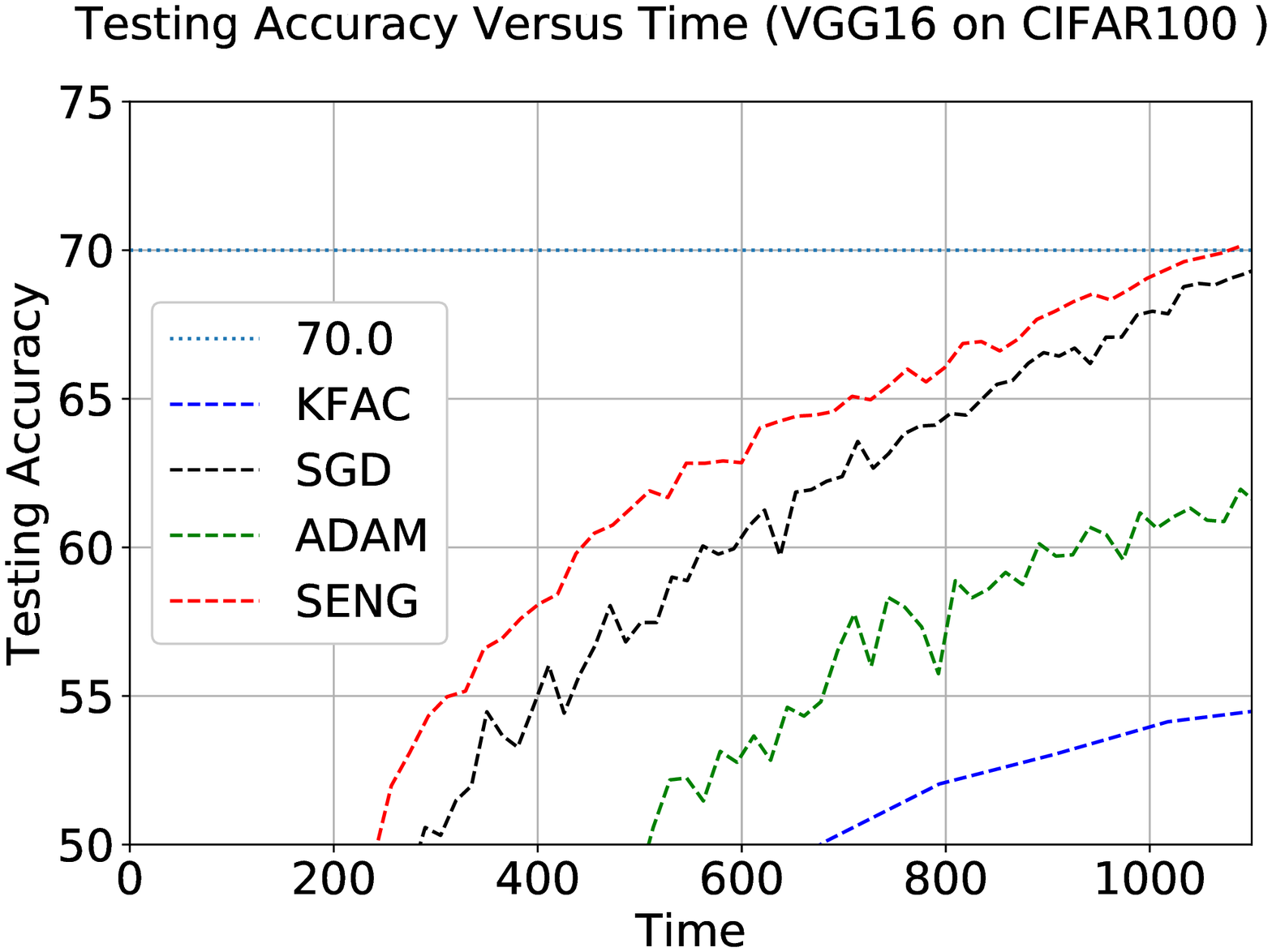}
\includegraphics[width=0.3\textwidth]{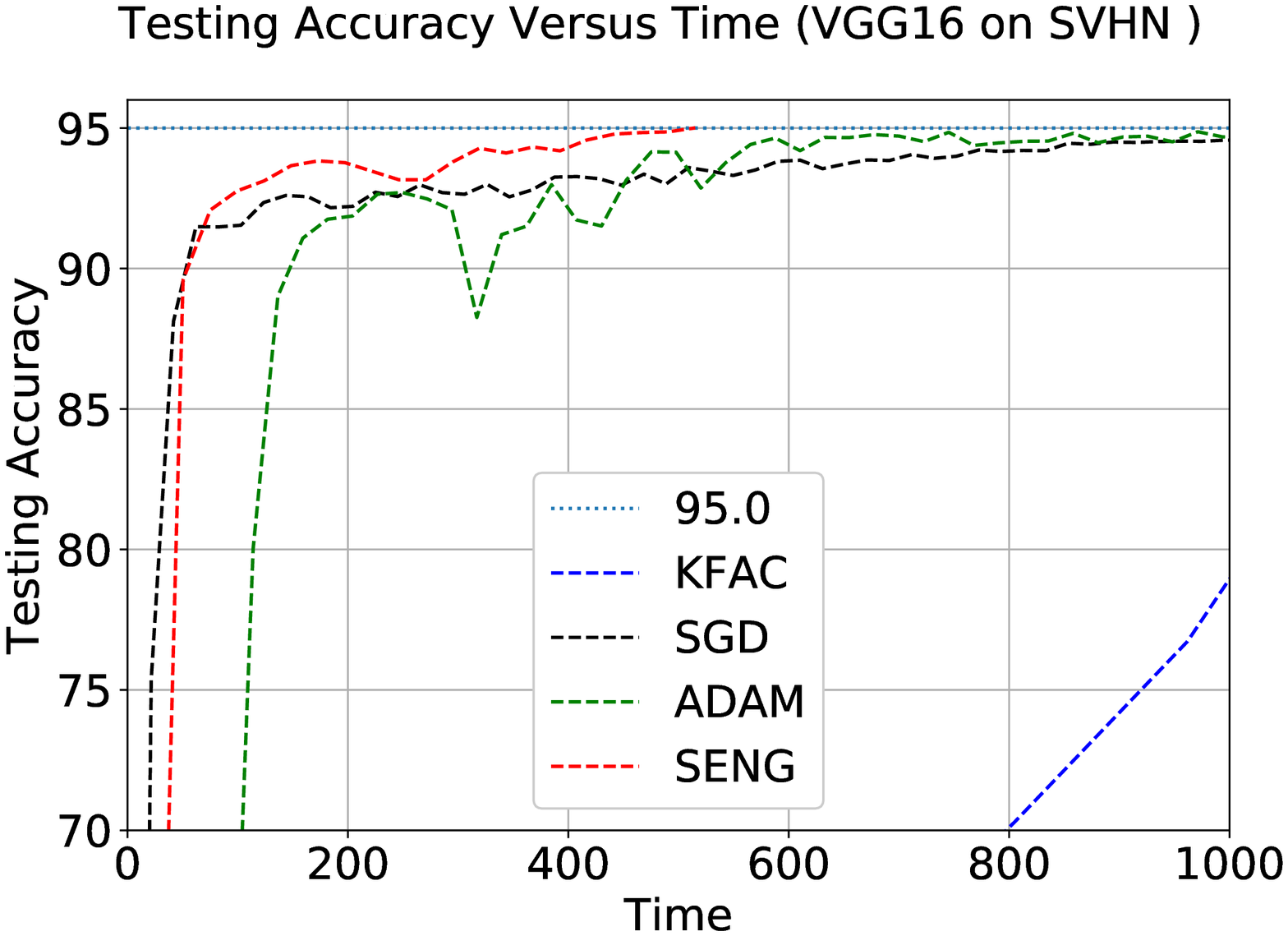}\\
\includegraphics[width=0.3\textwidth]{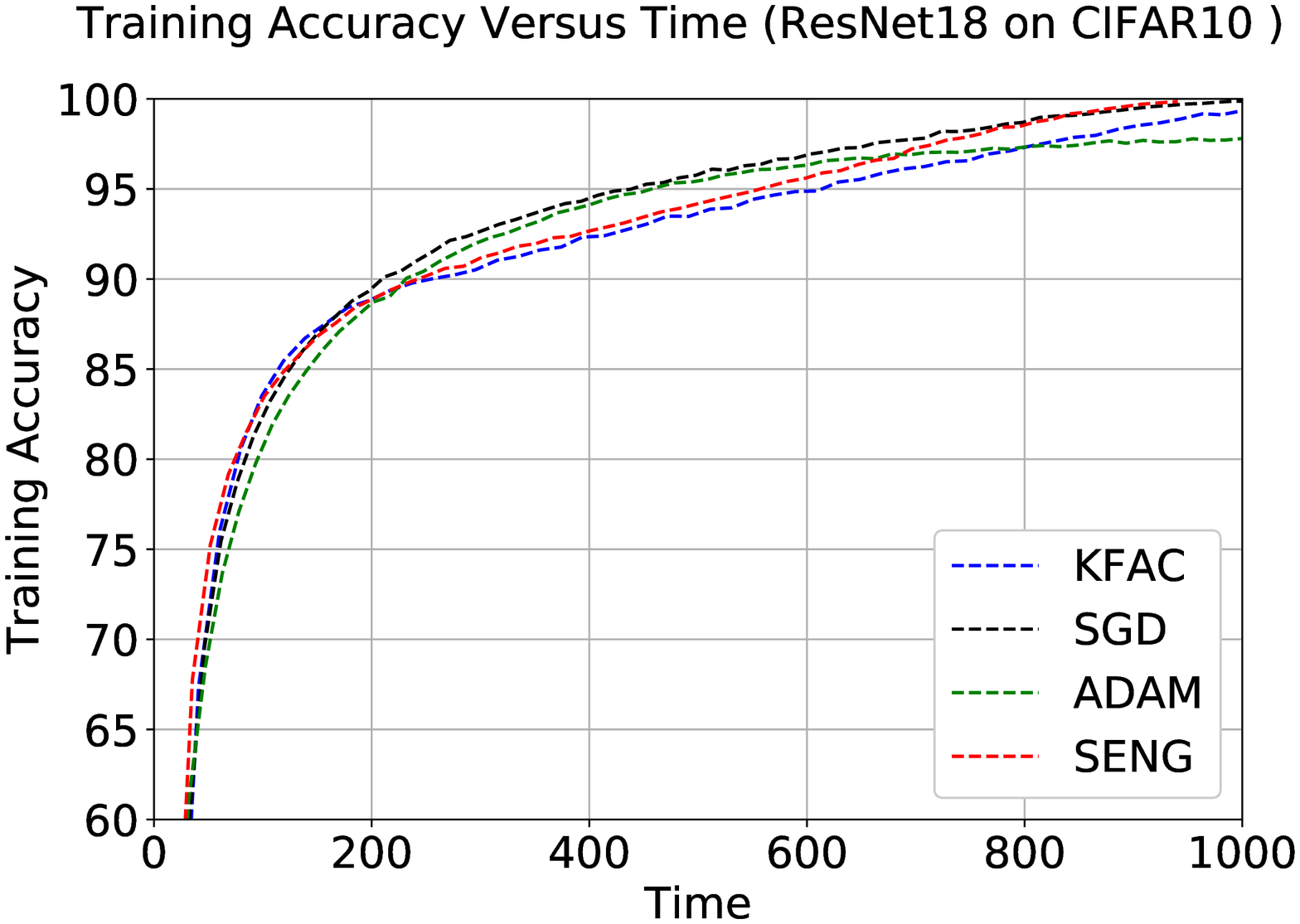}
\includegraphics[width=0.3\textwidth]{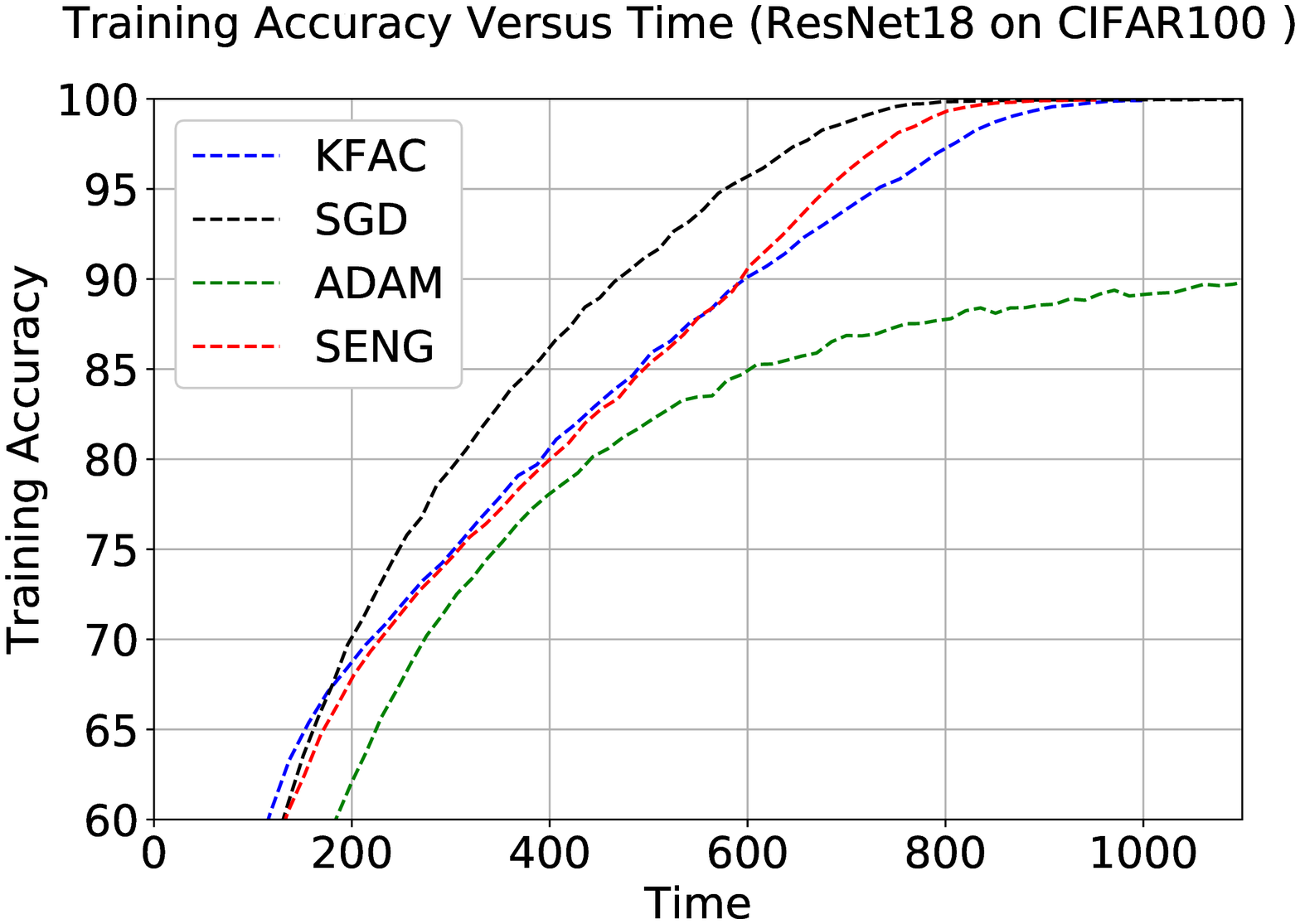}
\includegraphics[width=0.3\textwidth]{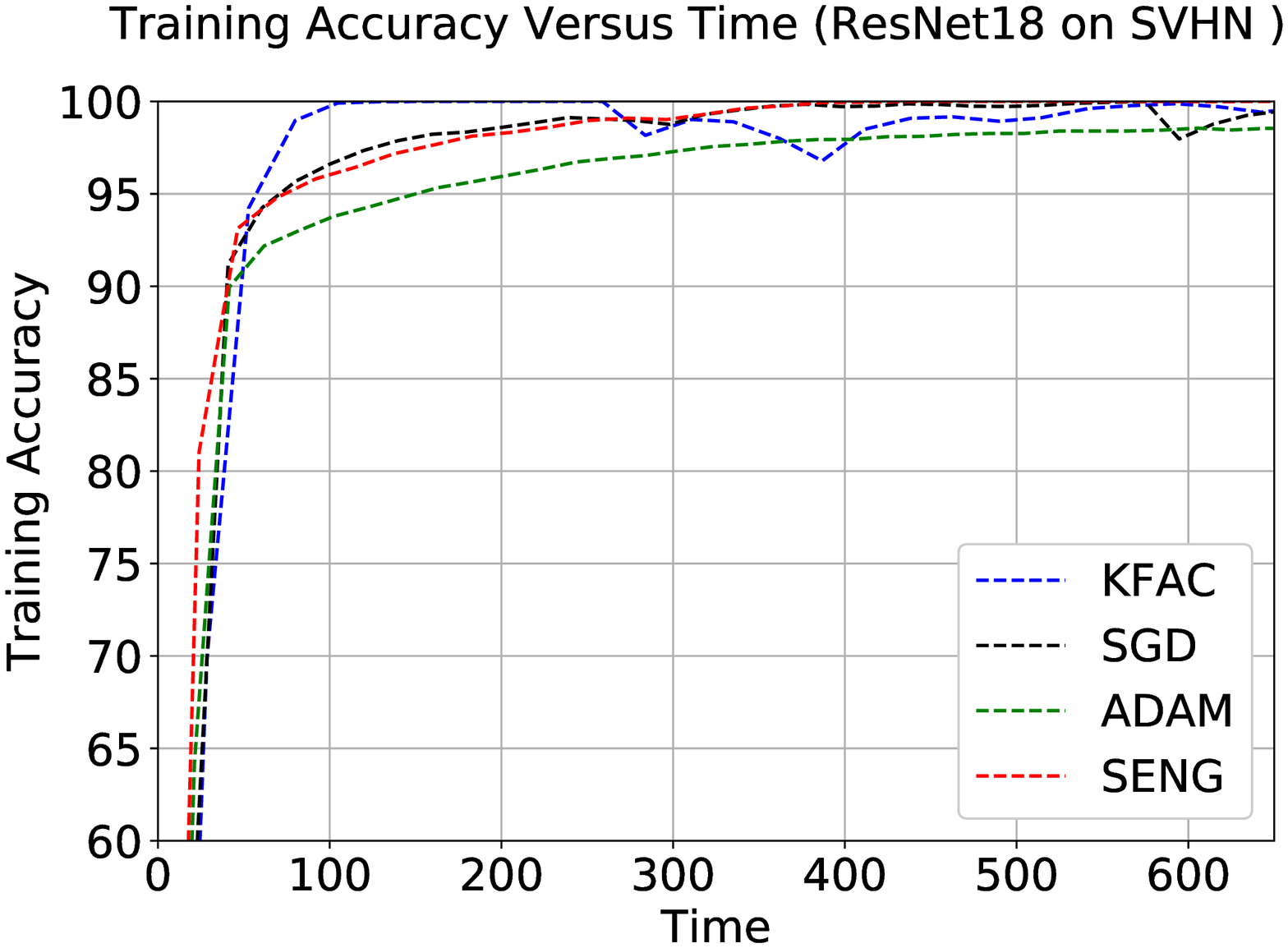}\\
\includegraphics[width=0.3\textwidth]{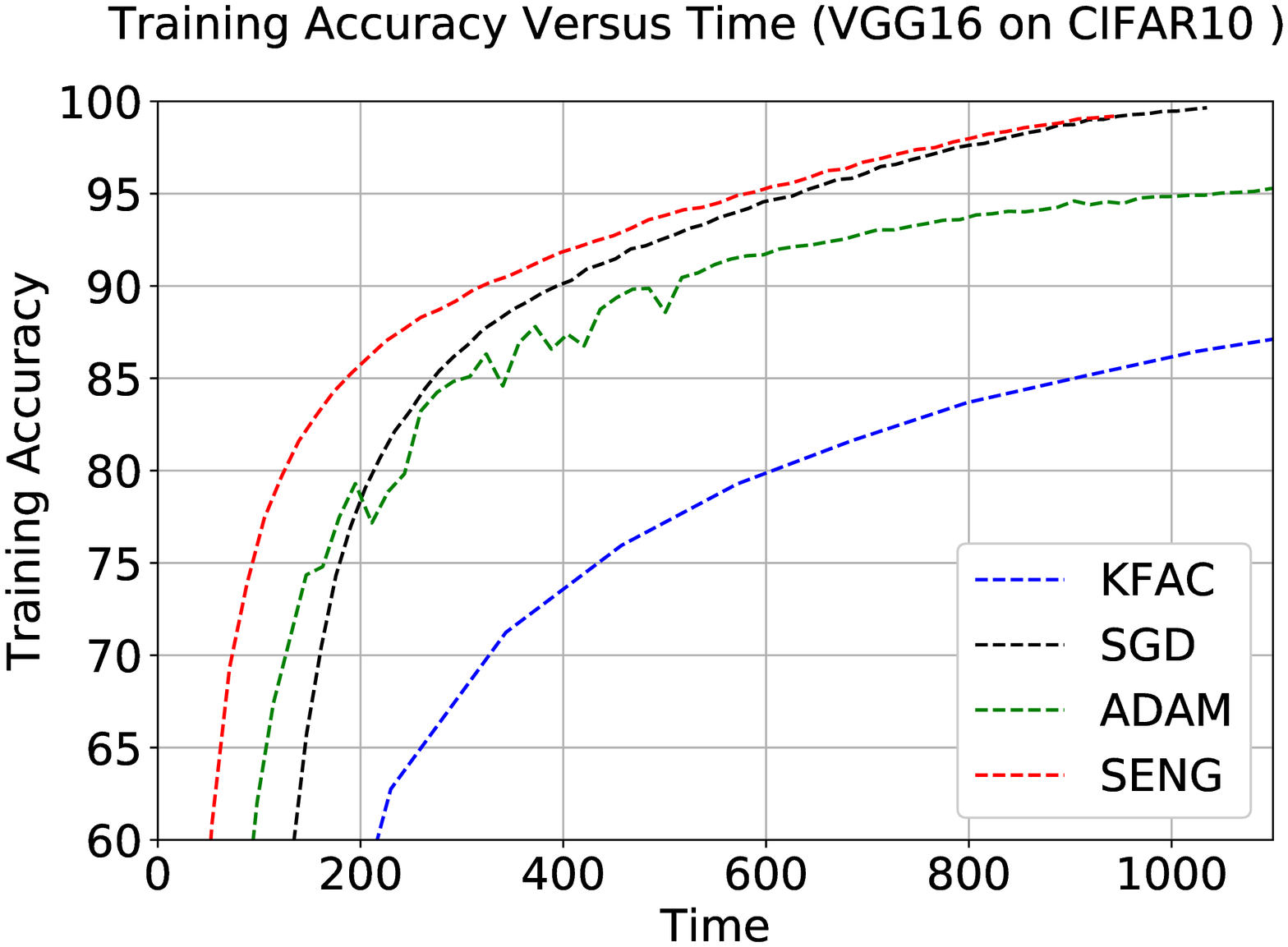}
\includegraphics[width=0.3\textwidth]{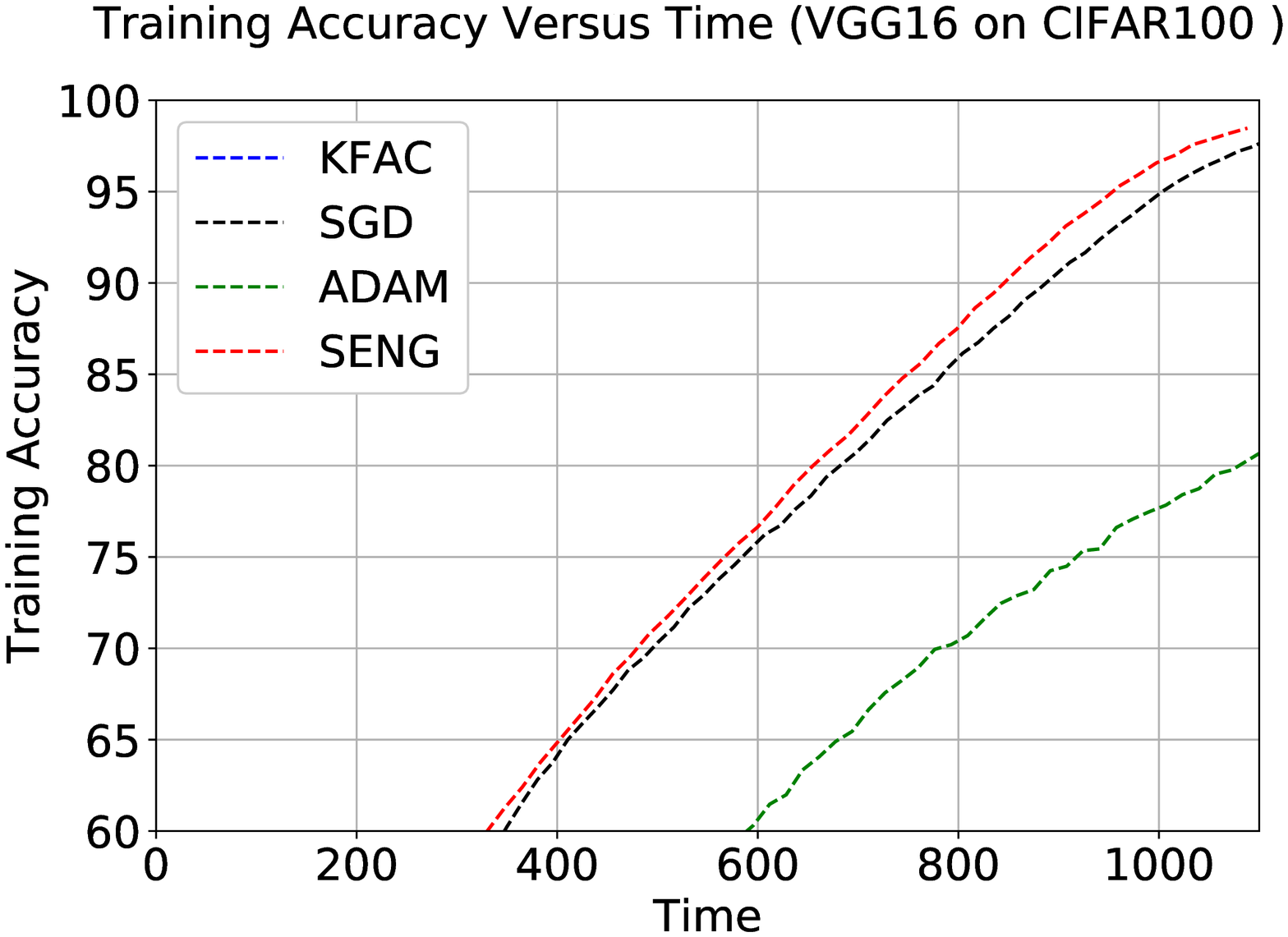}
\includegraphics[width=0.3\textwidth]{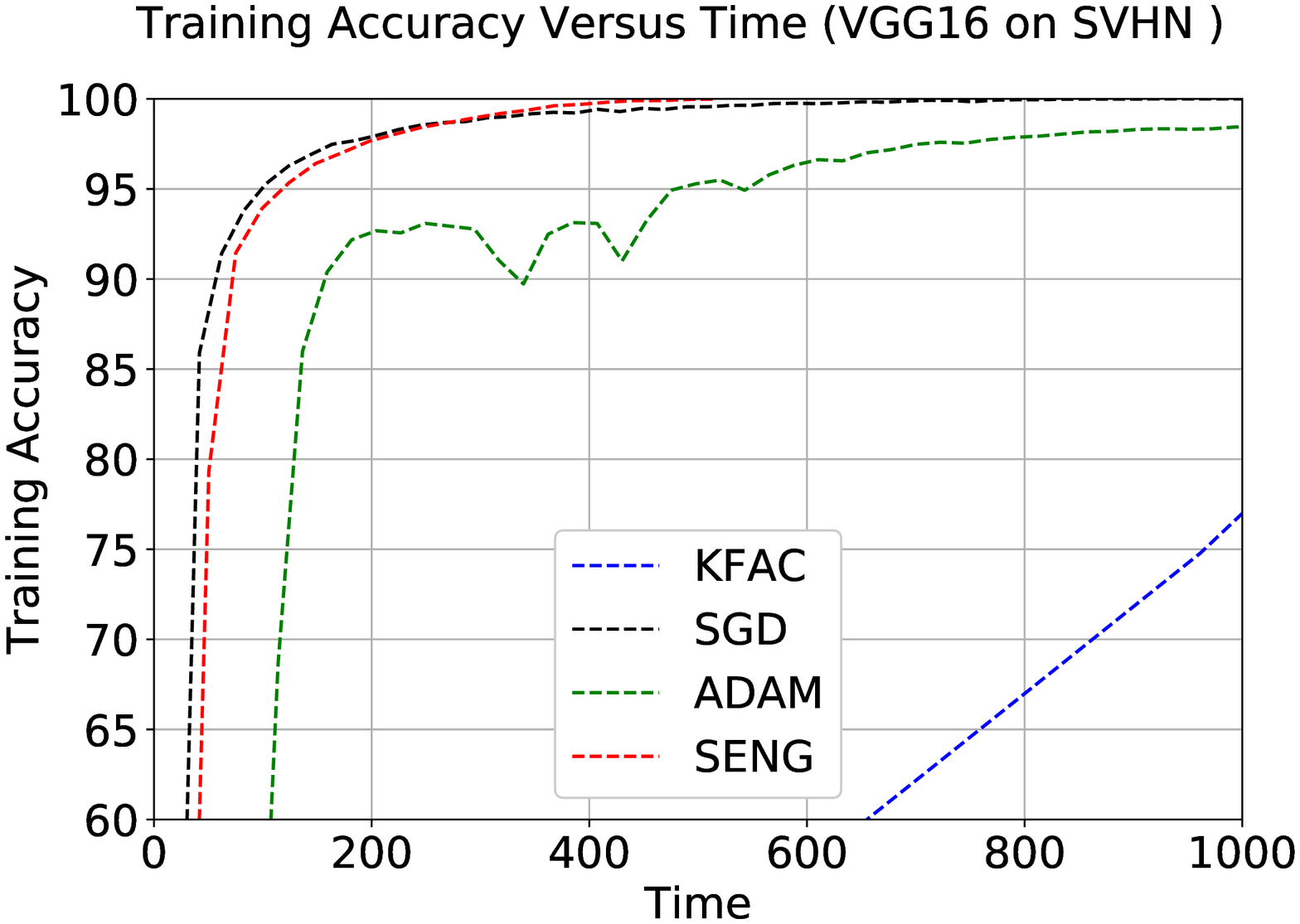}\\
\includegraphics[width=0.3\textwidth]{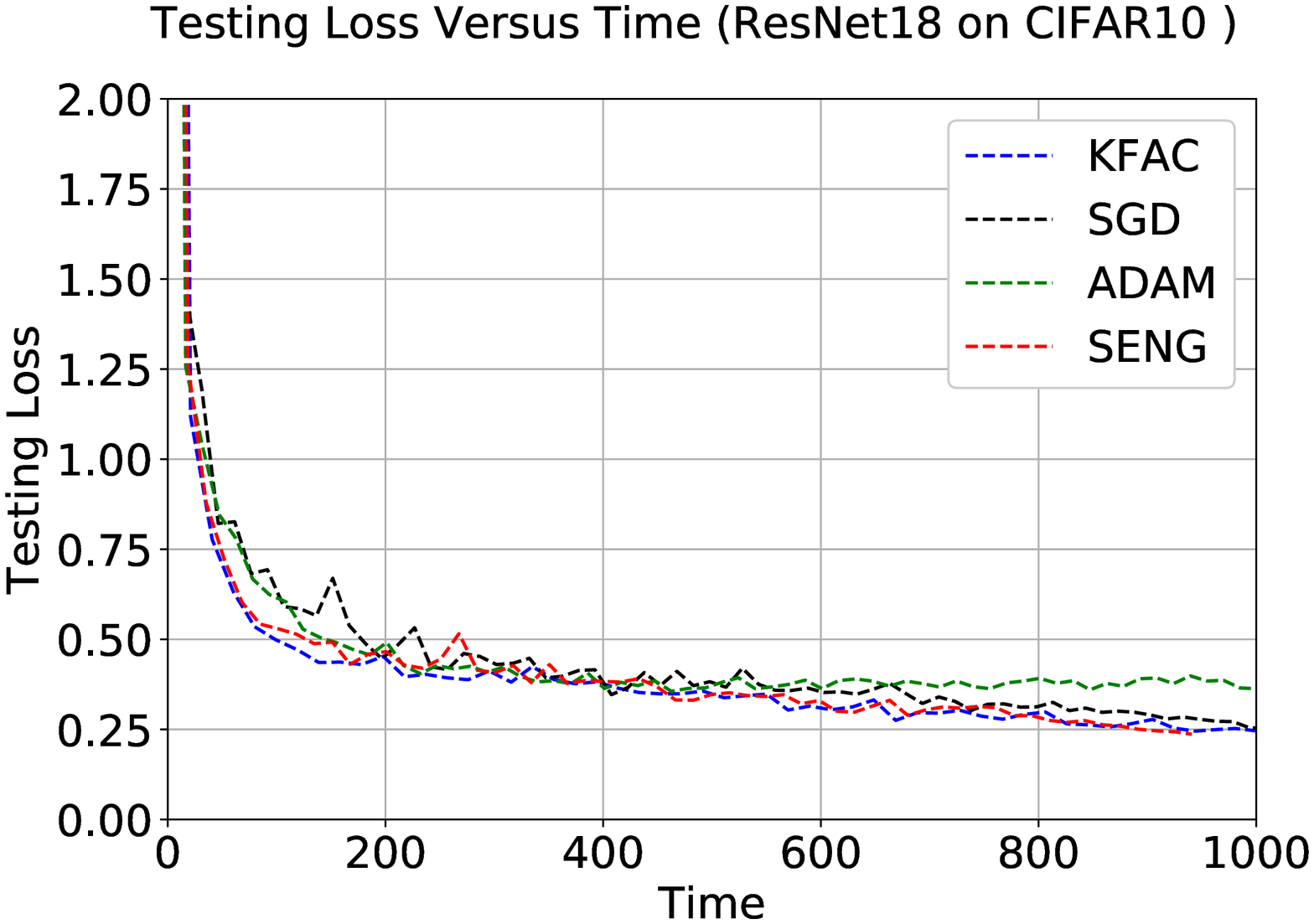}
\includegraphics[width=0.3\textwidth]{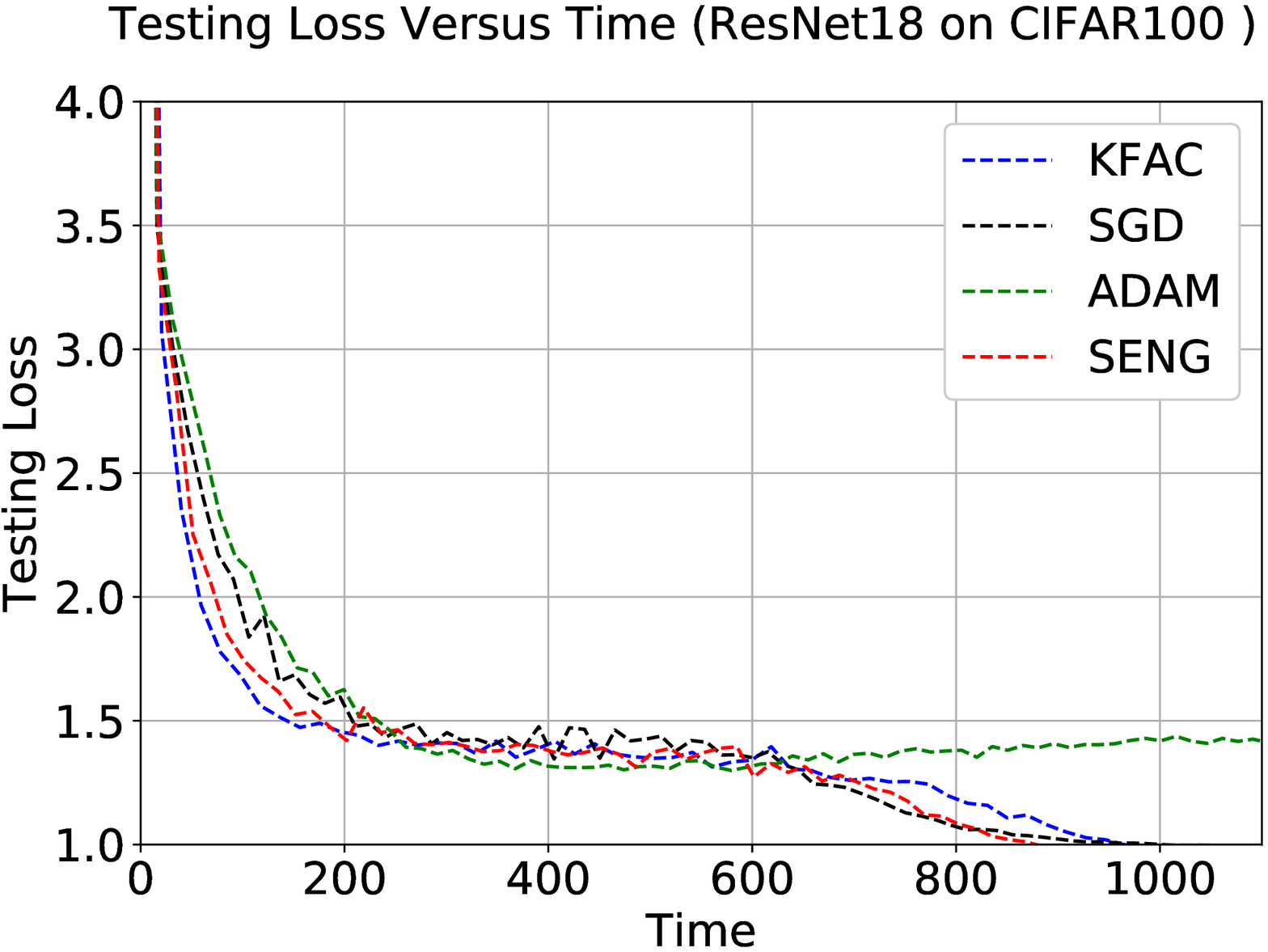}
\includegraphics[width=0.3\textwidth]{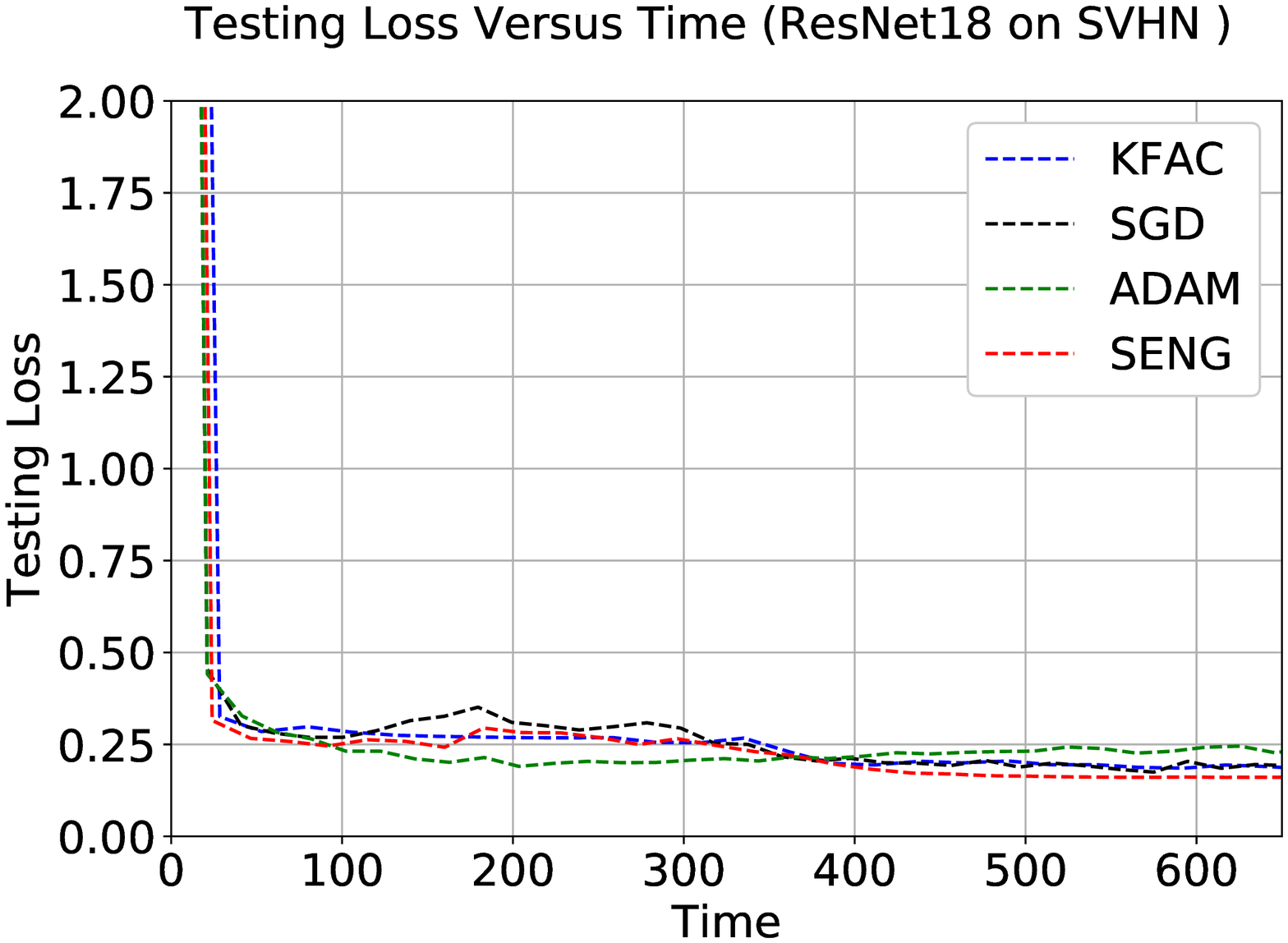}\\
\includegraphics[width=0.3\textwidth]{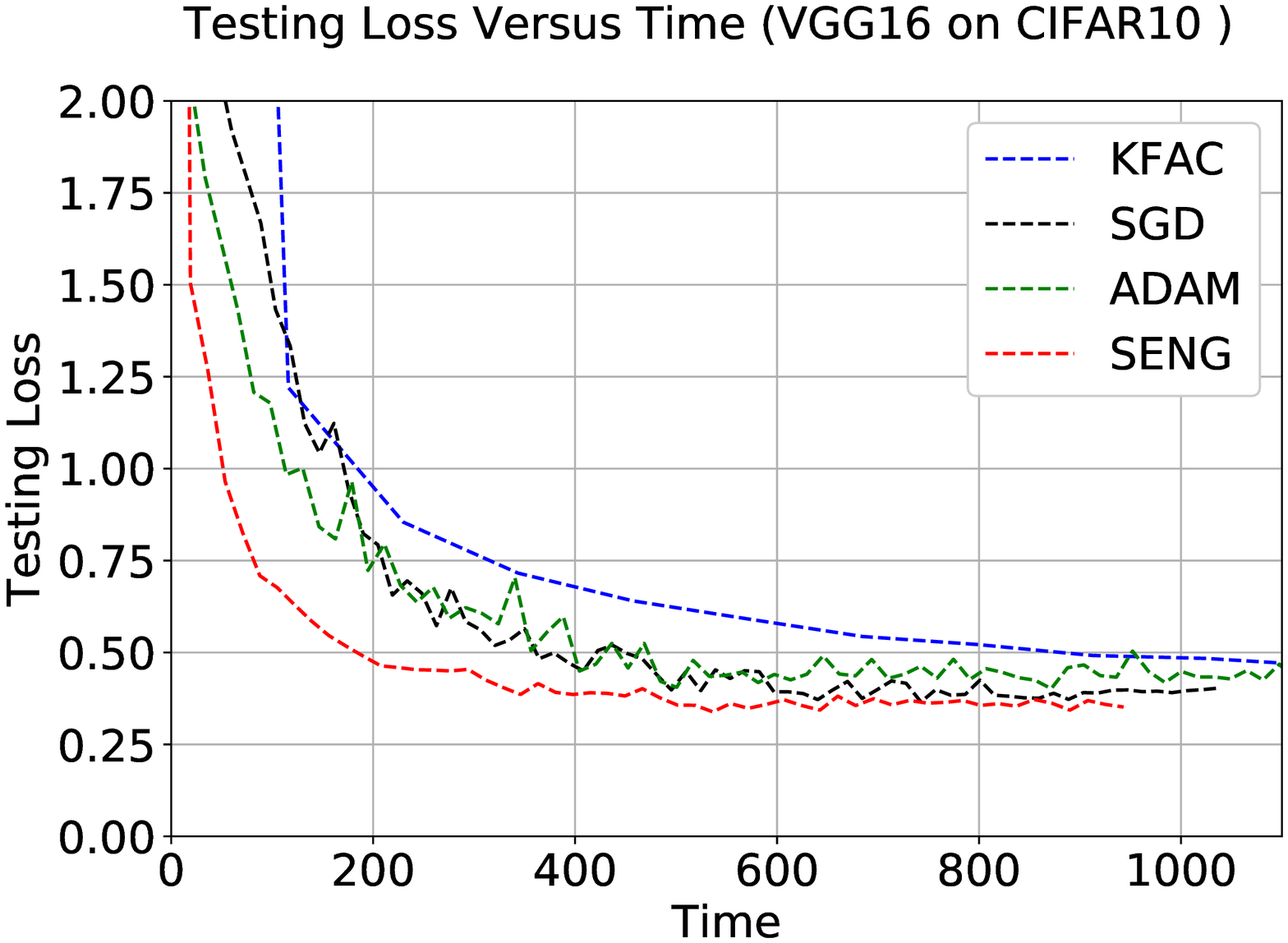}
\includegraphics[width=0.3\textwidth]{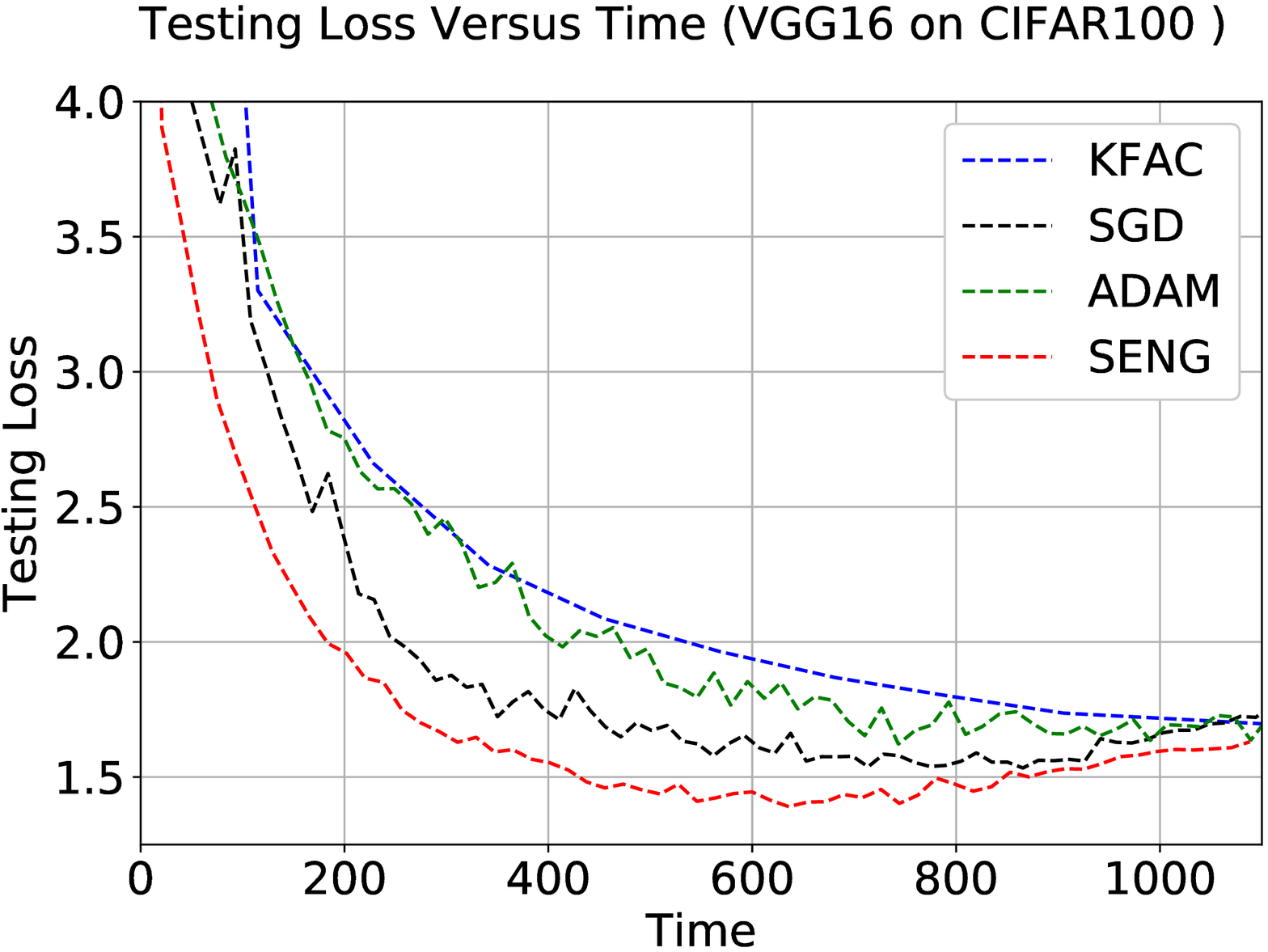}
\includegraphics[width=0.3\textwidth]{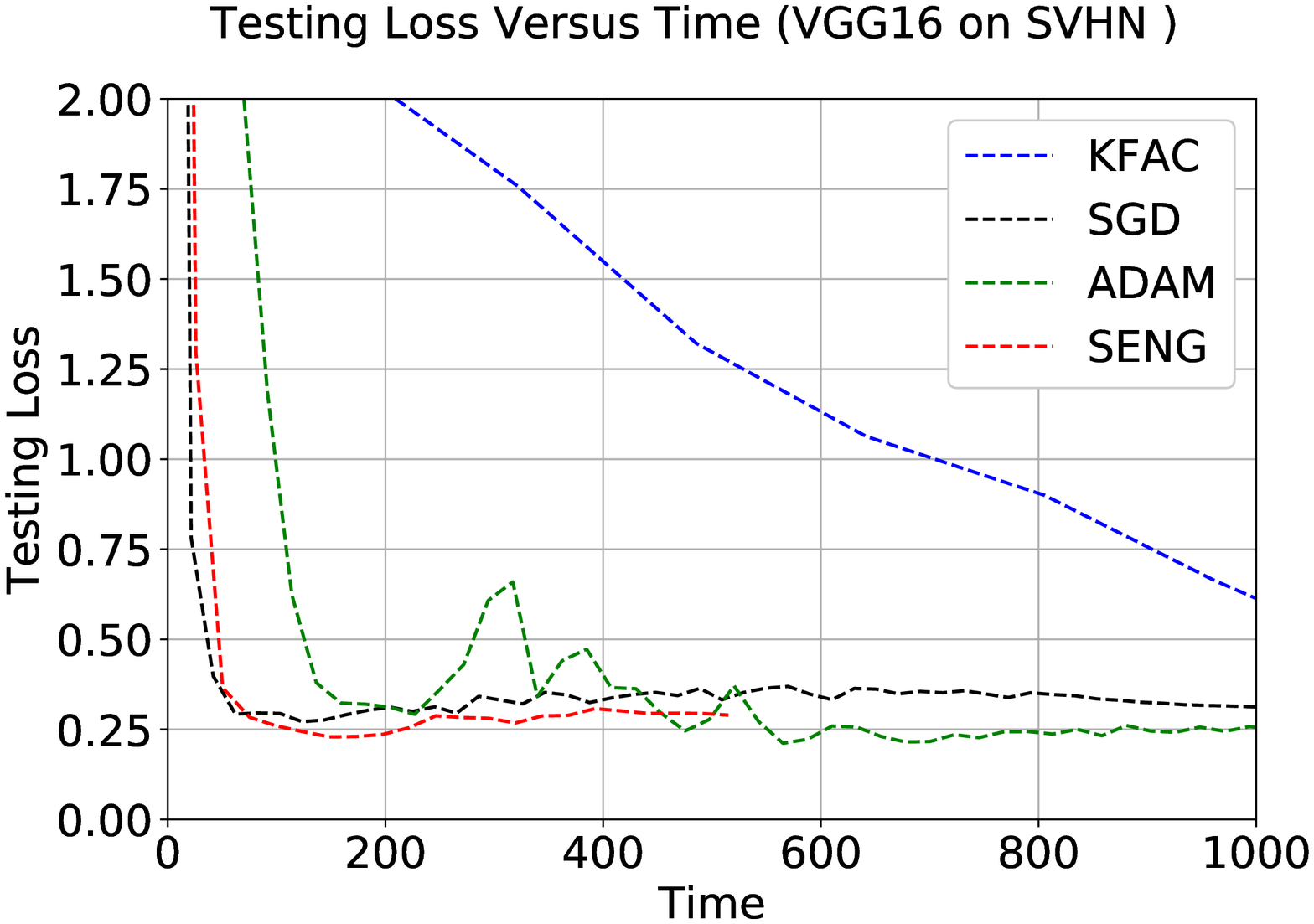}\\
\end{tabular}
\label{shallownet}
\end{figure}

\begin{figure}[H]
\caption{Numerical Comparison on ResNet50 on ImageNet-1k.
%Upper Three: The numerical comparison with KFAC and SGD. Lower Three: The performance between SENG variants with different matrix updating frequency. The number behind SENG is the number of frequency.
}
\centering
%\vspace{0.5ex}
\begin{tabular}{ccc}
\includegraphics[width=0.4\textwidth]{./fig/results-resnet50-Time-Training-Loss.eps}
\includegraphics[width=0.4\textwidth]{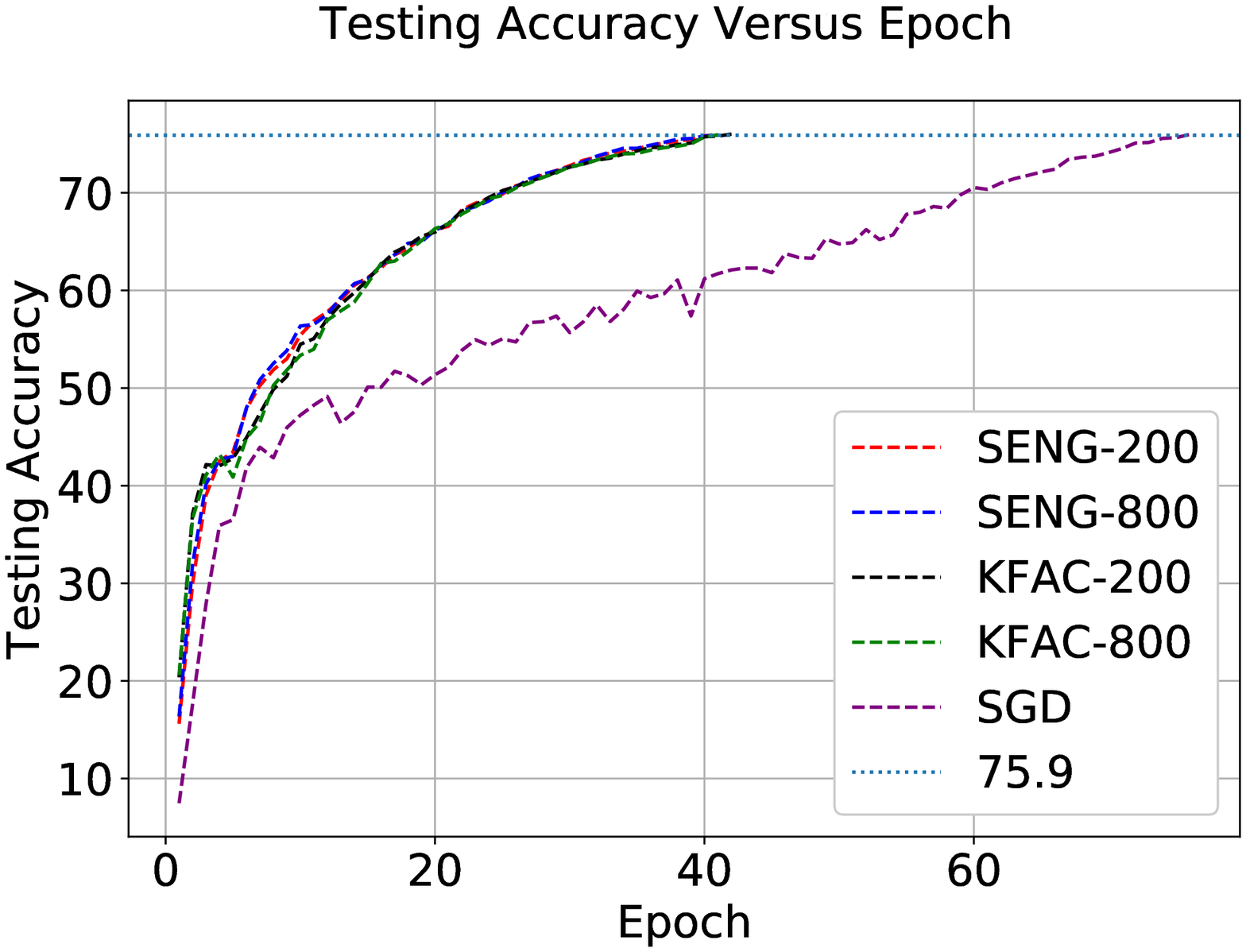}\\
\includegraphics[width=0.4\textwidth]{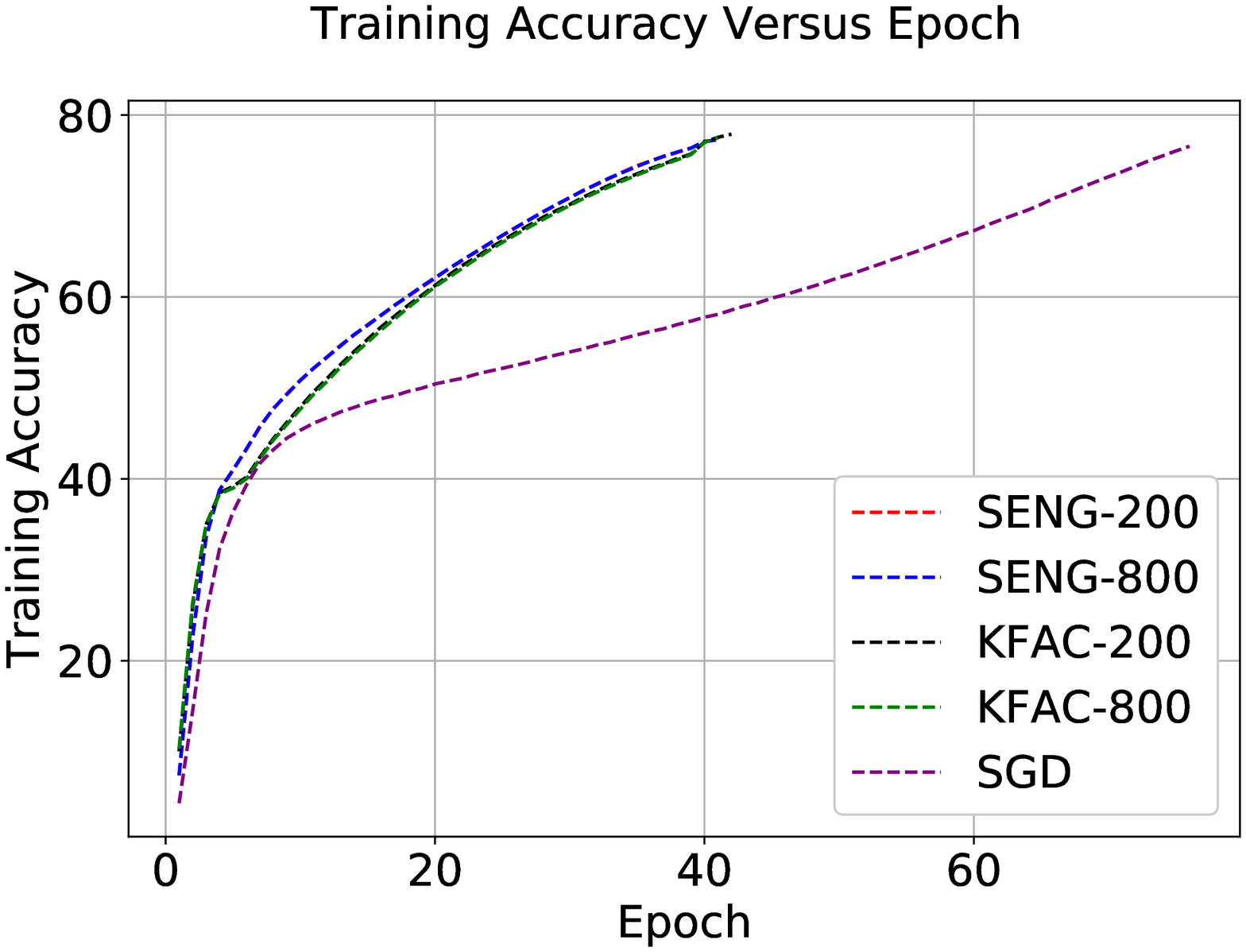}
\includegraphics[width=0.4\textwidth]{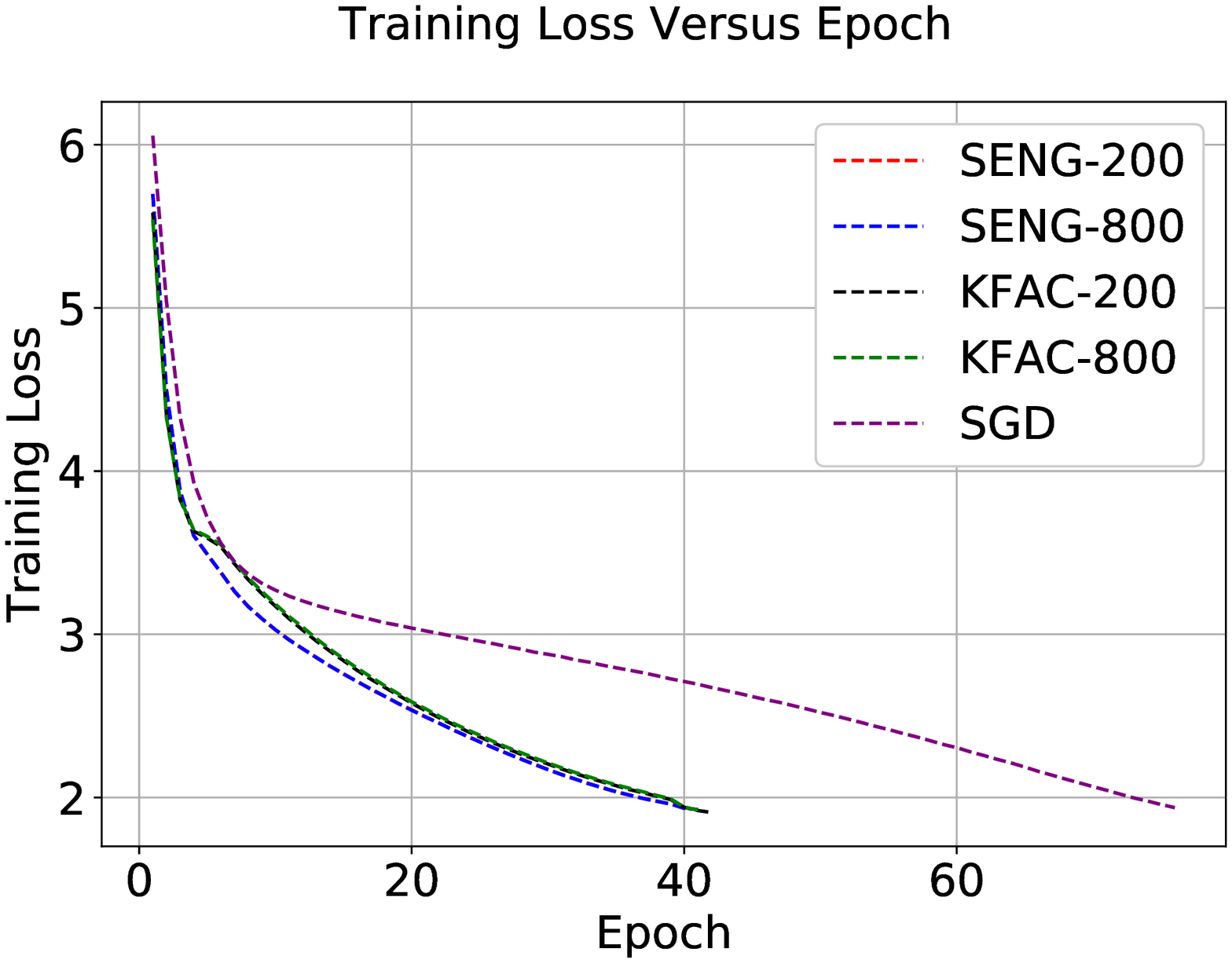}\\
\end{tabular}
\label{imagenet-epoch}
\vspace{-2ex}
\end{figure}

\begin{figure}[H]
\caption{Numerical performance on ResNet50 on ImageNet-1k.
The number behind SENG is the number of matrix update frequency.
}
\centering
\vspace{0.5ex}
\begin{tabular}{cccc}
\includegraphics[width=0.4\textwidth]{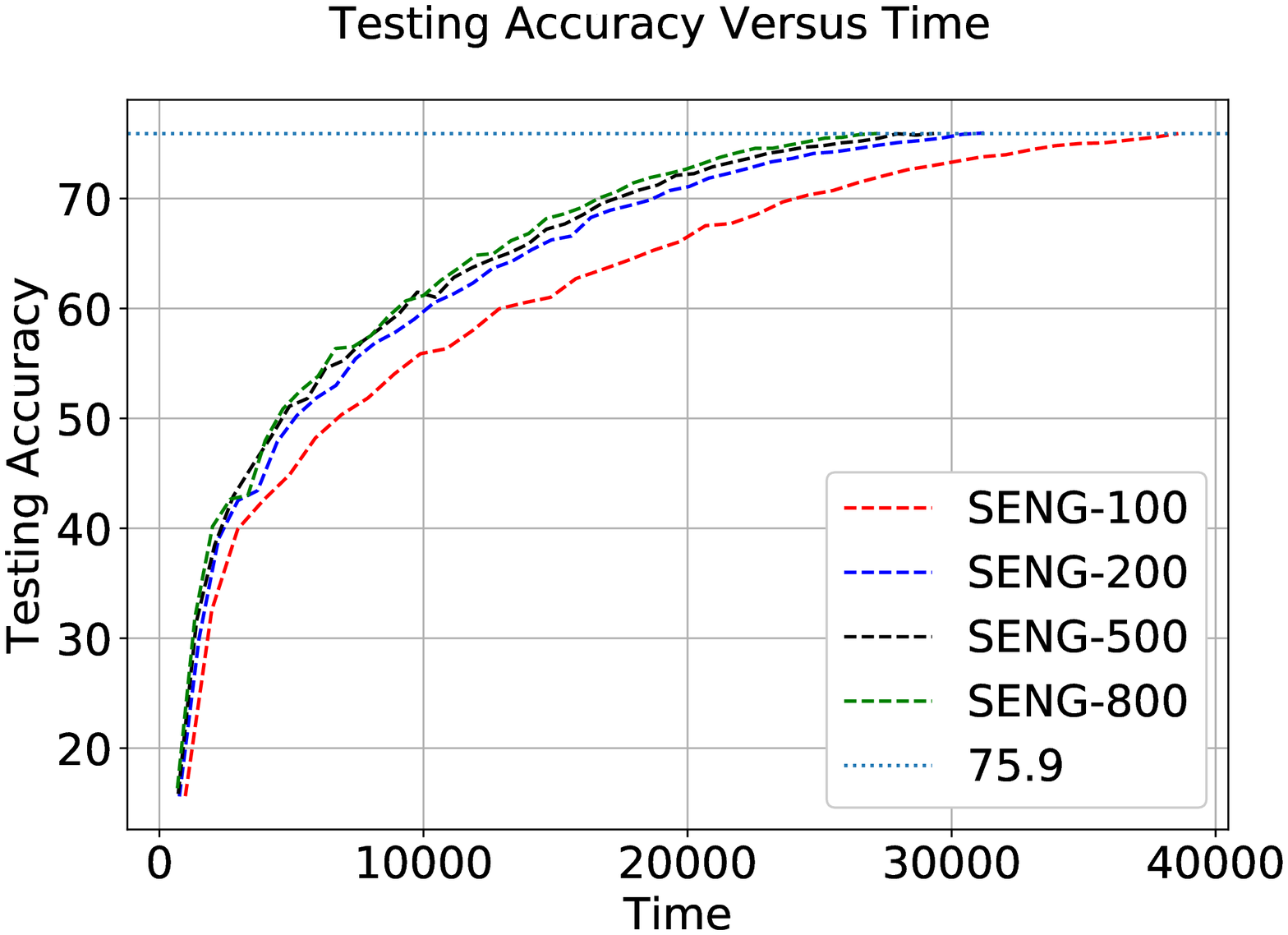}
\includegraphics[width=0.4\textwidth]{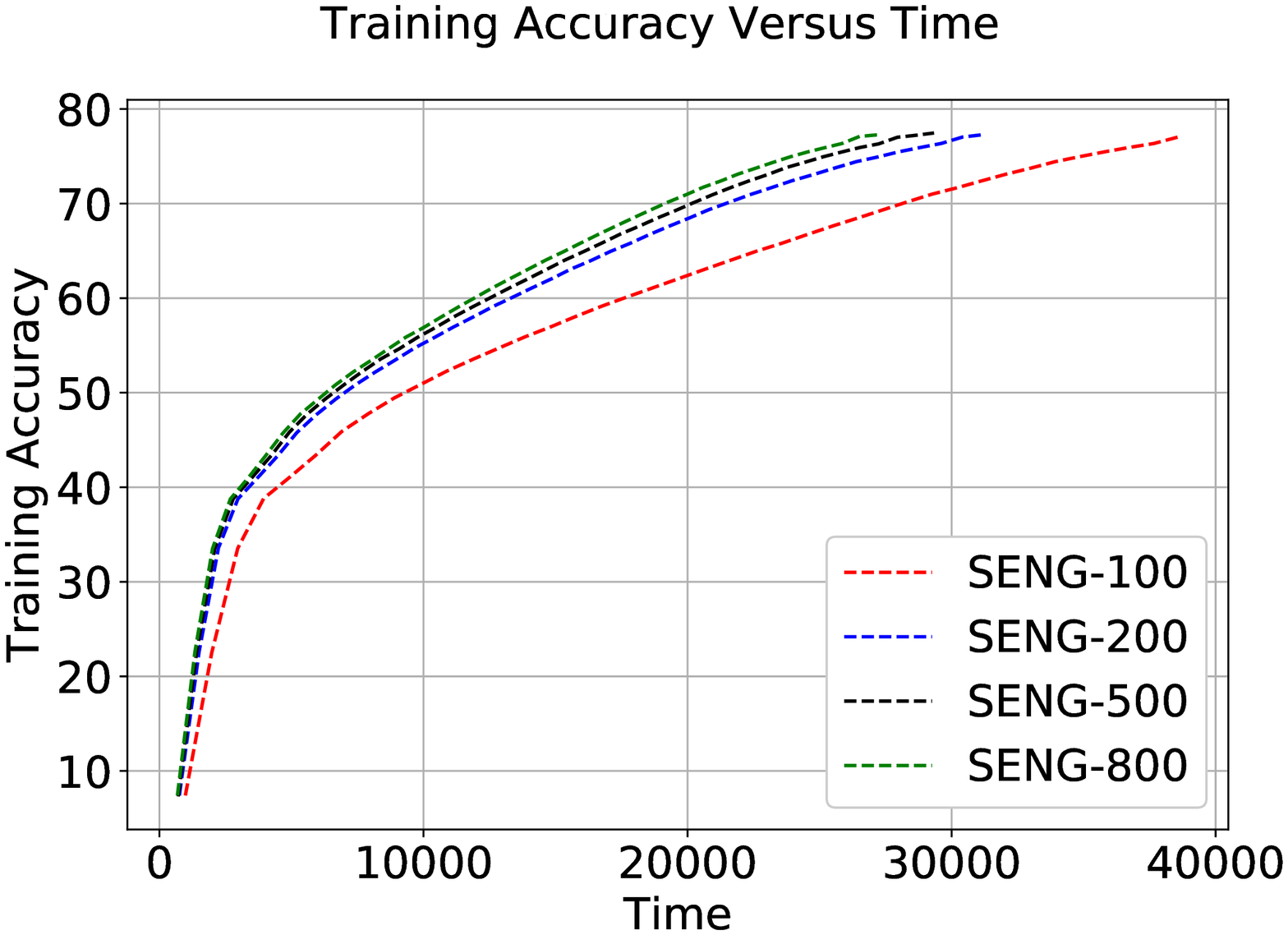}\\
\includegraphics[width=0.4\textwidth]{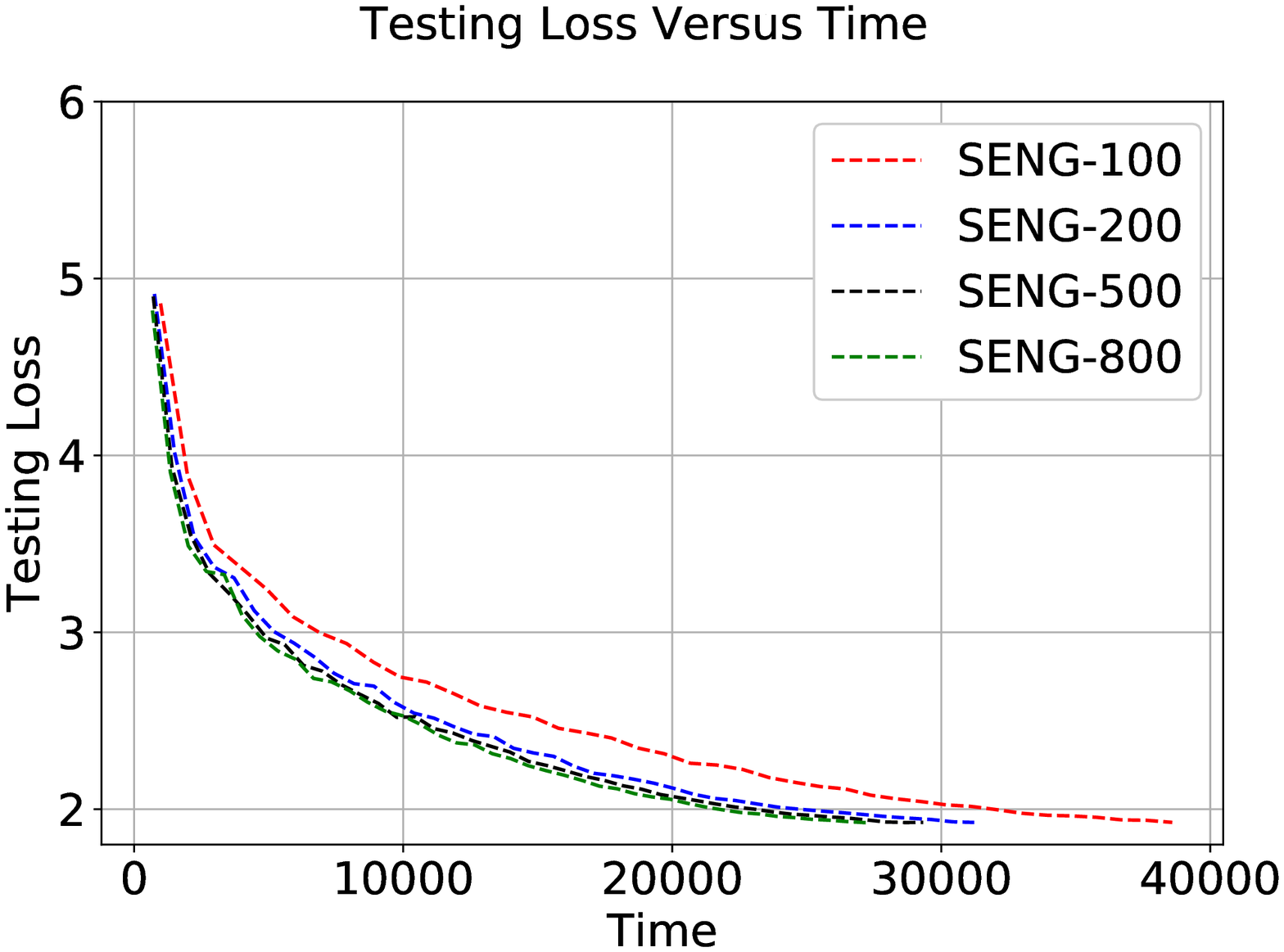}
\includegraphics[width=0.4\textwidth]{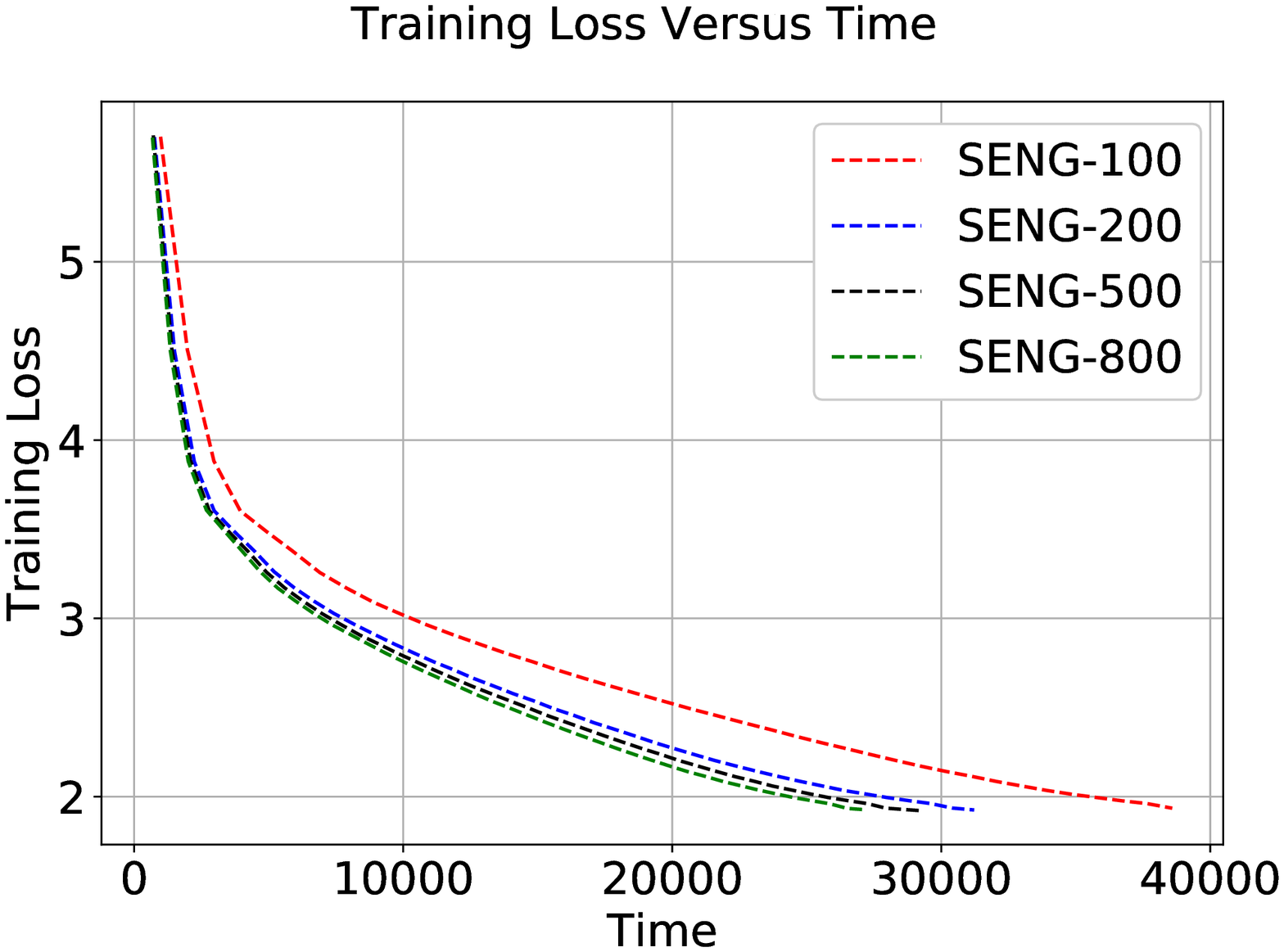}
\end{tabular}
\label{imagenet-freq-time}
\vspace{-2ex}
\end{figure}
\section{Proof of Theorem \ref{thm:globalconv}}
We next estimate the error between $b_k$ in (\ref{sol-ls-1}) and $\hat b_k$ in (\ref{sol-ls-2}) and then establish the global convergence.
\begin{lemma}
\label{diffbhatb}
Suppose that Assumptions A.1, A.2 and B.2 are satisfied with $\eta_k$ and $\epsilon_k$. It holds
\be
\|b_k-\hat b_k\|_2 \leq
%h_1^{1/2} \frac{\sqrt{\epsilon_k}+\sqrt{\frac{h_2}{h_1}}\eta_k}{1-\eta_k} \|g_k\| =
 \frac{\sqrt{\epsilon_k h_1}+\eta_k\sqrt{h_1}}{1-\eta_k}\|g_k\|_2
\ee
with probability at least $1-\delta_k$.
\end{lemma}
%The proof of Lemma \ref{diffbhatb} is shown in Appendix.

%\section{Proof of Lemma \ref{diffbhatb}}
\begin{proof}

%Define $\hat b:= (\lambda I + U^\top \Omega^\top \Omega U)^{-1} U^\top \Omega^\top \Omega g$ and  $ b:= (\lambda I + U^\top U)^{-1} U^\top g.$
%Since the direction $d_k^l$ for all layers are identical, we first focus on one layer and drop the layer indices, then give the deduction for $d_k=[d_k^1,\dots,d_k^L].$

The SVD decomposition of $U_k$ is: $U_k=N_k\Sigma_k V_k,$ where $N_k\in \Rn^{n\times\rho_k}, \Sigma_k \in \Rn^{\rho_k\times \rho_k}, V_k\in \Rn^{\rho_k\times \varrho}$ and $\rho_k$ is the rank of $U_k$. Let $g_k^{\perp} = g_k - U_kU_k^{\dagger}g_k =g_k- N_kN_k^\top g_k$, where $U_k^\dagger$ is the pseudoinverse of $U_k$. By the definition in (\ref{sol-ls-1}) and (\ref{sol-ls-2}),
\[b_k = (\lambda_k I + U_k^\top U_k)^{-1}(U_k^{\top}g_k)\; \text{and}\; \hat b_k = (\lambda_k I + U_k^\top \Omega_k^\top \Omega_k U_k )^{-1}(U_k^\top\Omega_k^\top \Omega_kg_k),\]
we have:
\be
\begin{aligned}
&(\lambda_k I + U_k^\top \Omega_k^\top \Omega_k U_k)(\hat b_k - b_k)\\
%=&    U_k^\top \Omega_k^\top \Omega_k g - (\lambda_k I + U_k^\top \Omega_k^\top \Omega_k U_k)(\lambda_k I + U_k^\top U_k)^{-1} U_k^\top g\\
=&   U_k^\top \Omega_k^\top \Omega_k g_k^{\perp}+ U_k^\top \Omega_k^\top \Omega_k U_kU_k^{\dagger}g_k - (\lambda_k I + U_k^\top \Omega_k^\top \Omega_k U_k) (\lambda_k I + U_k^\top U_k)^{-1} U_k^\top g_k  \\
%=& U_k^\top \Omega_k^\top \Omega_k g^{\perp} - \lambda_k U_k^{\dagger}g + (\lambda_k I + U_k^\top \Omega_k^\top \Omega_k U_k)[ U_k^{\dagger}g-(\lambda_k I + U_k^\top U_k)^{-1} U_k^\top g ]\\
=& U_k^\top \Omega_k^\top \Omega_k g_k^{\perp} - \lambda_k U_k^{\dagger}g_k + (\lambda_k I + U_k^\top \Omega_k^\top \Omega_k U_k)(\lambda_k I + U_k^\top U_k)^{-1} [(\lambda_k I + U_k^\top U_k) U_k^{\dagger}g_k-U_k^\top g_k ]\\
%=& U_k^\top \Omega_k^\top \Omega_k g^{\perp} - \lambda_k U_k^{\dagger}g + (\lambda_k I + U_k^\top \Omega_k^\top \Omega_k U_k)(\lambda_k I + U_k^\top U_k)^{-1} \lambda_k I U_k^{\dagger}g\\
=&  U_k^\top \Omega_k^\top \Omega_k g_k^{\perp}   + \lambda_k (U_k^\top \Omega_k^\top \Omega_k U_k- U_k^\top U_k)(\lambda_k I + U_k^\top U_k)^{-1}U_k^{\dagger}g_k. \end{aligned}
\ee
The last equality follows from the fact that $U_k^\top U_k U_k^\dagger g_k = U_k^\top g_k$.
By Assumption B.2, we know $U_k^\top U_k +\lambda_k I$ is positive definite. We define:
\begin{align}
\label{diffestimate}
(U_k^\top U_k +\lambda_k I)^{-1/2}(\lambda_k I + U_k^\top \Omega_k^\top \Omega_k U_k)(\hat b_k - b_k):=\Pi_k + \Delta_k,
\end{align}

where
\begin{equation*}
\begin{aligned}
\Pi_k =&  (U_k^\top U_k +\lambda_k I)^{-1/2}U_k^\top \Omega_k^\top \Omega_k g_k^{\perp} = V_k^\top(\Sigma_k^2+\lambda_k I)^{-1/2}\Sigma_k N_k^\top  \Omega_k^\top \Omega_k g_k^{\perp}\\
\Delta_k = & \lambda_k (U_k^\top U_k +\lambda_k I)^{-1/2} (U_k^\top \Omega_k^\top \Omega_k U_k- U_k^\top U_k)(\lambda_k I + U_k^\top U_k)^{-1}U_k^{\dagger}g_k\\
= & \lambda_k V_k^\top \Sigma_k(\Sigma_k^2+\lambda_k I)^{-1/2}(N_k^\top\Omega_k^\top\Omega_k N_k - I)(\Sigma_k^2+\lambda_k I)^{-1}N_k^{\top}g_k.
\end{aligned}
\end{equation*}
By Assumption A.1, it holds with probability $1-\delta_k$:
\[
(1-\eta_k)I\preceq N_k^\top\Omega_k^\top\Omega_k N_k \preceq (1+\eta_k) I.
\]
By a left multiplication $V_k^\top\Sigma_k^\top$ , right multiplication $\Sigma_k V_k$ to each matrix and using the fact $N_k^\top N_k = I$, we have
\[
(1-\eta_k) U_k^\top U_k  \preceq U_k^\top\Omega_k^\top\Omega_kU_k \preceq (1+\eta_k) U_k^\top U_k.
\]
This implies
\[
(1-\eta_k) (\lambda_k I + U_k^\top  U_k) \preceq (\lambda_k I + U_k^\top \Omega_k^\top \Omega_k U_k)\preceq (1+\eta_k) (\lambda_k I + U_k^\top   U_k).
\]
Hence, we have
\[
(1-\eta_k)I \preceq (U_k^\top U_k +\lambda_k I)^{-1/2}(\lambda_k I + U_k^\top \Omega_k^\top \Omega_k U_k) (U_k^\top U_k +\lambda_k I)^{-1/2} \preceq (1+\eta_k)I,
\]
which yields
\begin{equation*}
\begin{aligned}
&\|(U_k^\top U_k +\lambda_k I)^{1/2}(\hat b_k - b_k)\|_2\\
 \leq &  \|[(U_k^\top U_k +\lambda_k I)^{-1/2}(\lambda_k I +U_k^\top \Omega_k^\top \Omega_k U_k) (U_k^\top U_k +\lambda_k I)^{-1/2}]^{-1}\|_2 \|\Pi_k+\Delta_k\|_2 \\
 \leq& \frac{1}{1-\eta_k}  \|\Pi_k+\Delta_k\|_2 \leq \frac{1}{1-\eta_k} ( \|\Pi_k\|_2+\|\Delta_k\|_2).
\end{aligned}
\end{equation*}
By using $N_k^\top g_k^{\perp} = 0$ and Assumption A.2, we have with probability $1-\delta_k$:
\be
\label{Piestimate}
\begin{aligned}
\|\Pi_k\|_2 &\leq  \| (\Sigma_k^2+\lambda_k I)^{-1/2}\Sigma_k\|_2 \| N_k^\top\Omega_k^\top \Omega_k g_k^{\perp} - N_k^\top g_k^{\perp}\|_2\\
&\leq \sqrt{\epsilon_k }  \|(\Sigma_k^2+\lambda_k I)^{-1/2}\Sigma_k\|_2  \|g_k^{\perp}\|_2\\
&\leq  \sqrt{\epsilon_k}\|g_k^{\perp}\|_2 \leq \sqrt{\epsilon_k} \|g_k\|_2.
\end{aligned}
\ee
By using Assumptions A.1 and B.2, we have with probability $1-\delta_k$:
\be
\label{deltaestimate}
\begin{aligned}
\|\Delta_k\|_2 &\leq  \lambda_k \| \Sigma_k(\Sigma_k^2+\lambda_k I)^{-1/2}(N_k^\top\Omega_k^\top\Omega_k N_k - I)(\Sigma_k^2+\lambda_k I)^{-1}N_k^{\top}g_k\|_2\\
& \leq \lambda_k \eta_k \|\Sigma_k(\Sigma_k^2+\lambda_k I)^{-1/2}\|_2\|(\Sigma_k^2+\lambda_k I)^{-1}\|_2\|N_k^{\top}g_k\|_2\\
& \leq \eta_k \|\Sigma_k(\Sigma_k^2+\lambda_k I)^{-1/2}\|_2\|N_k^{\top}g_k\|_2\\
&\leq \eta_k  \|g_k\|_2.
\end{aligned}\ee
By Assumption B.2 and combining (\ref{diffestimate}), (\ref{deltaestimate}) and (\ref{Piestimate}), we have
\begin{align}
\|\hat b_k - b_k\|_2 \leq \frac{1}{\sqrt{h_1}} \frac{\sqrt{\epsilon_k}+\eta_k}{1-\eta_k} \|g_k\|_2
\end{align}
%Adding the layer indices, the above deductions actually yield
%\[
%\|\hat b_k^l - b_k^l\|_2 \leq \frac{\sqrt{\epsilon^l_k h_1}+\eta^l_k\sqrt{h_1}}{1-\eta^l_k} \|g^l_k\|_2
%\]
%with probability $1-\delta_k^l.$ Next, we consider the case for all layers. Let $\eta_k^l \equiv \eta_k$ and $\epsilon_k^l \equiv \epsilon_k$, we have
%\[
%\sum_{l=1}^L \|\hat b_k^l - b_k^l\|^2_2 \leq (\frac{\sqrt{\epsilon_k h_1}+\eta_k\sqrt{h_1}}{1-\eta_k})^2 \sum_{l=1}^L\|g^l_k\|^2_2
%\]
%with probability $1-\delta_k :=\Pi_{l=1}^L(1- \delta_k^l)$.
%This is equivalent to
%\[
%\|\hat b_k - b_k\|_2 \leq \frac{\sqrt{\epsilon_k h_1}+\eta_k\sqrt{h_1}}{1-\eta_k} \|g_k\|_2
%\]
with probability $1-\delta_k$ and this completes the proof of Lemma \ref{diffbhatb}.
\end{proof}

%With Lemma \ref{diffbhatb}, it is enough to give the proof of Theorem \ref{thm:globalconv}.
We now give the proof of Theorem  \ref{thm:globalconv} by using Lemma \ref{diffbhatb}.
\begin{proof}
For simplicity, we use $\Expe_k$ to denote the conditional expectation, i.e.,  $\Expe_{k+1} [·] = \Expe [·	| \theta_{k},\dots,\theta_0,\mathbf{\Omega} ]$.
By the definitions (\ref{inversion}) and (\ref{sketch-direc}), we have $\hat d_k- d_k = \frac{1}{\lambda_k}U_k( \hat b_k - b_k)$. Since
$\|U_k\|_2^2 \leq \|\lambda_k I + U_kU_k^\top\|_2 \leq h_2 $,
we obtain $\|U_k\|_2\leq \sqrt{h_2}$.
It follows from Lemma \ref{diffbhatb} that \be
\label{diff-esti}
\Expe_{k} [\|d_k - \hat d_k\|_2] \leq \sqrt{\frac{h_2}{h_1}} \frac{\sqrt{\epsilon_k}+\eta_k}{\lambda_k(1-\eta_k)} \Expe_k[\|g_k\|_2] =
% h_2 \frac{\sqrt{\epsilon_k }+\eta_k}{1-\eta_k} \|g_k\|_2 =
\sqrt{\frac{h_2}{h_1}}  t_k \Expe_k[\|g_k\|_2],
\ee
where $t_k = \frac{\sqrt{\epsilon_k }+\eta_k}{\lambda_k(1-\eta_k)} $ and can be small enough by carefully choosing $\epsilon_k$ and $\eta_k$.
%It follows from Assumptions A.1-A.2 that
%\be\label{dir-esti-1}
%-h_1^{-1}\|g_k\|_2^2\leq \left<g_k,d_k\right> = \left< g_k, -(B_k+\lambda_k I)^{-1}g_k \right > \leq -h_2^{-1}\|g_k\|_2^2\ee
%and
%\be
%\label{dir-esti-2}
%\|d_k\|_2^2 = \|(B_k+\lambda_k I)^{-1}g_k\|_2^2 \leq h_1^{-2}\|g_k\|_2^2.
%\ee

Denote $\tilde{d}_k = -(B_k+\lambda_k)^{-1}\nabla \Psi(\theta_k)$. By combining (\ref{diff-esti}), Assumptions A.1-A.2 and B.1-B.3, and taking the conditional expectation yields
%it holds with probability $(1-\delta_k)^2$:
\be
\label{descent:lemma}
\begin{aligned}
\Expe_{k} [\Psi(\theta_{k+1})] \leq & \Expe_{k}[ \Psi(\theta_k) + \left< \nabla \Psi(\theta_k),\theta_{k+1}-\theta_k\right> + \frac{L_{\Psi}}{2}\|\theta_{k+1}-\theta_k\|_2^2]\\
\leq &\Expe_{k} \left [\Psi(\theta_k) + \left<   \nabla \Psi(\theta_k),\alpha_k( \hat d_k - d_k + d_k - \tilde d_k + \tilde d_k)\right>  + {L_{\Psi}\alpha_k^2}\left [\|d_k\|_2^2 + \|\hat d_k - d_k\|_2^2\right] \right]\\
%&\leq \Psi(\theta_k) -(\alpha_k h_2^{-1}- {L_{\Psi}\alpha_k^2}h_1^{-2}) \|g_k\|_2^2 + \alpha_k \|g_k\|_2\|U_k( \hat b_k - b_k)\|_2 + {L_{\Psi}\alpha_k^2}\|U_k( \hat b_k - b_k)\|_2^2\\
\leq  &\Expe_{k} \left [ \Psi(\theta_k)  -\alpha_k h_2^{-1} \|\nabla \Psi(\theta_k)\|_2^2 + {L_{\Psi}\alpha_k^2}h_1^{-2}\|g_k\|^2_2 + \alpha_k \sqrt{\frac{h_2}{h_1}}t_k \|\nabla \Psi(\theta_k)\|_2 \|g_k\|_2 + {L_{\Psi}\alpha_k^2}\frac{h_2}{h_1}t_k^2\|g_k\|_2^2 \right]\\
\leq  & \Expe_{k} \left [\Psi(\theta_k)  -(\alpha_k h_2^{-1} - \frac{{\alpha_k}h_2^{-1}}{2}) \|\nabla \Psi(\theta_k)\|_2^2   + \alpha_k \frac{h_2^2}{2h_1} t_k^2 \|g_k\|_2^2  + \frac{L_{\Psi}\alpha_k^2h_1^{-2}}{2}\|g_k\|_2^2 + {L_{\Psi}\alpha_k^2}\frac{h_2}{h_1}t_k^2\|g_k\|_2^2 \right]\\
 \leq & \Psi(\theta_k) - \alpha_k \left( \frac{1}{2}h_2^{-1}- \frac{L_{\Psi}\alpha_kh_1^{-2}}{2} -\frac{h_2^2}{2h_1} t_k^2 - {L_{\Psi}\alpha_k}\frac{h_2}{h_1}t_k^2\right)\|\nabla \Psi(\theta_k)\|^2_2 \\
& +\alpha_k \left( \frac{L_{\Psi}\alpha_kh_1^{-2}}{2} +\frac{h_2^2}{2h_1} t_k^2 + {L_{\Psi}\alpha_k} \frac{h_2}{h_1}t_k^2\right)\sigma_k^2.
%\\
%&:=\Psi(\theta_k) - \alpha_kc_1 \|g_k\|^2.
\end{aligned}
\ee
%where in the last inequality, we use with at least $1-\delta_k$, $\|g_k - \nabla \Psi(\theta_k)\|^2 \leq \sigma_k^2$ holds.
If $\alpha_k \leq \min\{\frac{1}{2L_{\Psi}},\frac{h_1^2}{2L_\Psi h_2}\}$, we have
\[\frac{1}{h_2}- L_{\Psi}\alpha_k\frac{1}{h_1^2}  > \frac{1}{2h_2}.\]
If $t_k^2 <\frac{h_1}{4h_2(h_2^2-h_1^2)}$, we obtain
\be
\label{parameter-estimate}
\left( \frac{1}{2}h_2^{-1}- \frac{L_{\Psi}\alpha_kh_1^{-2}}{2} -\frac{h_2^2}{2h_1} t_k^2 - {L_{\Psi}\alpha_k}\frac{h_2}{h_1}t_k^2\right) > \frac{1}{4h_2}- (\frac{h_2^2}{2h_1} t_k^2 + {L_{\Psi}\alpha_k}\frac{h_2}{h_1}) t_k^2 >  \frac{1}{8h_2}.
\ee
Let $\epsilon_k$ and $\eta_k$ be small enough so that $t_k$ is small. This can be achieved by choosing suitable sample sizes. Combining (\ref{parameter-estimate}), summing the inequality (\ref{descent:lemma}) and taking expectation, we get
\be
\label{descent-fobj}
\sum_{k=0}^{\infty}   \frac{\alpha_k}{8h_2} \Expe\left[\|\nabla \Psi(\theta_k)\|_2^2|\mathbf{\Omega}\right] \leq  \Psi(\theta_0) - \Psi^*  + \sum_{k=0}^\infty \left( \frac{L_{\Psi}\alpha_kh_1^{-2}}{2} +\frac{h_2^2}{2h_1} t_k^2 + {L_{\Psi}\alpha_k} \frac{h_2}{h_1}t_k^2\right)\alpha_k \sigma_k^2
<\infty.
\ee
%This completes the proof of Theorem \ref{thm:globalconv}.
% $h_3$ is a function of $\epsilon$ and $\eta$, therefore it follows that $c_1= c_1(\epsilon,\eta)$. Since $\delta$ do not depend on $\eta,\epsilon$, we can let $\eta,\epsilon$ be small enough to let $c_1$ be positive.

Define the events:
$$\Gamma_k =\left \{\|N_k^\top \Omega_k^\top \Omega_k N_k- I\|_2 \leq \eta_k,
\|N_k^\top \Omega_k^\top \Omega_k v - N_k^\top v\|^2_2 \leq \epsilon_k \|v\|_2^2
\right\},$$ where $N_k$ is an orthogonal basis for the column span of $U_k$.  From Assumptions A.1 and A.2, it is easy to know $\Prob(\Gamma_k)=1-\delta_k$
and $\Prob( \Gamma) = \Pi_{i=0}^\infty(1-\delta_k)$ where $\Gamma = \cap_{k=0}^\infty \Gamma_k$. It follows from \eqref{descent-fobj} that on the event $\Gamma$, we have \be \label{step-var-summable2}\sum_{k=0}^\infty \alpha_k \Expe[\|\nabla \Psi(\theta_k)\|_2^2]< \infty. \ee
This implies that there exists an event $\widetilde{\mathbf{\Omega}}$ such that $ \Prob(\widetilde{\mathbf{\Omega}}) = \Pi_{i=1}^\infty(1-\delta_k)$ and on the event we have
\be \label{step-var-summable}\sum_{k=0}^\infty \alpha_k \|\nabla \Psi(\theta_k)\|_2^2< \infty.\ee

The deductions in the next are on the events $ \widetilde {\mathbf{\Omega}}$. %with probability $\Pi_{k=1}^\infty (1-\delta_k)$.
Since $\sum_{k=0}^\infty \alpha_k = \infty$, there is a subsequence $\{o_i\}_i$ such that $\{\|\nabla \Psi(\theta_{o_i})\|^2_2 \rightarrow 0\}$, which is equivalent to
\be
\label{inf-conv} \lim_{k\rightarrow\infty} \inf \|\nabla \Psi(\theta_k)\|_2 = 0.\ee
%with probability $\Pi_{k=1}^\infty (1-\delta_k)$.
By Assumptions B.1-B.3, (\ref{diff-esti}) and (\ref{step-var-summable}), we obtain
\be
\label{theta-summable}
\begin{aligned}
\sum_{k=0}^\infty \alpha_k^{-1}\|\theta_{k+1}-\theta_k\|_2^2 = \sum_{k=0}^\infty \alpha_k\|d_k - d_k + \hat d_k\|_2^2\leq \sum_{k=0}^\infty \alpha_k (h_1^{-2} + \frac{h_2}{h_1}t_k^2) \left (\|\nabla \Psi(\theta_k)\|_2^2 + \sigma_k^2 \right) <\infty.\\
\end{aligned}
\ee
%with probability $\Pi_{k=1}^\infty (1-\delta_k)$.
The last inequality follows from the fact that $t_k$ has an upper bound. Next, we prove the result by contradiction. Assume that there exists $\epsilon>0$ and two increasing sequences $\{p_i\}_i $, $\{q_i\}_i$ such that $p_i < q_i$ and
\[
    \|\nabla \Psi(\theta_{p_i})\|_2 \geq 2\epsilon, \quad \|\nabla \Psi(\theta_{q_i})\|_2 < \epsilon, \quad \|\nabla \Psi(\theta_{k})\|_2 \geq \epsilon,
\]
for $ k = p_i+1,\dots,q_i-1.$
%with probability $\Pi_{k=1}^\infty (1-\delta_k)$.
Thus, it follows that
\be
\begin{aligned}
\epsilon^2 \sum_{i=0}^{\infty} \sum_{k=p_i}^{q_i-1} \alpha_k
& \leq \sum_{i=0}^{\infty} \sum_{k=p_i}^{q_i-1} \alpha_k \|\nabla \Psi(\theta_k)\|_2^2
\leq \sum_{k=0}^{\infty}\alpha_k \|\nabla \Psi(\theta_k)\|_2^2< \infty .
\end{aligned}
\ee
Setting $\zeta_i = \sum_{k=p_i}^{q_i-1} \alpha_k$ implies $\zeta_i\rightarrow 0.$ Then by the H\"older's inequality and (\ref{theta-summable}), we obtain
\[
    \|\theta_{p_i} -\theta_{q_i} \|_2 \leq \sqrt{\zeta_i} [\sum_{k=p_i}^{q_i-1} \alpha_k^{-1}\|\theta_{k+1} - \theta_k\|_2^2]^{1/2} \rightarrow 0.
\]
Due to the Lipschitz property of $\nabla \Psi$, we have $\lim_{i\rightarrow \infty} \|\nabla \Psi(\theta_{p_i}) -\nabla \Psi(\theta_{q_i}) \|_2\rightarrow 0$
%with probability $\Pi_{k=1}^\infty (1-\delta_k)$
, which is a contradiction. This implies  $\lim_{k\rightarrow \infty }\|\nabla \Psi(\theta_k) \|_2 = 0$ with probability $\Pi_{k=0}^\infty (1-\delta_k)$ and completes the proof.
\end{proof}

\section{Proof of Theorem \ref{ntk-theorem}}
%Let us first list two conditions and prove these main results under the two conditions. 
%Note that under the standard assumptions in the NTK case the two conditions hold, which is proved in Lemma 6-8 in \cite{zhang2019fast}.

We assume that the following two conditions hold:
\begin{itemize}
\item Condition 1. The matrix $G_0$ is positive definite, where $G_0 = J_0J_0^\top$ and $J_0\in \Rn^{N \times n}$ is the Jacobian matrix. 
%and $n>N.$
\item Condition 2.
The exists constant $0\leq C < \frac{1}{2}$ such that for any $\|\theta - \theta_0\|_2 \leq \frac{3\|y-f^0\|_2}{\sqrt{\lambda_{\text{min}}(G_0)}}$, we have
\[
\|J(\theta) - J_0\| \leq \frac{C}{3}\sqrt{\lambda_{\text{min}}(G_0)}.
\]
\end{itemize}

%\begin{remark}
%\label{condition2-rmk}
As is proved in Lemmas 6-8 in \cite{zhang2019fast}, the above conditions hold with probability $1-\delta$ if $\hat m$ is set to be $ \Omega\left( \frac{n^4}{v^2\lambda_0^4\delta^3}\right)$. Define the corresponding event by $\Gamma$ such that $\Prob(\Gamma) = 1-\delta.$ Note that the sketching matrices are independent from each other and also from the weight initialization. Define the events:
$\Gamma_k =\left \{\|N_k^\top \Omega_k^\top \Omega_k N_k- I\|_2 \leq \eta_k,
\|N_k^\top \Omega_k^\top \Omega_k v - N_k^\top v\|^2_2 \leq \epsilon_k \|v\|_2^2
\right.\}$, such that $\Prob(\Gamma_k) = 1- \delta_k$, where $N_k$ is an orthogonal basis for the column span of $J_k^\top$. Let $\widetilde{\Omega} =\Gamma \cap \left ( \cap_{k=1}^\infty \Gamma_k\right)$ and our analysis is mainly on the event $\widetilde \Omega$.
%\end{remark}

\begin{lemma}
\label{psd-jacobian}
If Condition 2 holds and $\|\theta - \theta_0\|_2 \leq \frac{3\|y-f^0\|_2}{\sqrt{\lambda_{\text{min}}(G_0)}}$, we have 
 $\lambda_\mathrm{min}(G(\theta)) \geq \frac{4}{9} \lambda_\mathrm{min}(G_0)$, where $G(\theta) = J(\theta)J(\theta)^\top.$
\end{lemma}
\begin{proof}
By Condition 2, if $\|\theta - \theta_0\|_2 \leq \frac{3\|y-f^0\|_2}{\sqrt{\lambda_{\text{min}}(G_0)}}$, we have :
\begin{align}
\sigma_{\text{min}}(J(\theta)) & \geq \sigma_{\text{min}}(J_0) - \|J(\theta) - J_0\| \geq \frac{2}{3}\sqrt{\lambda_{\text{min}}(G_0)}, \label{sigma-estimate-lemma}
\end{align}
which means $\lambda_{\text{min}}(G(\theta)) \geq \frac{4}{9}{\lambda_{\text{min}}(G_0)} .$
\end{proof}

\begin{lemma}
\label{main-lemma-ntk}
Assume that  $\|\theta_{k+1} - \theta_0\|_2 \leq \frac{3\|y-f^0\|_2}{\sqrt{\lambda_{\text{min}}(G_0)}}$ and  $\|\theta_k - \theta_0\|_2 \leq \frac{3\|y-f^0\|_2}{\sqrt{\lambda_{\text{min}}(G_0)}}$. Under Conditions 1 and 2, there exists a constant $\zeta\in (0,1)$, such that 
\be \label{func-value-descent} \|f^{k+1} - y\|_2^2 \leq \zeta \|f^k - y\|_2^2.\ee
\end{lemma}

\begin{proof}
The following analysis is on the events $\widetilde \Omega$.
%
%We first show, given the initialization and at the iteration $k$, with probability $(1-\delta_k)$, we have
%\be \label{theorem-ntk-1}\|f^{k+1} - y\|_2^2 \leq \zeta \|f^{k} - y\|_2^2, \ee
%where $\zeta\in (0,1)$ is a constant.
From Condition 2, we can obtain the bound of the Jacobian matrix $J(\theta)$ for any $\theta$ satisfying $\|\theta - \theta_0\|_2 \leq \frac{3\|y-f^0\|_2}{\sqrt{\lambda_{\text{min}}(G_0)}}$:
\begin{align}
\sigma_{\text{min}}(J(\theta)) & \geq \sigma_{\text{min}}(J_0) - \|J(\theta) - J_0\| \geq \frac{2}{3}\sqrt{\lambda_{\text{min}}(G_0)}, \label{sigma-estimate-1} \\
\sigma_{\text{max}}(J(\theta)) & \leq \sigma_{\text{max}}(J_0) + \|J(\theta) - J_0\|  \leq \frac{C}{3} \sqrt{\lambda_{\text{min}}(G_0)}  + \sqrt{\lambda_{\text{max}}(G_0)}:= J_\Lambda.  \label{sigma-estimate-2}
\end{align}
Remind that the direction is
\be
\begin{aligned}
d_k
%& \frac{1}{\lambda_k}J_k^\top\left ( I -  (\lambda_k I + J_k\Omega_k \Omega_k^\top J_k^\top)^{-1}J_k\Omega_k \Omega_k^\top  J_k^\top \right ) f^k \\
=& \frac{1}{\lambda_k}J_k^\top \left (Q_k \left( I - (\lambda_k I + \widetilde \Sigma_k)^{-1}  \widetilde \Sigma_k \right ) Q_k^\top \right)( f^k-y)\\
= & J_k^\top \left (Q_k ( \lambda_k I + \widetilde \Sigma_k)^{-1}  Q_k^\top \right) (f^k-y),
\end{aligned}
\vspace{-1ex}
\ee
where $ Q_k \widetilde \Sigma_k Q_k$ is the eigenvalue decomposition of $J_k\Omega_k^\top \Omega_k J_k^\top$, $Q_k$ is orthogonal and $\widetilde \Sigma_k$ is diagonal.

Hence, we have the following estimate between $d_k$ and $\overline{d}_k$ where $\overline{d}_k = J_k^\top (J_k\Omega_k^\top \Omega_kJ_k^\top)^{-1}(f^k-y)$:
\be
\begin{aligned}
\label{damping-direc-error}
&\left \|d_k -\overline{d}_k\right \|_2 \\
= & \left \|J_k^\top \left (Q_k ( \lambda_k I + \widetilde \Sigma_k)^{-1}  Q_k^\top \right) (f^k-y) - J_k^\top \left (Q_k (  \widetilde \Sigma_k)^{-1}  Q_k^\top \right) (f^k-y) \right \|_2 \\
\leq & \left \|J_k^\top \left ( Q_k \lambda_k ( \lambda_k I + \widetilde \Sigma_k)^{-1} (  \widetilde \Sigma_k)^{-1}  Q_k^\top\right ) (f^k - y) \right \|_2\\
\leq & \frac{27\lambda_k }{8(1-\eta_k)^2} \frac{J_\Lambda}{ (\lambda_{\text{min}}(G_0))^{2}} \|f^k - y\|_2 := \frac{ \mathcal{D}(\lambda_k)}{ (\lambda_{\text{min}}(G_0))^{1/2}} \|f^k - y\|_2 .
\end{aligned}
\ee
The above error can be controlled by choosing a small enough $\eta_k$, $\lambda_k$ and the above error vanishes when $\lambda_k=0$.

 According to the update sequence, we calculate the difference of network outputs of two consecutive iterations:
\begin{align}
f^{k+1} - f^k  = &  \int_{s=0}^1 \left < \fracp{f(\theta(s))}{\theta}, - \alpha d_k  \right > ds \\
 = & - \int_{s=0}^1 \left < \fracp{f(\theta_k)}{\theta},\alpha J_k^\top (J_kJ_k^\top)^{-1}(f^k - y)   \right > ds \label{ntk-estimate-1}\\
 & + \int_{s=0}^1 \left < \fracp{f(\theta(s))}{\theta},\alpha J_k^\top \left ( (J_kJ_k^\top)^{-1} - (J_k\Omega_k^\top \Omega_kJ_k^\top)^{-1}\right)(f^k - y)   \right > ds \label{ntk-estimate-2} \\
  & +  \int_{s=0}^1 \left < \fracp{f(\theta_k)}{\theta} - \fracp{f(\theta(s))}{\theta},\alpha J_k^\top (J_kJ_k^\top)^{-1}(f^k - y)   \right > ds. \label{ntk-estimate-3}\\
  & + \int_{s=0}^1 \left < \fracp{f(\theta(s))}{\theta},\alpha \left( \overline{d}_k -  d_k\right )\right > ds \label{ntk-estimate-4} ,
\end{align}
where $\theta(s) = s \theta_{k+1} + (1-s)\theta_k$.
Then, we estimate four terms (\ref{ntk-estimate-1}-\ref{ntk-estimate-4}). It is easy to see that
\be \label{estimate-1} \|\eqref{ntk-estimate-1}\|_2 =-\alpha (f^k - y). \ee

By \eqref{sigma-estimate-1}, \eqref{sigma-estimate-2} and Lemma \ref{psd-jacobian}, we have
\be \label{estimate-2}\begin{aligned}
\|\eqref{ntk-estimate-3}\|_2 &\leq  \alpha \left\| \int_{s=0}^1  \fracp{f(\theta_k)}{\theta} - \fracp{f(\theta(s))}{\theta}    ds\right\|_2\| J_k^\top (J_kJ_k^\top)^{-1}(f^k - y)\|_2 \\
& \leq \alpha \frac{2C}{3}\sqrt{\lambda_{\text{min}}(G_0)} \frac{1}{\sqrt{\lambda_{\text{min}}(G_k)}}\|f^k - y\|_2\\
& \leq \alpha \frac{2C}{3}\sqrt{\lambda_{\text{min}}(G_0)} \frac{3}{2\sqrt{\lambda_{\text{min}}(G_0)}}\|f^k - y\|_2\\
&\leq \alpha C \|f^k - y\|_2 .
\end{aligned}\ee
By Assumptions A.1-A.2, on the event $\widetilde \Omega$, we have
\be
\label{efim-diff-1}
\|J_kJ_k^\top - J_k\Omega_k^\top \Omega_kJ_k^\top\|_2 \leq \eta_k \|J_k\|^2_2.
\ee
Then, we have
\be
\label{efim-diff-2}
\|(J_kJ_k^\top)^{-1} - (J_k\Omega_k^\top \Omega_kJ_k^\top)^{-1}\|_2 = \|(J_kJ_k^\top)^{-1} (J_kJ_k^\top - J_k\Omega_k^\top \Omega_kJ_k^\top)(J_k\Omega_k^\top \Omega_kJ_k^\top)^{-1}\|_2 \leq \frac{\eta_k}{1-\eta_k} \frac{J_\Lambda^2}{\lambda_{\text{min}}(G_k)^{2}}.
\ee
This leads to:
\be
\label{estimate-3}
\begin{aligned}
\left \| \eqref{ntk-estimate-2}\right \|  & \leq  \alpha \left \|  \int_{s=0}^1  \fracp{f(\theta(s))}{\theta}ds  \right \|_2  \| J_k^\top \left ( (J_kJ_k^\top)^{-1} - (J_k\Omega_k^\top \Omega_kJ_k^\top)^{-1}\right) \|_2 \|f^k - y\|_2  \\
& \leq \alpha \frac{\eta_k}{1-\eta_k} \frac{81 J_{\Lambda}^4}{16\lambda_{\text{min}}(G_0)^2}  \|f^k - y\|_2.
\end{aligned}
\ee
By \eqref{damping-direc-error}, the term \eqref{ntk-estimate-4} can be bounded by:
\be
\label{estimate-4}
\|\eqref{ntk-estimate-4}\|_2  \leq \alpha \frac{ \mathcal{D}(\lambda_k)J_{\Lambda}}{ (\lambda_{\text{min}}(G_0))^{1/2}}  \|f^k - y\|_2.
\ee

Combining \eqref{estimate-1},  \eqref{estimate-2}, \eqref{estimate-3}, \eqref{estimate-4},  denoting $\widetilde C = C + \frac{\eta_k}{1-\eta_k} \frac{81 J_{\Lambda}^4}{16\lambda^2_{\text{min}}(G_0)} + \mathcal{D}(\lambda_k) J_{\Lambda}$, and letting $\eta_k$, $\lambda_k$ be small enough such that $\widetilde C\in(0,\frac{1}{2})$, if $\alpha \leq \frac{1-2\widetilde C}{(1+\widetilde C)^2}$, we obtain 
\be
\label{theorem-ntk-1}
\begin{aligned}
\| f^{k+1} -y \|_2^2 &=  \| f^{k} -y \|_2^2 + 2\left <f ^{k} -y, f^{k+1} - f^k \right > + \| f^{k+1} -y \|_2^2 \\
& \leq \left ( 1 - 2 \alpha + 2\alpha \widetilde C   + \alpha^2(1+\widetilde C)^2  \right)  \| f^{k} -y \|_2^2\\
& \leq (1-\alpha) \| f^{k} -y \|_2^2:= \zeta \|f^{k}-y\|_2^2,
\end{aligned}
\ee
%Summing the inequality \eqref{theorem-ntk-1}, we know that with probability $(1-\delta)\Pi_{i=0}^\infty (1-\delta_i)$ where $(1-\delta)$ on the initialization and $\Pi_{i=0}^k (1-\delta_i)$ on the sketching, we have
%
%\[ \|f^{k} - y\|_2^2 \leq \zeta^k \|f^{0} - y\|_2^2
%\]
and this completes the proof.
\end{proof}

\begin{lemma}
\label{lemma-iteration}
If Conditions 1 and 2 hold and $\lambda_{\text{min}}(G_k) \geq \frac{4}{9}{\lambda_{\text{min}}(G_0)}$, we have
\[ \|\theta_{k+1} - \theta_0\|_2 \leq \frac{3\|y-f^0\|_2}{\sqrt{\lambda_{\text{min}}(G_0)}}.\]
\end{lemma}
%\fi

\begin{proof}
The distance of the parameters to the initialization can be bounded by:
\be
\label{sigma-estimate-lemma-2}
\begin{aligned}
  \|\theta_{k+1} - \theta_0 \|_2 & \leq \alpha \sum_{i=0}^k \left( \|J_i^\top (J_i\Omega_i^\top \Omega_iJ_i^\top)^{-1}(f^i-y)  \|_2 + \|d_i - J_i^\top (J_i\Omega_i^\top \Omega_iJ_i^\top)^{-1}(f^i-y)\|_2\right ) \\
    & \leq  \alpha\sum_{i=0}^k \frac{\left(1+\eta_k+\mathcal{D}(\lambda_k) \right)\|f^i - y \|_2}{\sqrt{\lambda_\mathrm{min}(G_i)}} \\
    & \leq \alpha\sum_{i=0}^k \frac{(1+\eta_k + \mathcal{D}(\lambda_k)) (1-\alpha)^{i/2}\|f^0 - y \|_2}{\sqrt{ \frac{4}{9}{\lambda_{\text{min}}(G_0)}}}\\
    &\leq \frac{3 \|y-f^0\|_2}{\sqrt{\lambda_{\text{min}}(G_0)}}.
\end{aligned}
\ee
\end{proof}
Theorem \ref{ntk-theorem} can be proved by using Lemmas \ref{psd-jacobian}, \ref{main-lemma-ntk} and \ref{lemma-iteration}.
\begin{proof}
The analysis is on the event $\widetilde \Omega$ where Conditions 1 and 2 hold. We prove Theorem \ref{ntk-theorem} by contradiction. Suppose that \be \label{bound-param}  \|\theta_k - \theta_0\|_2 \leq \frac{3\|y-f^0\|_2}{\sqrt{\lambda_{\text{min}}(G_0)}} \ee  does not hold for all iterations. Let (\ref{bound-param}) hold at the iterations $k = 0,1,\dots, \hat k$ but not $\hat k +1$. Then from Lemma \ref{main-lemma-ntk} and \ref{lemma-iteration} we know that there exists $0 < \tilde k \leq \hat k$ such that $\lambda_{\text{min}}(G_{\tilde k}) < \frac{4}{9}{\lambda_{\text{min}}(G_0)}$. However, by Lemma \ref{psd-jacobian}, since $\|\theta_k - \theta_0\|_2 \leq \frac{3\|y-f^0\|_2}{\sqrt{\lambda_{\text{min}}(G_0)}}$ holds for $k = 0,1,\dots, \hat k$, we have $\lambda_{\text{min}}(G_{\tilde k}) < \frac{4}{9}{\lambda_{\text{min}}(G_0)}$, which is a contradiction. Therefore, (\ref{bound-param}) hold for all iterations. By Lemma \ref{main-lemma-ntk}, this illustrates that $ \|f^{k+1} - y\|_2^2 \leq \zeta \|f^k - y\|_2^2$ for all iterations. 

Hence, on the event $\widetilde \Omega$ such that $\Prob(\widetilde \Omega) = (1-\delta)\Pi_{k=0}^\infty (1-\delta_k)$, the following conclusion holds:
 \be \|f^{k} - y\|_2^2 \leq \zeta^k \|f^0 - y\|_2^2\ee
 and this completes the proof.
\end{proof}


%%%%%%%%%%%%%%%%%%%%%%%%%%%%%%%%%%%%%%%%%%%%%%%%%%%%%%%%%%%%%%%%%%%%%%%%%%%%%%%
%%%%%%%%%%%%%%%%%%%%%%%%%%%%%%%%%%%%%%%%%%%%%%%%%%%%%%%%%%%%%%%%%%%%%%%%%%%%%%%

\end{document}